\documentclass[12pt,final,3p,times]{elsarticle}

\usepackage{amsthm,amsmath,amssymb,amsfonts,mathrsfs,graphicx,graphics,latexsym,exscale,cmmib57,dsfont,bbm,amscd}
\usepackage{mathrsfs}
\usepackage{tikz-cd}

\swapnumbers
\newtheorem{thm}{Theorem}[section]
\newtheorem{sthm}{Theorem}[subsection]
\newtheorem{lem}[thm]{Lemma}
\newtheorem*{2.1C}{\ref{2.1}C~Theorem}
\newtheorem*{2.5A}{\ref{2.5}A~Corollary}
\newtheorem*{2.5B}{\ref{2.5}B~Corollary}
\newtheorem*{3.1D}{\ref{3.1}D~Lemma}
\newtheorem*{4.2.1A}{\ref{4.2.1}A~Lemma}
\newtheorem*{6.1B}{\ref{6.1}B~Zero-One Law}
\newtheorem*{6.1D}{\ref{6.1}D~Ergodicity of $G$-shift}
\newtheorem*{6.1F}{\ref{6.1}F~Hewitt-Savage Zero-One Law}
\newtheorem*{6.1H}{\ref{6.1}H~Strong Law of Large Numbers}

\newtheorem*{4.1.5'}{\ref{4.1.5}$^\prime$~Theorem}
\newtheorem*{4.1.8'}{\ref{4.1.8}$^\prime$~Theorem}
\newtheorem*{4.2.1B}{\ref{4.2.1}B~Theorem}
\newtheorem*{7.2}{\ref{7.2}~Lemma}
\newtheorem*{A.1'}{A.1$^\prime$~Lemma}
\newtheorem*{A.1''}{A.1$^{\prime\prime}$~Corollary}
\newtheorem*{A.1}{A.1~Lemma}
\newtheorem*{A.2}{A.2~Theorem}
\newtheorem*{A.3}{A.3~Theorem}
\newtheorem*{A.5}{A.5~Lemma}
\newtheorem*{A.6}{A.6~Lemma}
\newtheorem*{A.7}{A.7~Lemma}
\newtheorem*{A.8}{A.8~Lemma}

\newtheorem{slem}[sthm]{Lemma}
\newtheorem{cor}[thm]{Corollary}
\newtheorem{scor}[sthm]{Corollary}

\theoremstyle{definition}

\newtheorem{ex}[thm]{Example}
\newtheorem*{6.1A}{\ref{6.1}A}
\newtheorem*{6.1C}{\ref{6.1}C}
\newtheorem*{6.1E}{\ref{6.1}E}
\newtheorem*{6.1G}{\ref{6.1}G}
\newtheorem*{3.1E}{\ref{3.1}E~Example}
\newtheorem*{3.1F}{\ref{3.1}F~Example}

\newtheorem{se}[thm]{}
\newtheorem{sse}[sthm]{}
\newtheorem*{A.4}{A.4}
\newtheorem{rem}[thm]{Remark}
\newtheorem{srem}[sthm]{Remark}
\newtheorem{defn}[thm]{Definitions}
\newtheorem{MT}[thm]{Main theorems}
\newtheorem{OL}[thm]{Outlines}
\newtheorem{SN}[thm]{Standing symbols}
\newtheorem{que}[thm]{Question}

\newtheorem*{note*}{Note}

\journal{Bulletin des Sciences Math\'{e}matiques}

\begin{document}

\begin{frontmatter}
\title{On generalized Namioka spaces and joint continuity of functions on product of spaces}

\author{Xiongping Dai}
\ead{xpdai@nju.edu.cn}
\author{Congying Lv}
\author{Yuxuan Xie}
\address{School of Mathematics, Nanjing University, Nanjing 210093, People's Republic of China}

\begin{abstract}
A space $X$ is called a \textit{generalized Namioka space} (g$\mathcal{N}$-space), if for every compact space $Y$ and every separately continuous function $f\colon X\times Y\rightarrow\mathbb{R}$, there exists at least one point $x\in X$ such that $f$ is jointly continuous at each point of $\{x\}\times Y$. 
We principally prove the following results:
\begin{enumerate}[1.]
\item If $X=\prod_{\alpha\in A}X_\alpha$ is non-meager such that each factor is a separable space or each factor is a pseudo-metric space, then $X$ is a g$\mathcal{N}$-space.

\item If $X$ is a separable space and $Y$ a pseudo-metric space such that $X\times Y$ is Baire (resp. non-meager), then $X\times Y$ is an $\mathcal{N}$-space (resp.~a g$\mathcal{N}$-space).

\item If $X=\prod_{\alpha\in A}X_\alpha$ such that each factor is separable and $\prod_{\alpha\in A^\prime}X_\alpha$ is a non-meager space for each countable subset $A^\prime$ of $A$, then $X$ is a non-meager g$\mathcal{N}$-space.

\item If $X=\prod_{\alpha\in A}X_\alpha$ such that each factor has a countable $\pi$-base, then each tail set having the property of Baire in $X$ is either meager or residual.
\end{enumerate}
If $G$ is a g$\mathcal{N}$ right-topological group and $X$ a locally compact regular space, or, if $G$ is a separable first countable non-meager right-topological group and $X\times X$ a countably compact completely regular space, then any separately continuous action $G\curvearrowright X$ is jointly continuous.
\end{abstract}

\end{frontmatter}
\section{Introduction}\label{s1}
In (1899) \cite{B99}, Ren\'{e}-Louis Baire (1874-1932) proved that if $f\colon [0,1]\times[0,1]\rightarrow\mathbb{R}$ is a separately continuous function, then there exists a dense subset $J$ of $[0,1]$ such that $f$ is jointly continuous at each point of $J\times[0,1]$. After the seminal work of Isaac~Namioka (1974) \cite{N74} on Baire's problem of joint continuity, a space $X$ is called a \textit{Namioka space} (an $\mathcal{N}$-space), if for every compact space $Y$ and every separately continuous function $f\colon X\times Y\rightarrow\mathbb{R}$, there exists a dense set $R\subseteq X$ such that $f$ is jointly continuous at each point of $R\times Y$ (cf.~\cite{Tr79, T79, C81, SR83, L84, T85, D86, D87} and so on). In that case, $\langle X,Y\rangle$ is sometimes called a \textit{Namioka pair} (cf., e.g., \cite{K95, BP05}). Equivalently, $X$ is an $\mathcal{N}$-space if for every compact space $Y$ and every continuous function $f\colon X\rightarrow \mathcal{C}_p(Y)$, there exists a dense set $J\subseteq X$ such that $f\colon X\rightarrow C_u(Y)$ is continuous at each point of $J$. This implies that if $X$ has the local $\mathcal{N}$-property (i.e., each point of $X$ has a neighborhood which is an $\mathcal{N}$-space), then $X$ is an
$\mathcal{N}$-space itself. Here and in the sequel, $\mathcal{C}(Y)$ is the set of continuous real-valued functions on $Y$; $\mathcal{C}_p(Y)$ and $\mathcal{C}_u(Y)$ are the spaces of $\mathcal{C}(Y)$ equipped with the topologies of pointwise convergence and uniform convergence, respectively.

Following Burke-Pol (2005) \cite{BP05}, in the realm of completely regular $T_1$-spaces (i.e., Tychonoff spaces \cite[p.~117]{K55}), $\langle X,K\rangle$ is called a \textit{weak-Namioka pair}, if $K$ is compact and for any separately continuous function $f\colon X\times K\rightarrow\mathbb{R}$ and a closed subset $F$ of $X\times K$ projecting irreducibly onto $X$, the set of points of continuity of $f\!\upharpoonright_F\colon F\rightarrow\mathbb{R}$ is dense in $F$.
In Piotrowski-Waller (2012) \cite{PW12} the so-called weakly Namioka space was studied by only requiring $Y$ to be second countable Hausdorff instead of $Y$ being compact. That is, $X$ is called \textit{weakly Namioka} if for every second countable Hausdorff space $Y$ and every separately continuous function $f\colon X\times Y\rightarrow\mathbb{R}$, there exists a dense $G_\delta$-set $R\subseteq X$ such that $f$ is jointly continuous at each point of $R\times Y$.

In the present paper we shall give another generalization of the $\mathcal{N}$-property (Def.~\ref{1.1}A) and consider several classes of spaces with this generalized $\mathcal{N}$-property.

\begin{defn}\label{1.1}
Let $X$ be any space and $A\subseteq X$. Recall that $A$ is \textit{meager} or $X$ is of first category in $X$, if $A=\bigcup_{i=1}^\infty F_i$ where $\textrm{int}\,\overline{F}_i$, the interior of the closure of $F_i$, is void for all $i=1,2,\dotsc$; $A$ is \textit{non-meager} or $A$ is of second category in $X$, if it is not meager in $X$. The complement of a meager set is called \textit{residual} in $X$. $X$ is called a \textit{Baire space}, iff every non-void open subset of $X$ is non-meager in $X$, iff every residual subset of $X$ is dense in $X$. In addition, we say that $X$ is \textit{non-meager}, or $X$ is of second category, if it is a non-meager subset of itself. See \cite{K55, W70, O80, E89}. There is a well-known basic fact: If $\emptyset\not=A\subseteq U\varsubsetneq X$ where $U$ is open in $X$, then $A$ is non-meager in $U$ if and only if $A$ is non-meager in $X$.

\begin{enumerate}[\textbf{A.}]
\item[\textbf{A.}] $X$ is called a \textit{generalized Namioka space} (g$\mathcal{N}$-space), if for every compact space $Y$ and every separately continuous function $f\colon X\times Y\rightarrow\mathbb{R}$, there exists at least one point $x\in X$ such that $f$ is jointly continuous at each point of $\{x\}\times Y$.
In other words, a space $X$ is a g$\mathcal{N}$-space iff for every compact space $Y$ and every continuous function $f\colon X\rightarrow \mathcal{C}_p(Y)$, there exists at least one point $x\in X$ at which $f$ is $\|\cdot\|$-continuous. In particular, in the realm of completely regular spaces a homogeneous g$\mathcal{N}$-space is a Baire space (by Thm.~\ref{7.3} and Rem.~\ref{2.9}).

\item[\textbf{B.}] Let $G$ be a group with a topology. By $G\curvearrowright_\pi\!X$, it means a left-action of $G$ on $X$ with phase mapping $\pi\colon G\times X\rightarrow X$, $(t,x)\mapsto tx$ (i.e., $ex=x$ and $(st)x=s(tx)$ $\forall x\in X$, $s,t\in G$, where $e$ is the unit element of $G$). If $\pi$ is separately continuous, then $G\curvearrowright_\pi\!X$ is said to be separately continuous; if $\pi$ is jointly continuous, then $G\curvearrowright_\pi\!X$ is referred to as a \textit{topological flow}.
\end{enumerate}
\end{defn}

The $\mathcal{N}$-property and the g$\mathcal{N}$-property are both topologically invariant; but g$\mathcal{N}$-space is conceptually weaker than $\mathcal{N}$-space. For example, if a space contains an open set which is a g$\mathcal{N}$-space, then it is a g$\mathcal{N}$-space itself; but a space with an open $\mathcal{N}$-subspace need not be an $\mathcal{N}$-space itself.
In fact, if a completely regular g$\mathcal{N}$-space is not a Baire space, then it is not an $\mathcal{N}$-space (see Ex.~\ref{1.4}).
However, this concept is still useful for the mathematics modeling of topological dynamics as shown by the following observation:

\begin{thm}\label{1.2}
Let $G$ be a g$\mathcal{N}$ right-topological group and $X$ a locally compact regular space. If $G\!\curvearrowright_\pi\!X$ is separately continuous, then $G\!\curvearrowright_\pi\!X$ is a topological flow.
\end{thm}

\begin{proof}
By considering the one-point compactification of $X$ in place of $X$ if necessary, we may assume $X$ is a compact regular space without loss of generality. Let $(t_i,x_i)\to (t,x)$ in $G\times X$. If $t_ix_i\not\to tx$ in $X$, then we may assume that $tx\notin\Lambda:=\overline{\{t_ix_i\,|\,i\ge i_0\}}$ for some $i_0$. Letting $\psi\in \mathcal{C}(X,[0,1])$ with $\psi\!\upharpoonright_\Lambda\equiv0$ and $\psi(tx)=1$, there exists an element $g\in G$ such that $f=\psi\circ\pi\colon G\times X\rightarrow[0,1]$ is jointly continuous at each point of $\{g\}\times X$. Then by $t_it^{-1}g\to g$ and $g^{-1}tx_i\to g^{-1}tx$, it follows that
$0=\psi(t_ix_i)=\psi\circ\pi(t_it^{-1}g,g^{-1}tx_i)\to\psi\circ\pi(g,g^{-1}tx)=\psi(tx)=1$,
which is impossible. 
\end{proof}

Theorem~\ref{1.2} is already a generalization of the classic joint continuity theorem of R. Ellis (1957) \cite[Thm.~1]{E57}, since any locally compact Hausdorff semitopological group is an $\mathcal{N}$-space (cf.~\cite{N74} or Lem.~\ref{3.2}) and $G$ and $X$ in \cite[Thm.~1]{E57} are both presupposed to be locally compact Hausdorff spaces. 
See Theorem~\ref{4.1.11} in $\S$\ref{s4.1} for another variation of Ellis' joint continuity theorem by considering only countably compact phase spaces.

\begin{MT}\label{1.3}
We note that the class of $\tau$-well $\alpha$-favorable spaces of Christensen is closed under arbitrary products (cf.~Christensen 1981 \cite{C81}).
In this paper we shall mainly prove the following sufficient conditions for the $\mathcal{N}$-property or g$\mathcal{N}$-property of products:
\begin{enumerate}[\textbf{(1)}]
\item[\textbf{(1)}] {\it If $X$ is a non-meager (resp.~Baire) open subspace of the product of a family of separable spaces, then $X$ is a g$\mathcal{N}$-space (resp. an $\mathcal{N}$-space).}

\item[\textbf{(2)}] {\it Let $X$ be an open subspace of the product of a family of pseudo-metric spaces; then:
\begin{enumerate}[a.]
    \item $X$ is Baire iff $X$ is $\tau$-well $\beta$-d\'{e}favorable in the sense of Christensen (Thm.~\ref{3.21}).
    \item $X$ is non-meager iff $X$ is a g$\mathcal{N}$-space.    
\end{enumerate}}

\item[\textbf{(3)}] {\it If $X$ is a separable space and $Y$ a pseudo-metric space such that $X\times Y$ is Baire (resp. non-meager), then $X\times Y$ is an $\mathcal{N}$-space (resp.~a g$\mathcal{N}$-space) (cf.~Calbrix-Troallic 1979~\cite{CT79} or Saint-Raymond 1983 \cite[Thm.~6]{SR83} for $Y=\{y\}$ and \cite[Thm.~7]{SR83} for $X=\{x\}$).}

\item[\textbf{(4)}] {\it If $X$ is a space which has countable tightness and a rich family of Baire subspaces, then $X$ is an $\mathcal{N}$-space (cf.~Lin-Moors (2008) \cite{LM08} in the realm of Hausdorff spaces).}
\end{enumerate}

There exists a completely regular Baire space whose product with itself is meager (cf.~Oxtoby \cite[Thm.~5]{O60} or Cohen \cite{C76}). Thus, there exists a completely regular non-meager space whose product with itself is meager. However, we shall prove the following category theorems:
\begin{enumerate}[(1)]
\item[\textbf{(5)}] {\it Let $X=\prod_{\alpha\in A}X_\alpha$, where each factor is a separable space. If $\prod_{\alpha\in A^\prime}X_\alpha$ is a Baire (resp. non-meager) space for each countable set $A^\prime\subset A$, then $X$ is a Baire $\mathcal{N}$-space (resp. non-meager g$\mathcal{N}$-space).}

\item[\textbf{(6)}] {\it If $X=\prod_{i\in I}X_i$ such that each factor has a countable $\pi$-base, then every tail set having the property of Baire in $X$ is either meager or residual (cf.~Oxtoby 1960 \cite[Thm.~4]{O60} for each factor is a Baire space).}
\end{enumerate}

It is a well-known fact that in the realm of completely regular spaces, an $\mathcal{N}$-space must be a Baire space (cf.~Saint-Raymond 1983 \cite[Thm.~3]{SR83}). In fact, this can be generalized as follows:
\begin{enumerate}[(1)]
\item[\textbf{(7)}] {\it Any completely regular g$\mathcal{N}$-space is non-meager in itself; and if $X\times Y$ is completely regular such that each factor is either separable or pseudo-metrizable, then $X\times Y$ is a Baire (resp.~non-meager) space iff it is an $\mathcal{N}$-space (resp.~a g$\mathcal{N}$-space).}
\end{enumerate}

Thus, by Theorems~\ref{1.3}-(1)/(2) and (7), in the realm of product spaces of pseudo-metric spaces or of completely regular separable spaces, the class of the g$\mathcal{N}$-spaces coincides with the class of non-meager spaces. However, a non-meager space is not necessarily a g$\mathcal{N}$-space (Ex.~\ref{7.5}).
\end{MT}

Now we shall introduce a simple counterexample that says there exists a g$\mathcal{N}$-space which is not an $\mathcal{N}$-space.

\begin{ex}[{cf.~\cite[p.~181]{W70}}]\label{1.4}
Let $X=\mathbb{Q}\cup[0,1]$ with the Euclidean topology. Then $X$ is a non-homogeneous, separable, non-meager metric space. Thus, by Theorem~\ref{1.3}-(1) or (2b), $X$ is a g$\mathcal{N}$-space. However, $X$ is not an $\mathcal{N}$-space. For otherwise, $X$ should be a Baire space (by \cite[Thm.~3]{SR83} or Thm.~\ref{7.1}) but $X$ is not Baire, for the open set $\mathbb{Q}\setminus[0,1]$ is meager in $X$.
Nevertheless the countable product $X^\mathbb{N}$ is a meager separable metric space. Indeed, both of $F:=[0,1]^\mathbb{N}$ and $F_{m,q}:=\{(x_n)\in X^\mathbb{N}\,|\,x_m=q\}$, for all $(m,q)\in\mathbb{N}\times \mathbb{Q}\setminus[0,1]$, are closed subsets of $X^\mathbb{N}$ with void interiors. So $X^\mathbb{N}=F\cup\bigcup\{F_{m,q}\colon m\in\mathbb{N}\ \&\ q\in\mathbb{Q}\setminus[0,1]\}$ is meager. In fact, $X^\mathbb{N}$ is not a g$\mathcal{N}$-space by Theorem~\ref{7.3}.
\end{ex}

\begin{OL}
This self-contained paper is simply organized as follows: In $\S$\ref{s2} we shall prove Theorem~\ref{1.3}-(1) using the Banach-Mazur topological game (Thm.~\ref{2.5}). In $\S$\ref{s3} we shall prove Theorems~\ref{1.3}-(2a) and (3) and the necessity part of Theorem~\ref{1.3}-(2b) using the Christensen/Saint-Raymond topological game (Thm.~\ref{3.21}, Thm.~\ref{3.11} and Thm.~\ref{3.4}). The sufficiency part of Theorem~\ref{1.3}-(2b) will be proved in $\S$\ref{s7} (Thm.~\ref{7.8}). In $\S$\ref{s4} we shall prove Theorem~\ref{1.3}-(4) by improving the approaches in \cite{LM08} (Thm.~\ref{4.1.8}$^\prime$); and we will extend a theorem of Hurewicz \cite{H28} (Thm.~\ref{4.2.5} and Thm.~{5.2.9}). In addition, Theorem~\ref{1.3}-(5) will be proved in $\S$\ref{s5.4} based on Oxtoby's theorems (Thm.~\ref{5.4.4} and Thm.~\ref{5.4.5}). Theorem~\ref{1.3}-(6)---a category analogue of Kolmogoroff's zero-one law will be proved in $\S$\ref{s6} (Thm.~\ref{6.2.4} and Thm.~\ref{6.2.6}). Finally, Theorem~\ref{1.3}-(7) will be proved in $\S$\ref{s7} using \cite[Lem.~4]{SR83} (Lem.~\ref{7.2}, Thm.~\ref{7.3} and Thm.~\ref{7.9}). In \ref{A}, we will present the proofs of two topological Fubini theorems (Lem.~A.1 and Lem.~A.8). In particular, Lemma~\ref{6.2.5}, as a result of Lemma~A.8, is a variant of the classical Kuratowski-Ulam-Sikorski theorem (Thm.~A.3).
\end{OL}

\begin{SN}
Let $\mathbb{N}=\{1,2,3,\dotsc\}$ be the set of positive integers. If $X$, $Y$ and $Z$ are topological spaces, then:
\begin{enumerate}[1.]
\item[\textbf{1.}] $\mathfrak{N}_x(X)$ and $\mathfrak{N}_x^\textrm{o}(X)$ stand for the filters of neighborhoods and open neighborhoods of $x$ in $X$, respectively.
$\mathscr{O}(X)$\index{$\mathscr{O}(X)$} stands the family of all open non-void subsets of $X$.

\item[\textbf{2.}] For any function $f\colon X\times Y\rightarrow Z$ and all point $(x,y)\in X\times Y$, let
$f_x\colon Y\rightarrow Z,\ {y\mapsto f(x,y)}$ and $f^y\colon X\rightarrow Z,\ {x\mapsto f(x,y)}$.

\item[\textbf{3.}] Given any $K\subseteq X\times Y$, we write
$K_x=\{y\in Y\,|\,(x,y)\in K\}$ and $K^y=\{x\in X\,|\,(x,y)\in K\}$ for all $x\in X$ and $y\in Y$.
\end{enumerate}
\textbf{Warning.} No ``separability'' conditions are presupposed for any topological space in our later arguments. In particular following \cite{K55}, a ``compact space'' is a topological space (not necessarily Hausdorff) in which every open cover admits a finite subcover.
\end{SN}

\section{BM-game, $\Pi$-separable spaces and g$\mathcal{N}$-spaces}\label{s2}
This section will be devoted to proving Theorem~\ref{1.3}-(1) stated in $\S$\ref{s1} under the guise of Theorem~\ref{2.5} (Cor.~\ref{2.5}A). First of all, we recall the concept---BM-game needed in our later discussion.

\begin{se}[Banach-Mazur game~\cite{O57,C69,O80, K95} and $\Pi$-separable spaces]\label{2.1}
Let $X$ be any topological space. We will need the following basic notions:

\item \textbf{A.} By a \textit{BM($X$)-play}, we mean a sequence $\{(B_i,A_i)\}_{i=1}^\infty$ of pairs of elements of $\mathscr{O}(X)$ such that $B_i\supseteq A_i\supseteq B_{i+1}$ for all $i\in\mathbb{N}$, where $B_i$ and $A_i$ are picked alternately by Player $\beta$ and Player $\alpha$, respectively; and moreover, Player $\beta$ is always granted the privilege of the first move. Our winning condition (W.C.) is defined as follows:

\begin{enumerate}[\textbf{W.C.}:]
    \item Player $\alpha$ \textit{wins} the BM($X$)-play $\{(B_i,A_i)\}_{i=1}^\infty$, if $\bigcap_{i\in\mathbb{N}}A_i\not=\emptyset$; otherwise, Player $\beta$ \textit{wins} this play. 
\end{enumerate}
Note here that BM($X$)-game is sometimes called Choquet game and denoted by $\mathcal{J}(X)$; see, e.g., \cite{C81, SR83, D86, K95}.
\begin{enumerate}[(i)]
    \item If Player $\alpha$ has a winning strategy in the BM($X$)-game, then $X$ is called \textit{$\alpha$-favorable of BM} or a \textit{Choquet space} \cite{K95}. 
    
    \item If Player $\beta$ has no winning strategy in the BM($X$)-game, then $X$ is said to be \textit{$\beta$-d\'{e}favorable of BM}. In that case, if $\tau$ is a strategy for Player $\beta$, then there always exists a $\tau$-play $\{(B_i,A_i)\}_{i=1}^\infty$ of BM($X$) such that $\bigcap_{i=1}^\infty A_i\not=\emptyset$.
\end{enumerate}

\item\textbf{B.} $X$ is called a \textit{$\Pi$-separable space}, if there exists a family $\{X_i\colon i\in I\}$ of separable spaces such that $X$ is homeomorphic to $\prod_{i\in I}X_i$. In that case, we shall identify $X$ with $\prod_{i\in I}X_i$ if no confusion. Clearly, a separable space is $\Pi$-separable; but not vice versa.
\end{se}

The class of Choquet spaces is closed under arbitrary products \cite[Thm.~7.12]{C69}. Although this is not the case for the class of $\beta$-d\'{e}favorable spaces of BM, we can have the following:

\begin{2.1C}[{cf.~\cite[Thm.~9]{SR83}}]
If $X$ is a Choquet space and $Y$ is a $\beta$-d\'{e}favorable space of BM, then $X\times Y$ is $\beta$-d\'{e}favorable of BM.
\end{2.1C}

\begin{proof}[An alternative proof]
First of all, for all $U\in\mathscr{O}(X\times Y)$, we can define a non-void collection of non-void open subsets of $Y$: 
\begin{enumerate}
    \item[] $\mathcal{Y}[U]=\{W\in\mathscr{O}(Y)\colon \exists W^\prime\in\mathscr{O}(X) \textrm{ s.t. }W^\prime\times W\subseteq U\}$. 
\end{enumerate}
Further, for all $W\in\mathcal{Y}[U]$, we can well define a set
\begin{enumerate}
    \item[] $W^*=\bigcup\{W^\prime\in\mathscr{O}(X)\,|\,W^\prime\times W\subseteq U\}\in\mathscr{O}(X)$.
\end{enumerate}
Then $W^*\times W\subseteq U$ for all $U\in\mathscr{O}(X\times Y)$ and all $W\in\mathcal{Y}[U]$.

Let $t$ be any strategy for Player $\beta$ in the BM($X\times Y$)-game.
Since $X$ is Choquet space, there is a winning strategy $s$ for Player $\alpha$ in the BM($X$)-game. We need only prove that $t$ is not a winning strategy for Player $\beta$ in the BM($X\times Y$)-game. For that, we will introduce an auxiliary strategy $\theta$ for Player $\beta$ in the BM($Y$)-game, based on $t$ and $s$, as follows:
Let $U_1=t(X\times Y)$ and set $\theta(Y)=W_1\in\mathcal{Y}[U_1]$; then, for all $V_{2,1}\in\mathscr{O}(W_1)$, as the possible choice of Player $\alpha$ at the 1st stroke in the BM($Y$)-game, write $V_1=s(W_1^*)\times V_{2,1}\in\mathscr{O}(U_1)$ as the possible answer of Player $\alpha$ to Player $\beta$'s 1st move $U_1$ in the BM($X\times Y$)-game. Let $U_2=t(X\times Y, V_1)$ and set $\theta(Y, V_{2,1})=W_2\in\mathcal{Y}[U_2]$; then, for all $V_{2,2}\in\mathscr{O}(W_2)$, as the possible choice of Player $\alpha$ at the 2nd stroke in the BM($Y$)-game, write $V_2=s(W_1^*,W_2^*)\times V_{2,2}\in\mathscr{O}(U_2)$. Let $U_3=t(X\times Y, V_1, V_2)$ and set $\theta(Y, V_{2,1}, V_{2,2})=W_3\in\mathcal{Y}[U_3]$; then, for all $V_{2,3}\in\mathscr{O}(W_3)$, as the possible choice of Player $\alpha$ at the 3rd stroke in the BM($Y$)-game, write $V_3=s(W_1^*,W_2^*,W_3^*)\times V_{2,3}\in\mathscr{O}(U_3)$. Continue this procedure indefinitely, we can introduce a strategy $\theta$ for Player $\beta$ in the BM($Y$)-game based on strategies $t$ and $s$.
Since $Y$ is $\beta$-d\'{e}favorable of BM, $\theta$ is not a winning strategy for Player $\beta$ in the BM($Y$)-game. Thus, there exists a sequence $\{V_{2,n}\}_{n=1}^\infty$ played by Player $\alpha$ such that Player $\alpha$ wins the $\theta$-play $\{(W_n,V_{2,n})\}_{n=1}^\infty$. This implies that Player $\alpha$ also wins the $t$-play $\{(U_n,V_n)\}_{n=1}^\infty$ in the BM($X\times Y$)-game, where $V_n=s(W_1^*,\dotsc,W_n^*)\times V_{2,n}$. Therefore, $X\times Y$ is $\beta$-d\'{e}favorable of BM.
\end{proof}

One of the points of the BM-game is the so-called Oxtoby-Christensen-Saint-Raymond Category Theorem, which characterizes Baire space using the BM-game played on it as follows:

\begin{thm}[cf.~\cite{O57, C81, SR83}]\label{2.2}
A space $X$ is Baire if and only if there exists no winning strategy for Player $\beta$ in the BM($X$)-game (i.e., $X$ is Baire if and only if $X$ is $\beta$-d\'{e}favorable of BM).
\end{thm}

Recall that a space $X$ is non-meager if and only if every residual set in $X$ is non-void. It should be mentioned that $U\in\mathscr{O}(X)$ is non-meager in $X$ if and only if $U$, as a subspace of $X$, is non-meager. Thus,
if $X$ contains a non-meager $U\in\mathscr{O}(X)$, then $X$ is non-meager itself.
Indeed, if $\{G_n\}_{n=1}^\infty$ is any sequence of dense open subsets of $X$, then $\{G_n\cap U\}_{n=1}^\infty$ is a sequence of dense open subsets of $U$ so that
$\emptyset\not=\bigcap_{n=1}^\infty(G_n\cap U)\subseteq\bigcap_{n=1}^\infty G_n$. In addition, if a closed set $A\subseteq X$ is non-meager, then $A$ is non-meager itself ($\because A=(A\setminus\textrm{int}_XA)\cup\textrm{int}_XA$ and $\textrm{int}_XA\not=\emptyset$ is non-meager in $X$); of course, not vice versa.
Then the above classic category theorem (Thm.~\ref{2.2}) may be slightly improved to the following local version, which in turn implies Theorem~\ref{2.2}.

\begin{thm}\label{2.3}
Let $X$ be a topological space. Then $U\in\mathscr{O}(X)$ is non-meager in $X$ if and only if there exists no winning strategy $\tau$ with $\tau(X)=U$ for Player $\beta$ in the BM($X$)-game.
\end{thm}

\begin{proof}
\textsl{Necessity}: Suppose to the contrary that there is a winning strategy $\tau$ with $\tau(X)=U$ for Player $\beta$ in the BM($X$)-game. To get a contradiction, let $I_1=\{\emptyset\}$, $U_{1,\emptyset}=U$ played firstly by Player $\beta$, and $V_{0,\emptyset}=X$. Now using transfinite induction, we can construct a maximal family $\{(V_{n-1,i},U_{n,i})\}_{i\in I_n}$ of open subsets of $X$, for each integer $n\ge2$, such that:
    \begin{enumerate}
    \item $U_{n,i}\cap U_{n,j}=\emptyset\ \forall i\not=j\in I_n$;
    \item $\forall i\in I_n$ $\exists j=j(i)\in I_{n-1}$ such that $V_{n-1,i}\subseteq U_{n-1,j}$;
    \item If $(i_2,\dotsc,i_n)\in I_2\times\dotsm\times I_n$ and $U_{1,\emptyset}\supseteq U_{2,i_2}\supseteq\dotsm\supseteq U_{n,i_n}$, then $U_{n,i_n}=\tau(X,V_{1,i_2},\dotsc,V_{n-1,i_n})$.
    \end{enumerate}
Let $\Omega_n=\bigcup_{i\in I_n}U_{n,i}$ for all $n\in\mathbb{N}$. Then each $\Omega_n$, $n\ge2$, is open dense in $\Omega_1$ by the maximality. Note that for all $i_n\in I_n$ and all $i_{n+1}\in I_{n+1}$, either $U_{n,i_n}\supseteq U_{n+1,i_{n+1}}$ or $U_{n,i_n}\cap U_{n+1,i_{n+1}}=\emptyset$, for 1. and $\bigcup\{U_{n+1,i_{n+1}}\,|\,i_{n+1}\in I_{n+1}\ \&\ U_{n+1,i_{n+1}}\subseteq U_{n,i_n}\}$ is dense in $U_{n,i_n}$. However, $\tau$ is a winning strategy for Player $\beta$, so $\bigcap_{n=1}^\infty U_{n,i_n}=\emptyset$ and $\bigcap_{n\ge2}\Omega_n=\emptyset$, and $\Omega_1$ is not non-meager, which is a contradiction.

\textsl{Sufficiency}: To prove $U$ is non-meager, suppose to the contrary that $U$ is meager in $X$. Then there exists a sequence $\{G_n\}_{n=1}^\infty$ of open dense subsets of $U$ such that $\bigcap_{n=1}^\infty G_n=\emptyset$. We may assume $G_1=U$ without loss of generality. Now we could define inductively a winning strategy $\tau$ with $\tau(X)=U$ for Player $\beta$ in the BM($X$)-game as follows:
Let $\tau(X):=G_1$; then for every $V_1\in\mathscr{O}(G_1)$ as the possible first move of Player $\alpha$, let $\tau(X,V_1)=U_2:=V_1\cap G_2$. If Player $\beta$ has played $(U_1,\dotsc,U_n)$ and Player $\alpha$ has played $(V_1,\dotsc,V_n)$, then at the $(n+1)$th-stroke, Player $\beta$ plays
$\tau(X, V_1,\dotsc,V_n)=U_{n+1}:=V_n\cap G_{n+1}$
and Player $\alpha$ plays an arbitrary set $V_{n+1}\in\mathscr{O}(U_{n+1})$. Thus, by induction, we can define a BM-play $\{(U_i,V_i)\}_{i=1}^\infty$ with $U_{i+1}=\tau(X, V_1,\dotsc,V_i)$ such that $\bigcap_{i=1}^\infty U_i\subseteq\bigcap_{i=1}^\infty G_i=\emptyset$. This shows that $\tau$ with $\tau(X)=U$ is a winning strategy for Player $\beta$ in the BM($X$)-game, contrary to the sufficiency condition. The proof is complete.
\end{proof}

\begin{se}[Countable compactness]\label{2.4}
Let $X$ be a topological space, $A\subseteq X$ and $x\in X$. Recall that the point $x$ is an \textit{accumulation}/\textit{cluster}/\textit{limit} point of $A$ iff $U\cap(A\setminus\{x\})\not=\emptyset$ $\forall U\in\mathfrak{N}_x(X)$. The point $x$ is an \textit{$\omega$-accumulation} point of $A$ of iff $U$ contains infinitely many points of $A$ for all $U\in\mathfrak{N}_x(X)$.
The point $x$ is a \textit{cluster point of a net} $\{x_n\colon n\in D\}$ in $X$ iff $\{x_n\colon n\in D\}$ is frequently in every $U\in\mathfrak{N}_x(X)$; i.e., $\forall m\in D$, $\exists n\ge m$ s.t. $x_n\in U$. If $x$ is a cluster point of a sequence $\{x_n\}_{n=1}^\infty$ in $X$, then there is a subnet $\{x_{n(\alpha)}\colon \alpha\in A\}$ of $\{x_n\}_{n=1}^\infty$ with $x_{n(\alpha)}\to x$.

\item \textbf{A.} $X$ is referred to as \textit{countably compact} \cite[p.~162]{K55}, iff every countable open cover of $X$ admits a finite subcover, iff each sequence has a cluster point in $X$, iff $X$ possesses the countable FIP (finite intersection property), and iff each infinite subset of $X$ has an $\omega$-accumulation point in $X$.

\item\textbf{B.} A countably compact space is pseudo-compact; i.e., every continuous real-valued function on it is bounded. However, the countable compactness is essentially weaker than the compactness. For example, \cite[Problem~5E-(e)]{K55} and the product of two countably compact spaces need not be countably compact \cite{E89}. However, if $X$ is compact and $Y$ countably compact, then $X\times Y$ is a countably compact space. See Theorem~\ref{5.2.7}-(2) for another condition for this.

\item\textbf{C.} In addition, we notice that there exists an $\mathcal{N}$-space $B$ and a countably compact completely regular space $C$ and a separately continuous function $f\colon B\times C\rightarrow\mathbb{R}$ such that the set of points of continuity is not dense in $B\times C$ (see \cite[Ex.~1.4]{BP05}).
\end{se}

In Calbrix-Troallic (1979) \cite{CT79} (or \cite[Thm.~6]{SR83}) it is proved that every separable Baire space has the $\mathcal{N}$-property. It turns out that this theorem can be extended to a $\Pi$-separable non-meager space via the following so-called joint continuity theorem.

\begin{thm}\label{2.5}
Let $X$ be a $\Pi$-separable space, $Y$ a space such that $Y\times Y$ is countably compact, $Z$ a pseudo-metric space, and $X_o\in\mathscr{O}(X)$. If $f\colon X_o\times Y\rightarrow Z$ is a separately continuous mapping, then there exists a residual set $R$ in $X_o$ such that $f$ is jointly continuous at each point of $R\times Y$.
\end{thm}

\begin{proof}
Let $\rho$ be a pseudo-metric for $Z$. Given $ n\in \mathbb{N}$, we can define a subset of $X_o$:
\begin{enumerate}
    \item[] $E_n=\left\{x\in X_o\,|\, \exists \, y\in Y\textit{ s.t. } \rho\textrm{-diam}(f(U\times V))>1/n \, \forall \,(U,V)\in \mathfrak{N}_{x}(X)\times \mathfrak{N}_{y}(Y)\right\}$.
\end{enumerate}
Set $ \mathrm{D}=\bigcup_{n\in\mathbb{N}}E_n $. Then $ f $ is jointly continuous at each point of $ (X_o\backslash \mathrm{D})\times Y $. We need only prove that $\mathrm{D}$ is meager in $X_o$. For that, by way of contradiction, suppose that $\mathrm{D}$ is non-meager in $X_o$. Then $ U_1:=\text{int}\,\overline{E}_\ell\neq\emptyset$ for some $\ell\in \mathbb{N} $, such that $ U_1\subseteq \overline{E}_\ell$ is a non-meager set in $X_o$ because
$\mathrm{D}\subseteq\left(\bigcup_{n\in\mathbb{N}}\text{int}\,\overline{E}_n\right)\cup\left(\bigcup_{n\in\mathbb{N}}\overline{E}_n\backslash\text{int}\,\overline{E}_n\right)$.
Assume $X=\prod_{i\in I}X_i$ is the product of a family of separable spaces. Let $\{a_{i,k}\,|\,k\in\mathbb{N}\}$, for each $i\in I$, be a dense sequence in $X_i$. Let $\flat=(\flat_i)_{i\in I}\in X_o$ be any fixed point. Given any finite set $I^\prime\subset I$ and $\vec{k}=(k_i)_{i\in I^\prime}\in\mathbb{N}^{I^\prime}$, let $|\vec{k}|=\sum_{i\in I^\prime}k_i$ and $\flat_{I^\prime,\vec{k}}=(a_{i,k_i})_{i\in I^\prime}\times(\flat_i)_{i\in I\setminus I^\prime}\in X_o$, where we have ignored the points $\flat_{I^\prime,\vec{k}}\notin X_o$.
Next we shall introduce a strategy $\tau$ with $\tau(X_o)=U_{1}$ for Player $\beta$ in the BM($X_o$)-game as follows:

Let $\tau(X_o)=U_1$; and for all $ V_1\in \mathscr{O}(U_1) $ and $ x_{1}\in V_1\cap E_\ell$ there is a point $ y_{1}\in Y $ such that $\rho\textrm{-diam}(f(U\times V))>1/n$ for all $(U,V)\in \mathfrak{N}_{x_1}(X)\times \mathfrak{N}_{y_1}(Y)$. Then there exists a point
$(x_{1}^{\prime},y_{1}^{\prime})\in V_1\times Y$ and a set $ \tau(X_o, V_1)=U_{2}=U_2^\prime\times\prod_{i\in I\setminus I_1}X_i\in \mathfrak{N}_{x_{1}^\prime}^\textrm{o}(V_1)$ where $I_1$ is some finite subset of $I$ and $U_2^\prime\in\mathscr{O}(\prod_{i\in I_1}X_i)$, such that:
\begin{gather*}
\rho(f(x_{1},y_{1}),f(x_{1}^{\prime},y_{1}^{\prime}))>\frac{1}{\ell};\\
\rho(f(U_{2}\times \{y_{1}^{\prime}\}),f(x_{1}^{\prime},y_{1}^{\prime}))<\frac{1}{6\ell},\quad
\rho(f(U_{2}\times \{y_{1}\}),f(x_{1},y_{1}))<\frac{1}{6\ell};\\
\rho(f(\flat,y_{1}),f(\flat,y_{1}^{\prime}))<\frac{1}{6\ell}.
\end{gather*}
For all $V_2\in \mathscr{O}(U_2)$ and $x_2\in V_2\cap E_\ell$, there is a point $y_2\in Y$ such that
$\rho\textrm{-diam}(f(U\times V))>1/n$ for all $(U,V)\in \mathfrak{N}_{x_2}(X)\times \mathfrak{N}_{y_2}(Y)$. Then there exists a point 
$(x_2^{\prime},y_2^{\prime})\in V_2\times Y$ and an open set $\tau(X_o, V_1, V_2)=U_3=U_3^\prime\times\prod_{i\in I\setminus I_2}X_i\in \mathfrak{N}_{x_2^\prime}^\textrm{o}(V_2)$ where $I_2$ is some finite subset of $I$ with $I_1\subseteq I_2$ and $U_3^\prime\in\mathscr{O}(\prod_{i\in I_2}X_i)$, such that:
\begin{gather*}
\rho(f(x_2,y_2),f(x_2^{\prime},y_2^{\prime}))>\frac{1}{\ell};\\
\rho(f(U_3\times \{y_2^{\prime}\}),f(x_2^{\prime},y_2^{\prime}))<\frac{1}{6\ell},\quad
\rho(f(U_3\times \{y_2\}),f(x_2,y_2))<\frac{1}{6\ell};\\
\rho(f(\flat_{I_1,\vec{k}_1},y_2),f(\flat_{I_1,\vec{k}_1},y_2^{\prime}))<\frac{1}{6\ell}\quad(\forall \vec{k}_1\in\mathbb{N}^{I_1}\textrm{ s.t. }|\vec{k}|\le \max\{1,\# I_1\}).
\end{gather*}
Inductively, we can find a sequence $I_1\subseteq I_2\subseteq I_3\subseteq\dotsm$ of finite subsets of $I$, a strategy $\tau$ for Player $\beta$ and a $\tau$-play $\{(U_n,V_n)\}_{n=1}^{\infty}$ of BM($X_o$) with $U_{n+1}=\tau(X_o, V_1,\dotsc,V_n)=U_{n+1}^\prime\times\prod_{i\in I\setminus I_n}X_i$ and $(x_n,y_n)\in (V_n\cap E_\ell)\times Y$, $(x_n^{\prime},y_n^{\prime})\in U_{n+1}\times Y $ such that:
\begin{gather*}
\rho(f(x_n,y_n),f(x_n^{\prime},y_n^{\prime}))>\frac{1}{\ell};\\
\rho(f(U_{n+1}\times \{y_n^{\prime}\}),f(x_n^{\prime},y_n^{\prime}))<\frac{1}{6\ell},\quad
\rho(f(U_{n+1}\times \{y_n\}),f(x_n,y_n))<\frac{1}{6\ell};\\
\rho(f(\flat_{I_j,\vec{k}_j},y_{n+1}),f(\flat_{I_j,\vec{k}_j},y_{n+1}^{\prime}))<\frac{1}{6\ell}\quad (\forall \vec{k}_j\in\mathbb{N}^{I_j}\textrm{ s.t. }|\vec{k}_j|\le\max\{n,\# I_n\}, j=1,\dotsc,n).
\end{gather*}
Let $J=\bigcup_{n=1}^\infty I_n\subseteq I$. Since $U_1$ is non-meager in $X_o$, it follows by Theorem~\ref{2.3} that $\tau$ with $\tau(X_o)=U_1$ is not a winning strategy for Player $\beta$ so that Player $\alpha$ has a choice
$\{V_n\}_{n=1}^{\infty}$ such that $\bigcap_{n=1}^{\infty}U_n \neq\emptyset$. We can choose $x^*=(x_i^*)_{i\in I}\in \bigcap_{n=1}^{\infty}U_n$ such that $x_i^*=\flat_i\ \forall i\in I\setminus J$. Since $Y\times Y$ is countably compact, we may assume (a subnet of) $(y_n,y_n^{\prime})\rightarrow (y,y^{\prime})\in Y\times Y$. Thus, for all $n,j\in\mathbb{N}$,
\begin{gather*}
\rho(f(x^*,y_n^{\prime}),f(x_n^{\prime},y_n^{\prime}))<\frac{1}{6\ell},\quad \rho(f(x^*,y_n),f(x_n,y_n))<\frac{1}{6\ell},\\
\rho(f(\flat_{I_j,\vec{k}_j},y),f(\flat_{I_j,\vec{k}_j},y^{\prime}))\leq\frac{1}{6\ell}\quad  \forall \vec{k}_j\in\mathbb{N}^{I_j}.
\end{gather*}
Since $\{\flat_{I_j,\vec{k}_j}\,|\, j\in\mathbb{N}\ \&\ \vec{k}_j\in\mathbb{N}^{I_j}\}$ is dense in $\prod_{i\in J}X_i\times(\flat_i)_{i\in I\setminus J}$, we can assume (a subnet of) $\flat_{I_j,\vec{k}_j}\rightarrow x^*$. Thus,
$\rho(f(x^*,y),f(x^*,y^{\prime}))\leq\frac{1}{6\ell}$, and so, as $n$ sufficiently big
\begin{equation*}\begin{split}
\frac{1}{\ell}&<\rho(f(x_n^{\prime},y_n^{\prime}),f(x_n,y_n)) \\
&\leq\rho(f(x_n^{\prime},y_n^{\prime}),f(x^*,y_n^{\prime}))+\rho(f(x^*,y_n^{\prime}),f(x^*,y^{\prime}))\\
&\quad+\rho(f(x^*,y^{\prime}),f(x^*,y))+\rho(f(x^*,y),f(x^*,y_n))+\rho(f(x^*,y_n),f(x_n,y_n))<\frac{1}{\ell}.
\end{split}\end{equation*}
This is impossible and thus, the proof is complete.
\end{proof}

\begin{2.5A}
If $X$ is an open subspace of a $\Pi$-separable space and if $X$ is Baire (resp. non-meager), then $X$ is an $\mathcal{N}$-space (resp. a g$\mathcal{N}$-space).
\end{2.5A}

This exactly proves Theorem~\ref{1.3}-(1) stated in $\S$\ref{s1} and generalizes \cite{CT79} and \cite[Thm.~6]{SR83}.

\begin{2.5B}
Let $X$ be a space containing a separable non-meager subset $F$. Let $f\colon X\times Y\rightarrow Z$ be a separately continuous mapping, where $Y\times Y$ is countably compact and $Z$ a pseudo-metric space. Then there exists a a point $x\in X$ such that $f$ is jointly continuous at each point of $\{x\}\times Y$. (So $X$ is a g$\mathcal{N}$-space.)
\end{2.5B}

Finally, by using Theorem~\ref{2.5} and a slight modification of the proof of Theorem~\ref{1.2} we can readily prove the following:

\begin{cor}\label{2.6}
Let $G$ be an open subgroup of a $\Pi$-separable non-meager right-topological group and $X$ a completely regular space such that $X\times X$ is countably compact. If $G\!\curvearrowright_\pi\!X$ is separately continuous, then $G\!\curvearrowright_\pi\!X$ is a topological flow.
\end{cor}

\begin{se}[Pseudo-complete space]\label{2.7}
A space $X$ is called \textit{pseudo-complete} \cite{O60} if $X$ is quasi-regular and there exists a sequence $\{\mathcal{B}(n)\}_{n=1}^\infty$ of $\pi$-base in $X$ such that whenever $U_n\in\mathcal{B}(n)$ and $U_n\supseteq\overline{U}_{n+1}$, then $\bigcap_{n=1}^\infty U_n\not=\emptyset$.
\end{se}

\begin{thm}\label{2.8}
Any pseudo-complete $\Pi$-separable space is a Choquet $\mathcal{N}$-space.
\end{thm}

\begin{proof}
Let $X$ be a pseudo-complete space. Clearly, $X$ is a Choquet space (i.e., $\alpha$-favorable of BM($X$); cf. \cite[(5.1)]{O60}). Thus, $X$ is an $\mathcal{N}$-space by Theorems~\ref{2.2} and \ref{2.5}. 
\end{proof}

\begin{rem}\label{2.9}
Any homogeneous non-meager topological space is Baire. In particular, any non-meager locally $\Pi$-separable left/right-topological group $G$ is an $\mathcal{N}$-space.
\end{rem}

\begin{proof}
Let $X$ be a homogeneous non-meager topological space. By Banach's category theorem (Thm.~A.2), there exists an open Baire subspace of $X$. Thus, $X$ is locally Baire so that $X$ is Baire. Further, if $G$ is non-meager locally $\Pi$-separable, then by Theorem~\ref{2.5} $G$ is an $\mathcal{N}$-space.
\end{proof}

\begin{rem}\label{2.10}
Let $X$ be a $\Pi$-separable space and $M(X)=\bigcup\{V\,|\,V\in\mathscr{O}(X)\textrm{ is meager in }X\}$. Then by Theorem~A.2, it follows that if $X$ is non-meager (so a g$\mathcal{N}$-space), then $X\setminus \overline{M(X)}\not=\emptyset$ is Baire and it is an $\mathcal{N}$-subspace of $X$ by Theorem~\ref{2.5}.
\end{rem}

Finally it should be noticed that if $X_o$ is an open subset of a non-normal space $X$, then one could not extend a continuous function $f\colon X_o\rightarrow\mathbb{R}$ to $X$. Thus, the $\mathcal{N}$-property need not be hereditary to open subsets in general; and in Theorem~\ref{2.5}, considering $f\colon X_o\times Y\rightarrow Z$ is better than considering $f\!\!\upharpoonright_{X_o\times Y}$ for some $f\colon X\times Y\rightarrow Z$. In addition, if $X_o$ is an open non-meager subset of a $\Pi$-separable space, then there exists an open Baire subspace $V$ of $X_o$ such that $V$ is $\Pi$-separable even if $X_o$ is not $\Pi$-separable itself. However, a residual subset of $V$ is possibly smaller than that of $X_o$.

\section{$\Pi$-pseudo-metric spaces, mixtures and g$\mathcal{N}$-property}\label{s3}
This section will be devoted to proving Theorems~\ref{1.3}-(2a) and (3) and the necessity part of Theorem~\ref{1.3}-(2b) stated in $\S$\ref{s1} under the guise of Theorems~\ref{3.21}, \ref{3.11} and \ref{3.4}. 

\begin{se}[Christensen/Saint-Raymond game, quasi-regular spaces and $\Pi$-pseudo-metric spaces]\label{3.1}
Let $X$ be a topological space. Then:
\item \textbf{A.} By a \textit{$\mathcal{J}_{\!p}(X)$-play} $\{(B_i;A_i,a_i)\}_{i=1}^\infty$ played by Player $\beta$ and Player $\alpha$ on $X$~(cf.~\cite{C81,SR83,D86,D87}), it means a sequence of elements of $\mathscr{O}(X)\times\mathscr{O}(X)\times X$ with $B_i\supseteq A_i\supseteq B_{i+1}$ for all $i\in\mathbb{N}$, where $B_i$ and $(A_i,a_i)$ are picked alternately by Player $\beta$ and Player $\alpha$, respectively; moreover, Player $\beta$ is granted the privilege of the first move as in the BM($X$)-game. Following Saint-Raymond (1983) \cite{SR83}, our W.C. is defined as follows:
\begin{enumerate}[\textbf{W.C.}:]
    \item Player $\alpha$ \textit{wins} this $\mathcal{J}_{\!p}(X)$-play $\{(B_i;A_i,a_i)\}_{i=1}^\infty$, if $\overline{\{a_i\,|\,i\in\mathbb{N}\}}\cap\left({\bigcap}_{i=1}^\infty A_i\right)\not=\emptyset$; otherwise, Player $\beta$ \textit{wins} this play.
\end{enumerate}
Note here that this game is denoted by $G_\sigma(X)$ or $\mathscr{G}_\sigma(X)$ in \cite{C81,SR83,B93}. Next, $X$ is called
\begin{enumerate}[(i)]
\item \textit{$\alpha$-favorable of $\mathcal{J}_{\!p}$} if Player $\alpha$ has a winning strategy in the $\mathcal{J}_{\!p}(X)$-game;

\item \textit{$\beta$-d\'{e}favorable of $\mathcal{J}_{\!p}$} if Player $\beta$ has no winning strategy in the $\mathcal{J}_{\!p}(X)$-game.
\end{enumerate}
Note that the $\beta$-d\'{e}favorable of $\mathcal{J}_{\!p}(X)$ $\Rightarrow$ the $\beta$-d\'{e}favorable of BM($X$).

\begin{note*}
If we write $L_i=\{a_1,\dotsc,a_i\}$, then $L_{i}\subseteq L_{i+1}$ and the play $\{(B_i;A_i,a_i)\}_{i=1}^\infty$ may be denoted by $\{(B_i;A_i,L_i)\}_{i=1}^\infty$ and $\overline{\{a_i\,|\,i\in\mathbb{N}\}}\cap\left({\bigcap}_{i=1}^\infty A_i\right)\not=\emptyset$ by $\left(\overline{\bigcup_iL_i}\right)\cap\left({\bigcap}_{i=1}^\infty A_i\right)\not=\emptyset$. Following Debs (1986) \cite{D86}, when each $L_i$ is a $\mathscr{K}$-analytic (resp. compact) set in $X$, then this play is called a $\mathcal{J}_{\!a}(X)$-play (resp.~$\mathcal{J}_{\!c}(X)$-play \cite{T85}); since $L_i$ consists of $i$ points, the play we consider here is called a $\mathcal{J}_{\!p}(X)$-play. If $L_i=X$ for all $i\in\mathbb{N}$, then this play, denoted $\mathcal{J}(X)$, is exactly the BM($X$)-play as in Definition~\ref{2.1}A. So, $\mathcal{J}(X)\supseteq\mathcal{J}_{\!a}(X)\supseteq\mathcal{J}_{\!c}(X)\supseteq\mathcal{J}_{\!p}(X)$. 
\end{note*}

\item \textbf{B.} $X$ is called \textit{quasi-regular} if for every $U\in\mathscr{O}(X)$ there exists a member $V\in\mathscr{O}(X)$ such that $\overline{V}\subseteq U$ (cf.~Oxtoby 1960 \cite{O60} and McCoy 1975 \cite{M75}).

\item \textbf{C.} $X$ is called a \textit{$\Pi$-pseudo-metric space}, if there exists a family $\{X_i\colon i\in I\}$ of pseudo-metric spaces such that $X\cong\prod_{i\in I}X_i$. In that case, we shall identify $X$ with $\prod_{i\in I}X_i$ if no confusion. Clearly, a $\Pi$-pseudo-metric space need not be pseudo-metrizable itself; but a pseudo-metric space is of course a $\Pi$-pseudo-metric space.

Although there is no special constraint for the sequence $\{a_i\}_{i=1}^\infty$ (e.g., $a_i\in A_i\ \forall i\in\mathbb{N}$) in the $\mathcal{J}_{\!p}(X)$-play, similar to the BM($X$)-game the $\beta$-d\'{e}favorable of $\mathcal{J}_{\!p}(X)$ may be hereditary to open subspaces as follows:

\begin{3.1D}
Let $X$ be any topological space. Then the following two statements are satisfied:
\begin{enumerate}[(1)]
\item $X$ is $\beta$-d\'{e}favorable of $\mathcal{J}_{\!p}$ iff every $U\in\mathscr{O}(X)$ is a $\beta$-d\'{e}favorable space of $\mathcal{J}_{\!p}$ itself.
\item $X$ is $\alpha$-favorable of $\mathcal{J}_{\!p}$ iff every $U\in\mathscr{O}(X)$ is an $\alpha$-favorable space of $\mathcal{J}_{\!p}$ itself.
\end{enumerate}
\end{3.1D}

\begin{proof}
(1): Sufficiency is obvious; for if $\tau$ is a strategy for Player $\beta$ in the $\mathcal{J}_{\!p}(X)$-game, then $\tau$ is also a strategy for Player $\beta$ in the $\mathcal{J}_{\!p}(\tau(X))$-game. For necessity, let $U\in\mathscr{O}(X)$ be not $\beta$-d\'{e}favorable of $\mathcal{J}_{\!p}$. Then there exists a winning strategy $\tau$ for Player $\beta$ in the $\mathcal{J}_{\!p}(U)$-game. We can define a strategy $\sigma$ for Player $\beta$ in the $\mathcal{J}_{\!p}(X)$-game accompanied by $\tau$ as follows:
Let $u\in U$ and set $\sigma(X)=\tau(U)=U_1\in\mathscr{O}(U)$. For any $(V_1,a_1)\in\mathscr{O}(U_1)\times X$, put $b_1=a_1$ if $a_1\in U$, $b_1=u$ if $a_1\notin U$. Now set $\sigma(X;V_1,a_1)=\tau(U;V_1,b_1)=U_2\in\mathscr{O}(V_1)$. For any $(V_2,a_2)\in\mathscr{O}(U_2)\times X$, put $b_2=a_2$ if $a_2\in U$, $b_2=u$ if $a_2\notin U$. Then set $\sigma(X;V_1,a_1;V_2,a_2)=\tau(U;V_1,b_1;V_2,b_2)=U_3\in\mathscr{O}(V_2)$. Repeating this indefinitely, we can define a strategy $\sigma$ for Player $\beta$ in the $\mathcal{J}_{\!p}(X)$-game accompanied by $\tau$. As $X$ is $\beta$-d\'{e}favorable of $\mathcal{J}_{\!p}$, it follows that there exists a $\sigma$-play $\{(U_i;V_i,a_i)\}_{i=1}^\infty$ of $\mathcal{J}_{\!p}(X)$ accompanied by the $\tau$-play $\{(U_i;V_i,b_i)\}_{i=1}^\infty$ of $\mathcal{J}_{\!p}(U)$ such that
$$
\overline{\{a_i\colon i\in\mathbb{N}\}}\cap\left({\bigcap}_{i=1}^\infty U_i\right)\not=\emptyset\quad \textrm{and}\quad \overline{\{b_i\colon i\in\mathbb{N}\}}\cap\left({\bigcap}_{i=1}^\infty U_i\right)=\emptyset.
$$
Let $A=\{a_i\,|\,i\in\mathbb{N}\textrm{ s.t. }a_i\notin U\}$. Then $\overline{A}\cap\left({\bigcap}_{i=1}^\infty U_i\right)\not=\emptyset$; and so, $A\cap U\not=\emptyset$, a contradiction.

(2): If Player $\alpha$ has a winning strategy in the $\mathcal{J}_{\!p}(U)$-game for any $U\in\mathscr{O}(X)$, then Player $\alpha$ also has a winning strategy in the $\mathcal{J}_{\!p}(X)$-game. So sufficiency is true.
Similar to the case (1) necessity is also true. The proof is complete.
\end{proof}

Since an open subset $U$ of a $\Pi$-pseudo-metric space $X$ need not have the cylindrical representation $U=\prod_{i\in I}U_i$, we cannot guarantee that $U$ is a $\Pi$-pseudo-metric space itself. However, if $U$ is an open non-meager subset of a $\Pi$-pseudo-metric space, then there always exists an open Baire subspace $V$ of $U$ such that $V$ is $\Pi$-pseudo-metrizable.

A regular space is of course quasi-regular; but not vice versa. In fact, unlike the regularity, the quasi-regularity is even not hereditary to closed subsets. Here are two counterexamples:

\begin{3.1E}[Suggested by the reviewer]
There exists a topological space $X$ which is countable and Hausdorff, with a dense open set $X_o$ homeomorphic to the space $\mathbb{Q}$ of rationals such that $X\setminus X_o$ is discrete and that every two non-void open subsets never have disjoint closures. Therefore every member of $\mathscr{O}(X_o)$ has cluster points in $X\setminus X_o$, and $X$ is connected.
This space is not quasi-regular since no member of $\mathscr{O}(X_o)$ can have its closure contained in $X_o$. Let $(x_n)_{n\in\mathbb{N}}$ be an enumeration of $X$ and consider the subspace $Y$ of $X\times\mathbb{Q}$ defined by
$Y =(X\times\{0\})\cup\{(x_n,2^{-p})\colon 0\le p\le n\}$
whose closed subset $Y_0 = X\times \{0\}$ is homeomorphic to $X$ hence Hausdorff non-quasi-regular. Then $Y_0$ is nowhere dense and every member of $\mathscr{O}(Y)$ contains some clopen singleton $\{(x_n,2^{-p})\}$. Thus $Y$ is quasi-regular Hausdorff.
\end{3.1E}

\begin{3.1F}
There exists a compact $T_1$-space $X$ which is countable, with a dense open set $X_0$
homeomorphic to the space $\mathbb{Q}$ of rationals such that $X\setminus X_0\not=\emptyset$ is discrete and $\overline{U}\cap\overline{V}\not=\emptyset$ for all $U,V\in\mathscr{O}(X)$. For example, $X=\mathbb{Q}\cup\{\infty\}$ is the one-point compactification of $\mathbb{Q}$ and $X_0=\mathbb{Q}$, where $\mathbb{Q}$ is regarded as a subspace of $\mathbb{R}$; thus, $\infty\in\overline{U}$ for every $U\in\mathscr{O}(X)$.
 Then $X$ is neither a Baire space nor a quasi-regular space, since no member of $\mathscr{O}(X)$ can have its closure contained
 in $X_0\in\mathscr{O}(X)$. Let $\{x_n\}_{n\in \mathbb{N}}$ be an enumeration of $X$ and we consider the subspace $Y$ of $X\times\mathbb{Q}$ defined by
 $Y =X\times\{0\}\cup\{(x_n, 2^{-j})\,|\, 1\le n\le j<\infty\}$
 whose closed subset $Y_0 = X\times\{0\}$ is homeomorphic to $X$ hence not quasi-regular. Clearly, $Y_0$ is nowhere
 dense in $Y$ and then every member of $\mathscr{O}(Y)$ contains some clopen singleton $\{(x_n, 2^{-j})\}$ in $Y$. Thus, $Y$ is
 quasi-regular, $T_1$, and non-regular such that $Y_0$ is a closed non-quasi-regular subspace of $Y$. In fact, $Y$ is a compact Baire space.
\end{3.1F}
\end{se}

\begin{lem}\label{3.2}
Let $X$ be a locally countably compact quasi-regular space. Then $X$ is Baire; and moreover, $X$ is $\beta$-d\'{e}favorable of $\mathcal{J}_{\!p}$.
\end{lem}

\begin{proof}
Let $\tau$ be a strategy for Player $\beta$ in the $\mathcal{J}_{\!p}(X)$-game. Let $U_1=\tau(X)$. Since $X$ is quasi-regular locally countably compact, we can select $a_1\in V_1\in\mathscr{O}(U_1)$ such that $\overline{V}_1\subseteq U_1$ is countably compact. Now let $U_2=\tau(X;V_1,a_1)\in\mathscr{O}(V_1)$ and then we can select $a_2\in V_2\in\mathscr{O}(U_2)$ such that $\overline{V}_2\subseteq U_2$. Continue this indefinitely, we can define a $\tau$-play $\{(U_i;V_i,a_i)\}_{i=1}^\infty$ of $\mathcal{J}_{\!p}$ such that $\emptyset\not=\bigcap_{n=1}^\infty\overline{\{a_i\,|\,i\ge n\}}\subseteq\bigcap_{n=1}^\infty\overline{V}_n=\bigcap_{n=1}^\infty U_n$. Thus, $\tau$ is not a winning strategy for Player $\beta$ in the $\mathcal{J}_{\!p}(X)$-game from Definition~\ref{3.1}A.
This also implies that there is no winning strategy for Player $\beta$ in the BM($X$)-game. Therefore, $X$ is Baire by Theorem~\ref{2.2}. The proof is complete.
\end{proof}

In fact, locally countably compact quasi-regular spaces are $\alpha$-favorable of $\mathcal{J}_{\!p}$ (so, Choquet spaces). 
Note that if $X$ is not quasi-regular, then the statement of Lemma~\ref{3.2} is false as shown by Example~\ref{3.1}F.

If $I$ is a finite set and if each $(X_i,\rho_i)$, $i\in I$, is a pseudo-metric space~\cite[p.~119]{K55}, then the product $\prod_{i\in I}X_i$ is also a pseudo-metric space with pseudo-metric $\rho_I^{}\colon (\prod_{i\in I}X_i)\times(\prod_{i\in I}X_i)\rightarrow\mathbb{R}_+$ that is canonically defined by $\rho_I^{}(x,y)=\max\{\rho_i(x_i,y_i)\colon i\in I\}$ for all $x=(x_i)_{i\in I}$, $y=(y_i)_{i\in I}\in \prod_{i\in I}X_i$.

Saint-Raymond (1983) \cite[Thm.~7]{SR83} asserts that if $X$ is a metric space, then it is Baire if and only if it is $\beta$-d\'{e}favorable of $\mathcal{J}_{\!p}$. Further Chaber-Pol (2005) \cite[Thm.~1.2]{CP05} implies that a $\Pi$-metric space is Baire if and only if it is an $\mathcal{N}$-space. In fact, we can extend this result as follows:

\begin{lem}\label{3.3}
Let $X$ be a $\Pi$-pseudo-metric space. Then the following are pairwise equivalent:
\begin{enumerate}[(1)]
\item $U\in\mathscr{O}(X)$ is non-meager;
\item Player $\beta$ has no winning strategy $\tau$ with $\tau(X)=U$ in the $\mathcal{J}_{\!p}(X)$-game;
\item Player $\beta$ has no winning strategy $\tau$ with $\tau(U)=U$ in the $\mathcal{J}_{\!p}(U)$-game.
\end{enumerate}
Consequently, if $X$ is a $\Pi$-pseudo-metric space, then $X$ is Baire if and only if it is $\beta$-d\'{e}favorable of $\mathcal{J}_{\!p}$.
\end{lem}

\begin{proof}
Let $X=\prod_{i\in I}X_i$, where each factor $X_i$ is a pseudo-metric space with a pseudo-metric $\rho_i$. Given $J\subset I$ a finite set,  $\rho_J^{}$ is the pseudo-metric on $\prod_{i\in J}X_j$ induced naturally by $\{\rho_i\colon i\in J\}$.

$(1)\Rightarrow(3)$: Suppose that $U_1=U\in\mathscr{O}(X)$ is a non-meager set in $X$. Let $\tau$ with $\tau(U)=U_1$ be a strategy for Player $\beta$ in the $\mathcal{J}_{\!p}(U)$-game. Let $\tau^\prime(X)=U_1^\prime=\tau(U)$.
For all $V_1\in\mathscr{O}(U_1^\prime)$ and all $a_1=(a_{1,i})_{i\in I}\in U_1^\prime$, write $U_2=\tau(U;V_1,a_1)$ and then define $\tau^\prime(X, V_1)=U_2^\prime\in\mathscr{O}(U_2)\,(\subseteq\mathscr{O}(V_1))$ such that
$U_2^\prime=U_2^{\prime\prime}\times\prod_{i\in I\setminus I_1}X_i$, where $I_1\subset I$ is some finite set and $U_2^{\prime\prime}\in\mathscr{O}(\prod_{i\in I_1}X_i)$ such that $\rho_{I_1}^{}\textrm{-diam}(U_2^{\prime\prime})<1/2$. Select arbitrarily $V_2\in\mathscr{O}(U_2^\prime)$ and $a_2=(a_{2,i})_{i\in I}\in U_2^\prime$ such that $a_{2,i}=a_{1,i}$ for all $i\in I\setminus I_1$.
We simply write $U_3=\tau(U;V_1,a_1; V_2,a_2)$; define $\tau^\prime(X, V_1, V_2)=U_3^\prime\in\mathscr{O}(U_3)\subseteq\mathscr{O}(V_2)$ such that
$U_3^\prime=U_3^{\prime\prime}\times\prod_{i\in I\setminus I_2}X_i$, where $I_2\subset I$ is some finite set with $I_1\subseteq I_2$ and $U_3^{\prime\prime}\in\mathscr{O}(\prod_{i\in I_2}X_i)$ with $\rho_{I_2}^{}\textrm{-diam}(U_3^{\prime\prime})<1/2^2$.
Select arbitrarily $V_3\in\mathscr{O}(U_3^\prime)$ and $a_3=(a_{3,i})_{i\in I}\in U_3^\prime$ such that $a_{3,i}=a_{1,i}$ for all $i\in I\setminus I_2$.

Continue this indefinitely, we can then define a sequence $I_1\subseteq I_2\subseteq I_3\subseteq\dotsm$ of non-void finite subsets of $I$, a sequence $\{(U_n;V_n,a_n)\}_{n=1}^\infty$---a $\tau$-play of $\mathcal{J}_{\!p}(U)$-type, and a sequence $\{(U_n^\prime,V_n)\}_{n=1}^\infty$---a $\tau^\prime$-play of BM($X$), such that $a_{n+1}=(a_{n+1,i})_{i\in I}\in U_{n+1}^\prime=U_{n+1}^{\prime\prime}\times\prod_{i\in I\setminus I_n}X_i\subseteq U_{n+1}$ with $a_{n+1,i}=a_{1,i}$ for all $i\in I\setminus I_n$ and $\rho_{I_n}^{}\textrm{-diam}(U_{n+1}^{\prime\prime})<1/2^n$ for all $n\ge1$. By Theorem~\ref{2.3}, $\tau^\prime$ is not a winning strategy for Player $\beta$ in the BM($X$)-game; and so, there is a choice $\{V_n\}_{n=1}^\infty$ for Player $\alpha$ such that $\bigcap_{n=1}^\infty U_n^\prime\subseteq\bigcap_{n=1}^\infty U_n\not=\emptyset$. Thus, for any point $x=(x_i)_{i\in I}\in\bigcap_{n=1}^\infty U_n$ with $x_i=a_{1,i}\ \forall i\in I\setminus(\bigcup_{n=1}^\infty I_n)$, by $\rho_i(a_{n+1,i},x_i)\le 1/2^n\ \forall i\in I$, it follows that $a_{n,i}\to x_i$ in $(X_i,\rho_i)$ as $n\to\infty$. Hence $\tau$ is not a winning strategy for Player $\beta$ in the $\mathcal{J}_{\!p}(U)$-game. 

$(3)\Rightarrow(2)$: Obvious by Definition~\ref{3.1}A.

$(2)\Rightarrow(1)$: Obvious by Theorem~\ref{2.3}. This is because if $\tau$ with $\tau(X)=U$ is a winning strategy for Player $\beta$ in the BM($X$)-game (cf.~Def.~\ref{2.1}A), then it is also a winning strategy for Player $\beta$ in the $\mathcal{J}_{\!p}(X)$-game (cf.~Def.~\ref{3.1}A).

Finally, if $X$ is Baire, then every $U\in\mathscr{O}(X)$ is non-meager; and so, $X$ is $\beta$-d\'{e}favorable of $\mathcal{J}_{\!p}$ by (1)$\Leftrightarrow$(2). Conversely, if $X$ is $\beta$-d\'{e}favorable of $\mathcal{J}_{\!p}$, then every $U\in\mathscr{O}(X)$ is non-meager so that $X$ is Baire.
The proof is complete.
\end{proof}

If $X$ is the product of an uncountable family of pseudo-metric spaces in Lemma~\ref{3.3}, then $X$ is not a pseudo-metrizable space; and in addition, a non-meager space need not be Baire (Ex.~\ref{1.4}). In view of that, Lemma~\ref{3.3} is an essential improvement of \cite[Thm.~7]{SR83} (see Thm.~\ref{7.8}-(1)).

If $X$ is $\beta$-d\'{e}favorable of $\mathcal{J}_{\!p}$, then Player $\beta$ has no winning strategy $\tau$ such that $\tau(X)$ is non-meager ($\because$ $X$ is Baire and each $U\in\mathscr{O}(X)$ is non-meager in this case).
However, a space that admits no $\mathcal{J}_{\!p}$-winning strategy $\tau$ with $\tau(X)$ being non-meager for Player $\beta$ is not necessarily $\beta$-d\'{e}favorable of $\mathcal{J}_{\!p}$; for instance, $X$ is a meager space itself.
Now we shall prove a theorem, which together with Lemma~\ref{3.3} implies the necessity part of Theorem~\ref{1.3}-(2) stated in $\S$\ref{s1}:

\begin{thm}\label{3.4}
Let $X$ be such that Player $\beta$ has no winning strategy $\tau$ with $\tau(X)$ being non-meager in the $\mathcal{J}_{\!p}(X)$-game. Let $Y$ be such that $Y\times Y$ is countably compact and $Z$ a pseudo-metric space. If $f\colon X\times Y\rightarrow Z$ is a separately continuous mapping, then there exists a residual set $R\subseteq X$ such that $f$ is jointly continuous at each point of $R\times Y$.
Consequently, if $X$ is an open subspace of a $\Pi$-pseudo-metric space, then there is a residual set $R\subseteq X$ such that $f$ is jointly continuous at each point of $R\times Y$.
\end{thm}

\begin{proof}
By Lemma~\ref{3.3}, we need only prove the first part of Theorem~\ref{3.4}. For that, we let $\rho$ be the pseudo-metric for $Z$. For all $n\in \mathbb{N}$ let
\begin{enumerate}
    \item[] $E_n=\left\{x\in X\,|\, \exists \, y\in Y\textit{ s.t.} \,\rho\textrm{-diam}(f(U\times V))>\frac{1}{n} \, \forall \,(U,V)\in \mathfrak{N}_{x}(X)\times \mathfrak{N}_{y}(Y)\right\}$.
\end{enumerate}
Set $\mathrm{D}=\bigcup_{n\in\mathbb{N}}\overline{E}_n$. We need only prove that $\mathrm{D}$ is meager in $X$. By way of contradiction, suppose $ \mathrm{D} $ is non-meager in $X$. Then $ U_1:=\text{int}\,\overline{E}_\ell\neq\emptyset$ for some $\ell\in \mathbb{N} $, such that $ U_1 $ is non-meager in $X$ because
$\mathrm{D}=\left(\bigcup_{n\in\mathbb{N}}\text{int}\,\overline{E}_n\right)\cup\left(\bigcup_{n\in\mathbb{N}}\overline{E}_n\backslash\text{int}\,\overline{E}_n\right)$.
Next we introduce a strategy $\tau$ with $ \tau(X)=U_{1} $ for Player $ \beta $ in the $\mathcal{J}_{\!p}(X)$-game as follows:
Let $\tau(X)=U_1$; and for all $(V_1,a_1)\in \mathscr{O}(U_1)\times X$ and for every $ x_{1}\in V_1\cap E_\ell$ there is $y_1\in Y$ such that $\rho\textrm{-diam}(f(U\times V))>1/n$ for all $(U,V)\in \mathfrak{N}_{x_1}(X)\times \mathfrak{N}_{y_1}(Y)$. Then there exists $ (x_{1}^{\prime},y_{1}^{\prime})\in V_1\times Y$ and $\tau(X; V_1,a_1)=U_2\in \mathfrak{N}_{x_1^\prime}^\textrm{o}(V_1)$ such that:
\begin{gather*}
\rho(f(x_{1},y_{1}),f(x_{1}^{\prime},y_{1}^{\prime}))>{1}/{\ell},\\
\rho(f(U_{2}\times \{y_{1}^{\prime}\}),f(x_{1}^{\prime},y_{1}^{\prime}))<\frac{1}{6\ell},\quad
\rho(f(U_{2}\times \{y_{1}\}),f(x_{1},y_{1}))<\frac{1}{6\ell},\\
\rho(f(a_1,y_{1}),f(a_1,y_{1}^{\prime}))<\frac{1}{6\ell}.
\end{gather*}
Inductively, we can define a $\mathcal{J}_{\!p}(X)$-play $\{(U_i;V_i, a_i)\}_{i=1}^{\infty} $ with $ U_{i+1}=\tau(X; V_1,a_1;\dotsc;V_i,a_i)$ and $(x_i,y_i)\in (V_i\cap E_\ell)\times Y $, $(x_{i}^{\prime},y_{i}^{\prime})\in U_{i+1}\times Y$ such that:
\begin{gather*}
\rho(f(x_{i},y_{i}),f(x_{i}^{\prime},y_{i}^{\prime}))>\frac{1}{\ell},\\
\rho(f(U_{i+1}\times \{y_{i}^{\prime}\}),f(x_{i}^{\prime},y_{i}^{\prime}))<\frac{1}{6\ell},\quad \rho(f(U_{i+1}\times \{y_{i}\}),f(x_{i},y_{i}))<\frac{1}{6\ell},\\
\rho(f(a_{j},y_{i+1}),f(a_{j},y_{i+1}^{\prime}))<\frac{1}{6\ell}\ (j=1,\dotsc,i).
\end{gather*}
Since $U_1$ is non-meager, $\tau$ with $\tau(X)=U_1$ is not a winning strategy for Player $\beta$ so that Player $\alpha$ has a choice $ \{(V_i,a_i)\}_{i=1}^{\infty}$ with $\overline{\{a_i\colon i\in\mathbb{N}\}}\cap(\bigcap_{i=1}^{\infty}U_i) \neq\emptyset $. Let $ x\in \overline{\{a_i\colon i\in\mathbb{N}\}}\cap(\bigcap_{i=1}^{\infty}U_i)$. In addition, since $Y\times Y$ is countably compact, we may assume (a subnet of) $ (y_{i},y_{i}^{\prime})\rightarrow (y,y^{\prime})\in Y\times Y $. Thus, for all $ i,j\in\mathbb{N} $,
\begin{gather*}
\rho(f(x,y_{i}^{\prime}),f(x_{i}^{\prime},y_{i}^{\prime}))<\frac{1}{6\ell},\quad \rho(f(x,y_{i}),f(x_{i},y_{i}))<\frac{1}{6\ell},\quad
\rho(f(a_{j},y),f(a_{j},y^{\prime}))\leq\frac{1}{6\ell}.
\end{gather*}
By $x\in \overline{\{a_i\colon i\in\mathbb{N}\}}$, we can assume (a subnet of) $a_j\rightarrow x $. Thus,
$\rho(f(x,y),f(x,y^{\prime}))\leq\frac{1}{6\ell}$ and so
\begin{equation*}\begin{split}
\frac{1}{\ell}&<\rho(f(x_{i}^{\prime},y_{i}^{\prime}),f(x_{i},y_{i})) \\
&\leq\rho(f(x_{i}^{\prime},y_{i}^{\prime}),f(x,y_{i}^{\prime}))+\rho(f(x,y_{i}^{\prime}),f(x,y^{\prime}))\\
&\quad+\rho(f(x,y^{\prime}),f(x,y))+\rho(f(x,y),f(x,y_{i}))+\rho(f(x,y_{i}),f(x_{i},y_{i}))<\frac{1}{\ell}.
\end{split}\end{equation*}
This is impossible. The proof is complete.
\end{proof}

As analogous to the $\Pi$-separable space case, the second part of Theorem~\ref{3.4} is better than only choosing a basic open Baire subspace $U$ of $X$ such that for some residual set $R\subseteq U$, $f$ is jointly continuous at each point of $R\times Y$.

If $Y$ is a compact space, then $Y\times Y$ is compact so that $Y\times Y$ is countably compact. Now by Lemma~\ref{3.2} and Theorem~\ref{3.4} we can readily obtain the following.

\begin{cor}[{cf.~\cite[Thm.~5]{SR83}}]\label{3.5}
If $X$ is a $\beta$-d\'{e}favorable space of $\mathcal{J}_{\!p}$, then it is an $\mathcal{N}$-space. In particular, any locally countably compact quasi-regular space is an $\mathcal{N}$-space.
\end{cor}

\begin{cor}[{cf.~\cite[Thm.~1]{E57} for $G,X$ to be locally compact Hausdorff}]
Let $G$ be a quasi-regular locally countably compact right-topological group and $X$ a completely regular space such that $X\times X$ is countably compact. If $G\!\curvearrowright_\pi\!X$ is separately continuous, then it is a topological flow.
\end{cor}

\begin{proof}
Let $\{(t_i,x_i)\,|\,i\in A\}$ be any net in $G\times X$ with $(t_i,x_i)\to (t,x)\in G\times X$. If $t_ix_i\not\to tx$ in $X$ and $\Lambda_{i_0}=\overline{\{t_ix_i\,|\,i\ge i_0\}}$ for all $i_0\in A$, then we may assume $tx\notin\Lambda_{i_0}$ for some $i_0\in A$. Let $\psi\in C(X,[0,1])$ with $\psi\!\upharpoonright_{\Lambda_{i_0}}\equiv0$ and $\psi(tx)=1$. Then by Lemma~\ref{3.2} and Theorem~\ref{3.4}, there exists an element $g\in G$ such that $f=\psi\circ\pi\colon G\times X\rightarrow[0,1]$ is jointly continuous at each point of $\{g\}\times X$.
Then by $t_it^{-1}g\to g$ and $g^{-1}tx_i\to g^{-1}tx$, it follows that
$0=\psi(t_ix_i)\to\psi(tx)=1$, which is impossible. 
\end{proof}

\begin{cor}\label{3.7}
If $X$ is an open Baire subspace of a $\Pi$-pseudo-metric space, then $X$ is an $\mathcal{N}$-space.
\end{cor}

\begin{se}[$F$-group]\label{3.8}
Recall that a semitopological group is called an \textit{$F$-group} \cite{V77} if its inversion is continuous. Note that an Ellis group associated to a minimal flow is a compact $T_1$ $F$-group (not necessarily a topological group in general).
\end{se}

Finally we consider the case where $X\times X$ is locally countably compact instead of ``countably compact'' condition. The following result is known in the case that $G$ is regular (see, e.g., \cite[Thm.~5]{DX} by using a Baire curve theorem).

\begin{cor}\label{3.9}
Let $G$ be a quasi-regular locally countably compact $F$-group and $X$ a completely regular space such that $X\times X$ is locally countably compact. If $G\!\curvearrowright_\pi\!X$ is separately continuous, then it is a topological flow.
\end{cor}

\begin{proof}
It is enough to prove that $\pi$ is jointly continuous at each point of $\{e\}\times X$. Let $x_0\in X$ and suppose to the contrary that $\pi$ is not continuous at $(e,x_0)$. Then we may assume there exists a net $\{(t_i,x_i)\,|\,i\in\Lambda\}$ in $G\times X$ with $(t_i,x_i)\to(e,x_0)$ and such that $x_0=ex_0\notin \bigcap_{i\in\Lambda}\overline{\{t_jx_j\,|\,j\ge i\}}$. Then $x_0\notin W:=\overline{\{t_jx_j\,|\,j\ge i_0\}}$ for some $i_0\in\Lambda$. Further, there is a continuous function $\psi\colon X\rightarrow[0,1]$ such that $\psi(x_0)=0$ and $\psi\!\upharpoonright_W\equiv1$.
Let $U\in\mathfrak{N}_{x_0}(X)$ such that $U\times U$ is countably compact. Then we can choose a set $V\in\mathfrak{N}_e(G)$ such that $V^{-1}x_0\subseteq U$. Write $f\colon G\times U\rightarrow[0,1]$ for the restriction of $\psi\circ\pi$ to $G\times U$. Then by Lemma~\ref{3.2} and Theorem~\ref{3.4}, there exists a dense set $R\subseteq G$ such that $f$ is jointly continuous at each point of $R\times U$. Now, let $a\in V\cap R$. Then by $t_ia\to a$ and $a^{-1}x_i\to a^{-1}x_0\in U$, it follows that $1=\psi(t_ix_i)=f(t_ia,a^{-1}x_i)\to f(a,a^{-1}x_0)=\psi(x_0)=0$, which is impossible. The proof is complete.
\end{proof}

\begin{rem}\label{3.10}
Let $X$ be a $\Pi$-pseudo-metric space and $M(X)=\bigcup\{V\in\mathscr{O}(X)\,|\,V\textrm{ is meager}\}$. By Theorem~A.2, if $X$ is non-meager (so a g$\mathcal{N}$-space by Theorem~\ref{1.3}-(2b)), then $X\setminus \overline{M(X)}\not=\emptyset$ is Baire and it is an $\mathcal{N}$-subspace of $X$ by Lemma~\ref{3.3} and Theorem~\ref{3.4}.
\end{rem}

\begin{thm}\label{3.11}
    Let $X$ be a separable space and $Y$ a pseudo-metric space. Then:
    \begin{enumerate}[(a)]
        \item If $X\times Y$ is a Baire space, then it is $\beta$-d\'{e}favorable of $\mathcal{J}_{\!p}$ and an $\mathcal{N}$-space.

        \item If $X\times Y$ is a non-meager space, then it is a g$\mathcal{N}$-space.
    \end{enumerate}    
\end{thm}

\begin{proof}
First of all, assume that $X$ has no isolated points without loss of generality; for otherwise, the statement may be reduced to the special case $X=\{x\}$ and follows from Theorem~\ref{3.4}. Let $\rho$ be a pseudo-metric for $Y$ and write $B_\rho(y,r)=\{y^\prime\in Y\colon \rho(y,y^\prime)<r\}$ for all $y\in Y$ and $r>0$.

    (a): Suppose $X\times Y$ is a Baire space. Let $\{x_i\colon i\in\mathbb{N}\}$ be dense in $X$. 
    Let $\tau$ be a strategy for Player $\beta$ in the $\mathcal{J}_{\!p}(X\times Y)$-game. In view of Corollary~\ref{3.5}, we need only prove that $\tau$ is not a winning strategy for Player $\beta$ in the $\mathcal{J}_{\!p}(X\times Y)$-game.  
    For that, we will define a strategy $\tau^\prime$ for Player $\beta$ in the BM($X\times Y$)-game as follows:
    
    Firstly, let $U_1=\tau(X\times Y)\supseteq U_{1,1}\times U_{2,1}$, where $U_{1,1}\in\mathscr{O}(X)$ and $U_{2,1}=B_\rho(y_1,r_1)\in\mathscr{O}(Y)$ for some $r_1\in(0,1/2]$ and some $y_1\in Y$; then define $\tau^\prime(X\times Y)=U_{1,1}\times U_{2,1}$. Now, for all $V_{1}\in\mathscr{O}(U_{1,1}\times U_{2,1})$, let $a_1=(x_1,y_1)$, $U_2=\tau(X\times Y; V_1,a_1)\supseteq U_{1,2}\times U_{2,2}$, where $U_{1,2}\in\mathscr{O}(X)$ and $U_{2,2}=B_\rho(y_2,r_2)\in\mathscr{O}(Y)$ for some $r_2\in(0,1/2^2]$ and $y_2\in Y$; then define $\tau^\prime(X\times Y, V_1)=U_{1,2}\times U_{2,2}$.
    Next, for all $V_{2}\in\mathscr{O}(U_{1,2}\times U_{2,2})$, let $a_2=(x_2,y_2)$, $U_3=\tau(X\times Y; V_1,a_1;V_2,a_2)\supseteq U_{1,3}\times U_{2,3}$, where $U_{1,3}\in\mathscr{O}(X)$ and $U_{2,3}=B_\rho(y_3,r_3)\in\mathscr{O}(Y)$ for some $r_3\in(0,1/2^3]$ and $y_3\in Y$; then define $\tau^\prime(X\times Y,V_{1}, V_{2})=U_{1,3}\times U_{2,3}$. Repeating this procedure indefinitely, we can construct a $\tau$-play $\{(U_i;V_i,a_i)\}_{i=1}^\infty$ and define a strategy $\tau^\prime$ for Player $\beta$ in the BM($X\times Y$)-game.

    Since $X\times Y$ is a Baire space, by Theorem~\ref{2.2} there are sequences $\{V_{i}\}_{i=1}^\infty$ against $\tau^\prime$ played by Player $\alpha$ such that $\bigcap_{i=1}^\infty V_{i}=\bigcap_{i=1}^\infty U_i\not=\emptyset$. Then for $(x,y)\in\bigcap_{i=1}^\infty V_{i}$, we have that $(x,y)\in\bigcap_{i=1}^\infty U_i$ and $x\in\overline{\{x_i\colon i\in\mathbb{N}\}}$ and $y_i\to y$. Since $X$ has no isolated point, there is a subnet $\{x_{i(a)}\}$ of $\{x_i\}$ with $x_{i(a)}\to x$ and $y_{i(a)}\to y$.
    Thus, $\overline{\{a_i=(x_i,y_i)\colon i\in\mathbb{N}\}}\cap\left(\bigcap_{i=1}^\infty U_i\right)\not=\emptyset$. Therefore, $X\times Y$ is $\beta$-d\'{e}favorable of $\mathcal{J}_{\!p}$.

    (b): If $X\times Y$ is only a non-meager space, then by Banach's category theorem (Thm.~A.2), there exists a set $W\in\mathscr{O}(X\times Y)$ such that $W$ is a Baire subspace of $X\times Y$. Further, there exist $U\in\mathscr{O}(X)$ and $V\in\mathscr{O}(Y)$ such that $U\times V\subseteq W$ and $U\times V$ is Baire. Then by (a), $X\times Y$ is a g$\mathcal{N}$-space. 
    The proof is complete.
\end{proof}

If one of $X$ and $Y$ is a Choquet space and the other is a Baire space, then $X\times Y$ is a Baire space (cf.~\cite[Thm.~9]{SR83}).
However, a Baire (non-meager) space need not be an $\mathcal{N}$-space (a g$\mathcal{N}$-space) as shown by Example~\ref{7.5}. In view of that, the mixing of $X$-factor separability and $Y$-factor pseudo-metrizability in $X\times Y$ is critical for Theorem~\ref{3.11}.

\begin{cor}
    Let $G=G_1\times G_2$, where $G_1$ is a separable topological group and $G_2$ a metrizable topological group, and $X$ a locally compact regular space. If $G\curvearrowright X$ is separately continuous, then $G\curvearrowright X$ is a topological flow.
\end{cor}

\begin{proof}
    By Theorems~\ref{1.2} and \ref{3.11}.
\end{proof}

\begin{que}
    Let $X$ be a $\Pi$-separable space and $Y$ a pseudo-metric space such that $X\times Y$ is non-meager. Is $X\times Y$ a g$\mathcal{N}$-space?
\end{que}

\begin{que}
    Let $X$ be a separable space and $Y$ a $\Pi$-pseudo-metric space such that $X\times Y$ is non-meager. Is $X\times Y$ a g$\mathcal{N}$-space?
\end{que}

\begin{que}
    Let $X$ be a $\Pi$-separable space and $Y$ a $\Pi$-pseudo-metric space such that $X\times Y$ is non-meager. Is $X\times Y$ a g$\mathcal{N}$-space?
\end{que}

\begin{se}[$\sigma$-well $\alpha$-favorable/$\beta$-d\'{e}favorable spaces of $\mathcal{J}_{\!p}$~\cite{C81, SR83}]\label{3.16}
Let $X$ be any topological space. Let $\{(B_n;A_n,a_n)\}_{n=1}^\infty$ be a $\mathcal{J}_{\!p}(X)$-play as in Definition~\ref{3.1}A. Then in the sense of Christensen~\cite{C81}, 
\begin{enumerate}[\textbf{W.C.}:]
    \item Player $\alpha$ \textit{wins $\sigma$-well} this play, if $\{a_{n(i)}\}_{i=1}^\infty$ has a cluster point in $\bigcap_{n=1}^\infty A_n$ for every subsequence $\{a_{n(i)}\}_{i=1}^\infty$ of $\{a_n\}_{n=1}^\infty$; otherwise, Player $\beta$ \textit{wins $\sigma$-well} this play.
\end{enumerate}
\item\textbf{A.} $X$ is called a \textit{$\sigma$-well $\alpha$-favorable space of $\mathcal{J}_{\!p}$}, if Player $\alpha$ has a  $\sigma$-well winning strategy in the $\mathcal{J}_{\!p}(X)$-game.

\item\textbf{B.} $X$ is called a \textit{$\sigma$-well $\beta$-d\'{e}favorable space of $\mathcal{J}_{\!p}$}, if Player $\beta$ has no  $\sigma$-well winning strategy in the $\mathcal{J}_{\!p}(X)$-game. 
\end{se}

\begin{se}[$\tau$-well $\alpha$-favorable/$\beta$-d\'{e}favorable spaces of $\mathcal{J}_{\!p}$~\cite{C81, SR83}]\label{3.17}
Let $X$ be any topological space. Let $\{(B_n;A_n,a_n)\}_{n=1}^\infty$ be a $\mathcal{J}_{\!p}(X)$-play. Then in the sense of Christensen~\cite{C81},
\begin{enumerate}[\textbf{W.C.}:]
    \item Player $\alpha$ \textit{wins $\tau$-well} this play, if $\{a_{n(i)}\colon i\in D\}$ has a cluster point in $\bigcap_{n=1}^\infty A_n$ for every subnet $\{a_{n(i)}\colon i\in D\}$ of $\{a_n\}_{n=1}^\infty$; otherwise, Player $\beta$ \textit{wins $\tau$-well} this play.
\end{enumerate}
\item\textbf{A.} $X$ is called a \textit{$\tau$-well $\alpha$-favorable space of $\mathcal{J}_{\!p}$}, if Player $\alpha$ has a $\tau$-well winning strategy in the $\mathcal{J}_{\!p}(X)$-game.

\item\textbf{B.} $X$ is called a \textit{$\tau$-well $\beta$-d\'{e}favorable space of $\mathcal{J}_{\!p}$}, if Player $\beta$ has no  $\tau$-well winning strategy in the $\mathcal{J}_{\!p}(X)$-game.

\item\textbf{C.} Clearly, for any topological space $X$, in the $\mathcal{J}_{\!p}(X)$-game we have the implications:
$$
\tau\textrm{-well }\alpha\textrm{-favorable }\Rightarrow
\begin{cases} 
\sigma\textrm{-well }\alpha\textrm{-favorable}\\
\tau\textrm{-well }\beta\textrm{-d\'{e}favorable} 
\end{cases}
\!\!\Rightarrow\sigma\textrm{-well }\beta\textrm{-d\'{e}favorable}\Rightarrow\beta\textrm{-d\'{e}favorable}.
$$
\end{se}

Note that the class of $\tau$-well $\alpha$-favorable spaces of $\mathcal{J}_{\!p}$ is closed under arbitrary products~\cite{C81}. Moreover, the product of a $\sigma$-well $\beta$-d\'{e}favorable space and a $\tau$-well $\alpha$-favorable space of $\mathcal{J}_{\!p}$ is $\sigma$-well $\beta$-d\'{e}favorable \cite[Thm.~8]{SR83}. In fact, we can prove the following theorem, in which the ``$\alpha$-favorablity'' condition is essential for otherwise there are counterexamples \cite{C76, FK78}. 

\begin{thm}\label{3.18}
Let $X$ and $Y$ be two topological spaces; then we have the following two statements:
\begin{enumerate}[(1)]
\item If $X$ is $\tau$-well $\alpha$-favorable of $\mathcal{J}_{\!p}$ and $Y$ is $\sigma$-well $\beta$-d\'{e}favorable of $\mathcal{J}_{\!p}$, then $X\times Y$ is $\sigma$-well $\beta$-d\'{e}favorable of $\mathcal{J}_{\!p}$ (cf.~Saint-Raymond 1983 \cite[Thm.~8]{SR83}).

\item If $X$ is $\sigma$-well $\alpha$-favorable of $\mathcal{J}_{\!p}$ and $Y$ is $\tau$-well $\beta$-d\'{e}favorable of $\mathcal{J}_{\!p}$, then $X\times Y$ is $\sigma$-well $\beta$-d\'{e}favorable of $\mathcal{J}_{\!p}$.
\end{enumerate}  
\end{thm}

\begin{proof}
First of all, for all $U\in\mathscr{O}(X\times Y)$, let 
\begin{enumerate}
    \item[] $\mathcal{Y}[U]=\{W\in\mathscr{O}(Y)\colon \exists W^\prime\in\mathscr{O}(X) \textrm{ s.t. }W^\prime\times W\subseteq U\}$; 
\end{enumerate}
and for all $W\in\mathcal{Y}[U]$, let
\begin{enumerate}
    \item[] $W^*=\bigcup\{W^\prime\in\mathscr{O}(X)\,|\,W^\prime\times W\subseteq U\}\in\mathscr{O}(X)$.
\end{enumerate}
Then $W^*\times W\subseteq U$ for all $U\in\mathscr{O}(X\times Y)$ and all $W\in\mathcal{Y}[U]$.
Let $t$ be any strategy for Player $\beta$ in the $\mathcal{J}_{\!p}(X\times Y)$-game.
Since $X$ is a $\tau$-well ($\sigma$-well) $\alpha$-favorable space of $\mathcal{J}_{\!p}$, there is a $\tau$-well ($\sigma$-well) winning strategy $s$ for Player $\alpha$ in the $\mathcal{J}_{\!p}(X)$-game. We need only prove that $t$ is not a $\sigma$-well winning strategy for Player $\beta$ in the $\mathcal{J}_{\!p}(X\times Y)$-game. For that, we will introduce an auxiliary strategy $\theta$ for Player $\beta$ in the $\mathcal{J}_{\!p}(Y)$-game as follows:

Let $U_1=t(X\times Y)$ and $\theta(Y)=W_1\in\mathcal{Y}[U_1]$; then, let $(V_{1,1},x_1)=s(W_1^*)$; and for every $(V_{2,1},y_1)\in\mathscr{O}(W_1)\times Y$, as Player $\alpha$'s possible answer to Player $\beta$'s 1st move $W_1$ in the $\mathcal{J}_{\!p}(Y)$-game, write $(V_1,a_1)=(V_{1,1}\times V_{2,1},(x_1,y_1))$ as the possible choice of Player $\alpha$ in the $\mathcal{J}_{\!p}(X\times Y)$-game to answer Player $\beta$'s 1st move $U_1$. Let $U_2=t(X\times Y;V_1,a_1)$, $\theta(Y;V_{2,1},y_1)=W_2\in\mathcal{Y}[U_2]$; then, write $(V_{1,2},x_2)=s(W_1^*,W_2^*)$; and for all $(V_{2,2},y_2)\in\mathscr{O}(W_2)\times Y$, as the possible choice of Player $\alpha$ at the 2nd stroke in the $\mathcal{J}_{\!p}(Y)$-game, write $(V_2,a_2)=(V_{1,2}\times V_{2,2},(x_2,y_2))$ as Player $\alpha$'s possible answer to Player $\beta$'s 2nd move $U_2$. Let $U_3=t(X\times Y; V_1,a_1; V_2,a_2)$, $\theta(Y;V_{2,1},y_1; V_{2,2},y_2)=W_3\in\mathcal{Y}[U_3]$; then, write $(V_{1,3},x_3)=s(W_1^*,W_2^*,W_3^*)$; and for all $(V_{2,3},y_3)\in\mathscr{O}(W_3)\times Y$, as the possible choice of Player $\alpha$ at the 3rd stroke in the $\mathcal{J}_{\!p}(Y)$-game, write $(V_3,a_3)=(V_{1,3}\times V_{2,3},(x_3,y_3))$. Continue this procedure indefinitely, we have introduced a strategy $\theta$ for Player $\beta$ in the $\mathcal{J}_{\!p}(Y)$-game based on strategies $t$ and $s$.

Since $Y$ is $\sigma$-well ($\tau$-well) $\beta$-d\'{e}favorable of $\mathcal{J}_{\!p}$, $\theta$ is not a $\sigma$-well ($\tau$-well) winning strategy for Player $\beta$ in the $\mathcal{J}_{\!p}(Y)$-game. Thus, Player $\alpha$ has a counterforce $\{(V_{2,n},y_n)\}_{n=1}^\infty$ against $\theta$ such that Player $\alpha$ wins $\sigma$-well ($\tau$-well) the $\theta$-play $\{(W_n;V_{2,n},y_n)\}_{n=1}^\infty$. This implies that Player $\alpha$ wins $\sigma$-well the $t$-play $\{(U_n;V_n,a_n)\}_{n=1}^\infty$ in the $\mathcal{J}_{\!p}(X\times Y)$-game, where $(V_n,a_n)=(V_{1,n}\times V_{2,n},(x_n,y_n))$ and $(V_{1,n},x_n)=s(W_1^*,\dotsc,W_n^*)$. Therefore, $X\times Y$ is $\sigma$-well $\beta$-d\'{e}favorable of $\mathcal{J}_{\!p}$.
\end{proof}

\begin{thm}\label{3.19}
If $X$ is a Choquet pseudo-metric space, then it is $\tau$-well $\alpha$-favorable of $\mathcal{J}_{\!p}$. In particular, if $X_\alpha$, $\alpha\in A$, is a family of Choquet pseudo-metric spaces, then $\prod_{\alpha\in A}X_\alpha$ is $\tau$-well $\alpha$-favorable of $\mathcal{J}_{\!p}$ and so it is a Choquet $\mathcal{N}$-space.
\end{thm}  
 
\begin{proof}
    Let $\eta^\prime$ be a winning strategy for Player $\alpha$ in the BM($X$)-game. We will define a $\tau$-well winning strategy $\eta$ for Player $\alpha$ in the $\mathcal{J}_{\!p}(X)$-game as follows: 
    If Player $\beta$ begins with $U_1\in\mathscr{O}(X)$, then let $V_1^\prime=\eta^\prime(U_1)$ and we can choose $a_1\in V_1\subseteq V_1^\prime$ with $\rho\textrm{-diam}(V_1)<1/2$. Now define $\eta(U_1)=(V_1,a_1)$. If $U_2\in\mathscr{O}(V_1)$ is the possible choice of Player $\beta$ at the 2nd stroke, then set $V_2^\prime=\eta^\prime(U_1,U_2)$ and we can choose $a_2\in V_2\subseteq V_2^\prime$ with $\rho\textrm{-diam}(V_2)<1/2^2$. Repeating this procedure indefinitely, it gives rise to a strategy $\eta$ for Player $\alpha$ in the $\mathcal{J}_{\!p}(X)$ accompanied by $\eta^\prime$ such that: if $\{(U_n;V_n,a_n)\}_{n=1}^\infty$ is an $\eta$-play, then it accompanied by an $\eta^\prime$-play $\{(U_n,V_n^\prime)\}_{n=1}^\infty$ with $U_{n+1}\subseteq V_n\subseteq V_n^\prime$ for all $n\in\mathbb{N}$. Since $x\in\bigcap_nU_n\not=\emptyset$, $a_n\to x$ for $\rho(x,a_n)<1/2^n$. Thus, for every subnet $\{a_{n(i)}\}$ of $\{a_n\}$ we have that $a_{n(i)}\to x\in\bigcap_nU_n$. Then $\eta$ is a $\tau$-well winning strategy for Player $\alpha$ in the $\mathcal{J}_{\!p}(X)$-game as in Definition~\ref{3.17}A.
\end{proof}

\begin{cor}\label{3.20}
If $X$ is $\sigma$-well $\beta$-d\'{e}favorable of $\mathcal{J}_{\!p}$ and $Y$ a Choquet $\Pi$-pseudo-metric space, then $X\times Y$ is an $\mathcal{N}$-space.
\end{cor}

\begin{thm}[{cf.~Saint-Raymond (1983) \cite[Thm.~7]{SR83} in setting:  $X$ a metric space}]\label{3.21}
If $X$ is a $\Pi$-pseudo-metric space, then $X$ is Baire, iff it is $\beta$-d\'{e}favorable of $\mathcal{J}_{\!p}$, iff it is $\sigma$-well $\beta$-d\'{e}favorable of $\mathcal{J}_{\!p}$, iff it is $\tau$-well $\beta$-d\'{e}favorable of $\mathcal{J}_{\!p}$.
\end{thm}

\begin{proof}
In view of Lemma~\ref{3.3}, Theorem~\ref{2.2} and \ref{3.17}C, we need only prove that if $X$ is a Baire space, then it is $\tau$-well $\beta$-d\'{e}favorable of $\mathcal{J}_{\!p}$. Let $X=\prod_{i\in I}X_i$, where each factor $(X_i,\rho_i)$ is a pseudo-metric space. Let $\eta$ be any strategy for Player $\beta$ in the $\mathcal{J}_{\!p}(X)$-game. We need prove that $\eta$ is not a $\tau$-well winning strategy for Player $\beta$ in the $\mathcal{J}_{\!p}(X)$-game. 

For that, we will introduce an auxiliary strategy $\eta^\prime$ for Player $\beta$ in the BM($X$)-game as follows: Let $\eta^\prime(X)=U_1^\prime\subseteq U_1=\eta(X)$.
For all $V_1\in\mathscr{O}(U_1^\prime)$ and all $a_1=(a_{1,i})_{i\in I}\in V_1$, simply write $U_2=\eta(X;V_1,a_1)$ and then define $\eta^\prime(X, V_1)=U_2^\prime\in\mathscr{O}(U_2)$ such that
$U_2^\prime=U_2^{\prime\prime}\times\prod_{i\in I\setminus I_1}X_i$, where $I_1\subset I$ is some finite set, $U_2^{\prime\prime}\in\mathscr{O}(\prod_{i\in I_1}X_i)$ with $\rho_{I_1}^{}\textrm{-diam}(U_2^{\prime\prime})<1/2$ (here $\rho_{I_1}^{}$ is the pseudo-metric on $\prod_{i\in I_1}X_i$ induced naturally by $\{\rho_i\colon i\in I_1\}$). Select arbitrarily $V_2\in\mathscr{O}(U_2^\prime)$ and then choose a point $a_2=(a_{2,i})_{i\in I}\in V_2$ such that $a_{2,i}=a_{1,i}\ \forall i\in I\setminus I_1^\prime$ for some finite set $I_1^\prime\subset I$ with $I_1\subseteq I_1^\prime$.
Write $U_3=\eta(X;V_1,a_1; V_2,a_2)$; define $\eta^\prime(X, V_1, V_2)=U_3^\prime\in\mathscr{O}(U_3)$ such that
$U_3^\prime=U_3^{\prime\prime}\times\prod_{i\in I\setminus I_2}X_i$, where $I_2\subset I$ is some finite with $I_1^\prime\subseteq I_2$, $U_3^{\prime\prime}\in\mathscr{O}(\prod_{i\in I_2}X_i)$ with $\rho_{I_2}^{}\textrm{-diam}(U_3^{\prime\prime})<1/2^2$.
Select arbitrarily $V_3\in\mathscr{O}(U_3^\prime)$ and then choose a point $a_3=(a_{3,i})_{i\in I}\in V_3$ such that $a_{3,i}=a_{1,i}\ \forall i\in I\setminus I_2^\prime$ for some finite set $I_2^\prime$ with $I_2\subseteq I_2^\prime\subset I$. Continue this procedure indefinitely, we can define a strategy $\eta^\prime$ for Player $\beta$ in the BM($X$)-game and a sequence $I_1\subseteq I_1^\prime\subseteq I_2\subseteq I_2^\prime\subseteq I_3\subseteq I_3^\prime\subseteq\dotsm$ of non-void finite subsets of $I$, a sequence $\{(U_n;V_n,a_n)\}_{n=1}^\infty$---an $\eta$-play of $\mathcal{J}_{\!p}(X)$-type with $a_n\in V_n$, and a sequence $\{(U_n^\prime,V_n)\}_{n=1}^\infty$---an $\eta^\prime$-play of BM($X$), such that $a_n=(a_{n,i})_{i\in I}\in U_n^\prime=U_n^{\prime\prime}\times\prod_{i\in I\setminus I_n}X_i\subseteq U_n$ with $a_{n,i}=a_{1,i}\ \forall i\in I\setminus I_n^\prime$ and $\rho_{I_n}^{}\textrm{-diam}(U_n^{\prime\prime})<1/2^{n-1}$ for all $n\ge2$. 

By Theorem~\ref{2.2}, $\eta^\prime$ is not a winning strategy for Player $\beta$ in the BM($X$)-game; and so, there is a choice $\{V_n\}_{n=1}^\infty$ against $\eta^\prime$ for Player $\alpha$ such that $\bigcap_{n=1}^\infty U_n^\prime\subseteq\bigcap_{n=1}^\infty U_n\not=\emptyset$. Thus, for any point $x=(x_i)_{i\in I}\in\bigcap_{n=1}^\infty U_n$ with $x_i=a_{1,i}\ \forall i\in I\setminus(\bigcup_{n=1}^\infty I_n)$, by $\rho_i(a_{n+1,i},x_i)\le 1/2^n\ \forall i\in I$, it follows that $a_{n,i}\to x_i$ in $(X_i,\rho_i)$ as $n\to\infty$. Thus, $a_n\to x$ as $n\to\infty$. This shows that $\eta$ is not a $\tau$-well winning strategy for Player $\beta$ in the $\mathcal{J}_{\!p}(X)$-game. 
The proof is complete.
\end{proof}

\section{Countable tightness, rich family and hereditarily Baire spaces}\label{s4}
This section will be devoted to proving Theorem~\ref{1.3}-(4) stated in $\S$\ref{s1} under the guise of Theorem~\ref{4.1.8}$^\prime$, and extending another theorem of \cite{LM08} (Thm.~\ref{4.1.6}). Finally a theorem of Hurewicz (1928) will be extended here (Cor.~\ref{4.2.6}).
\subsection{Countable tightness and rich family}\label{s4.1}
We begin with recalling two concepts---countable tightness and rich family for a topological space, needed in our later discussion.

\begin{sse}[Countable tightness]\label{4.1.1}
We say that a space $X$ has \textit{countable tightness} or $x$ is \textit{countably tight}\,(cf. \cite[Def.~13.4.1]{W70} or \cite{G76, E89}), if for each subset $A$ of $X$ and each point $p\in\overline{A}$, there exists a countable subset $C\subseteq A$ such that $p\in\overline{C}$. Note that countable tightness is hereditary to any subspace; however, the finite product of countably tight spaces may fail to have countable tightness.

If $X$ is a compact space and $Z$ a metric space, then $C_p(X,Z)$ has countable tightness~\cite[Thm.~13.4.1]{W70};
the one-point compactification $X^*$ \cite{K55} of a discrete space $X$ has countable tightness; and every first countable space is of course countably tight. 

However, a compact Hausdorff space is not necessarily countably tight (cf.~Ex.~\ref{7.5}). See Theorem~\ref{5.2.7}-(1) for a sufficient condition of countable tightness.
\end{sse}

\begin{sse}[Rich family]\label{4.1.2}
Let $X$ be a space, $\mathcal{S}_\textrm{cl}(X)$ the family of non-void, closed, separable subspaces of $X$. Then a subfamily $\mathcal{F}$ of $\mathcal{S}_\textrm{cl}(X)$ is called a \textit{rich family} for $X$\,\cite[$\S$3]{LM08} if for every $A\in\mathcal{S}_\textrm{cl}(X)$ there exists an $F\in\mathcal{F}$ such that $A\subseteq F$ (i.e., $\mathcal{S}_\textrm{cl}(X)\preceq\mathcal{F}$), and
$\overline{\bigcup_{n\in\mathbb{N}}F_n}\in\mathcal{F}$ for every increasing sequence $\{F_n\}_{n=1}^\infty$ in $\mathcal{F}$.
Clearly, $\mathcal{S}_\textrm{cl}(X)$ is the greatest element in the collection of all rich families for $X$ under the binary relation of set inclusion.
\end{sse}

\begin{slem}[{cf.~\cite[Prop.~3.2]{LM08}}]\label{4.1.3}
Let $X$ be a space having countable tightness and $E$ a dense subset of $X$. Then
\begin{equation*}
\mathcal{F}[E]:=\{F\in\mathcal{S}_\textrm{cl}(X)\,|\,F\cap E\textrm{ is dense in }F\}=\{F\in\mathcal{S}_\textrm{cl}(X)\,|\,\exists\{a_n\in E\colon n\in\mathbb{N}\}\textrm{ dense in }F\}
\end{equation*}
is a rich family for $X$.
\end{slem}

\begin{proof}
By the density of $E$ and countable tightness of $X$, it is easy to verify that $\mathcal{S}_\textrm{cl}(X)\preceq\mathcal{F}[E]$. Clearly, $\mathcal{F}[E]$ is closed under the closure of countable union of members of $\mathcal{F}[E]$. Thus, $\mathcal{F}[E]$ is a rich family for $X$.
\end{proof}

\begin{slem}[{cf.~\cite[Prop.~1.1]{BM00} or \cite[Prop.~3.1]{LM08}}]\label{4.1.4}
Let $\{\mathcal{F}_n\,|\,n\in\mathbb{N}\}$ be a sequence of rich families for a space $X$, then $\bigcap_{n\in\mathbb{N}}\mathcal{F}_n$ is a rich family for $X$.
\end{slem}

\begin{proof}
It is enough to prove that for any $A\in\mathcal{S}_\textrm{cl}(X)$, there exists a member $F\in\bigcap_{n\in\mathbb{N}}\mathcal{F}_n$ with $A\subseteq F$. Indeed, first choose $F_{1,1}\in\mathcal{F}_1$ with $A\subseteq F_{1,1}$; and then choose $F_{2,1}\in\mathcal{F}_2$ and $F_{1,2}\in\mathcal{F}_1$ with $F_{1,1}\subseteq F_{2,1}\subseteq F_{1,2}$. Next, choose $F_{3,1}\in\mathcal{F}_3$, $F_{2,2}\in\mathcal{F}_2$ and $F_{1,3}\in\mathcal{F}_1$ with $F_{1,2}\subseteq F_{3,1}\subseteq F_{2,2}\subseteq F_{1,3}$. Repeating this procedure indefinitely, one can choose sequences $\{(F_{n,j})_{j=1}^\infty\}_{n\in\mathbb{N}}$ with $(F_{n,j})_{j=1}^\infty\subseteq\mathcal{F}_n$ for all $n\in\mathbb{N}$, such that $A\subseteq\overline{\bigcup_{j=1}^\infty F_{1,j}}=\overline{\bigcup_{j=1}^\infty F_{2,j}}=\dotsm=\overline{\bigcup_{j=1}^\infty F_{n,j}}=\dotsm\in\bigcap_{n=1}^\infty\mathcal{F}_n$. The proof is complete.
\end{proof}

We then have a criterion for the Baire space connecting countable tightness and the rich family of Baire subspaces.

\begin{sthm}[{cf.~\cite[Thm.~3.3]{LM08}}]\label{4.1.5}
If $X$ is a countably tight Hausdorff space that possesses a rich family of Baire subspaces, then $X$ is a Baire space.
\end{sthm}

Note that a space that has a non-meager subspace is not necessarily non-meager itself. For instance, any singleton subspace is Baire and so non-meager itself. Next we shall first generalize Theorem~\ref{4.1.5} to give us a sufficient condition for the non-meagerness connecting countable tightness and the rich family of non-meager subspaces.

\begin{sthm}\label{4.1.6}
If $X$ is a countably tight space that possesses a rich family of non-meager subspaces, then $X$ is non-meager in itself.
\end{sthm}

\begin{proof}
Let $\mathcal{F}$ be a rich family of non-meager subspaces for $X$.
Let $\{U_n\,|\,n\in\mathbb{N}\}$ be a sequence of open dense subsets of $X$. Given $n\in\mathbb{N}$, define $\mathcal{F}_n=\mathcal{F}[U_n]$ as in Lemma~\ref{4.1.3} with $E=U_n$.
Then $\mathcal{F}_n$, for each $n\in\mathbb{N}$, is a rich family for $X$. Let $\mathcal{F}^*=\bigcap_{n\in\mathbb{N}}(\mathcal{F}_n\cap\mathcal{F})$. Then $\mathcal{F}^*$ is a rich family for $X$ by Lemma~\ref{4.1.4}.
Let $F\in\mathcal{F}^*$. Since $\mathcal{F}^*\subseteq\mathcal{F}$, $F$ is non-meager itself. As $F\in\mathcal{F}_n$, it follows that $U_n\cap F$ is relatively open dense in $F$ for all $n\in\mathbb{N}$. Thus, $\bigcap_{n\in\mathbb{N}}(U_n\cap F)$ is a residual subset of $F$ so that
$\emptyset\not={\bigcap}_{n\in\mathbb{N}}(U_n\cap F)\subseteq{\bigcap}_{n\in\mathbb{N}}U_n$ and $X$ is non-meager in itself.
\end{proof}

\begin{scor}\label{4.1.7}
Let $X$ be a homogeneous space (for example, a right/left-topological group) with countable tightness and having a rich family of non-meager subspaces. Then $X$ is a Baire space.
\end{scor}

If $\mathcal{F}$ is a rich family of non-meager subspaces for $X$, then by Banach's category theorem we can find a family $\mathcal{F}^\prime$ of closed separable Baire subspaces of $X$. But here we cannot assert that $\mathcal{F}^\prime$ is a rich family for $X^\prime=\bigcup\{F^\prime\colon F^\prime\in\mathcal{F}^\prime\}$; and moreover, then non-meagerness of $X^\prime$ does not imply the non-meagerness of $X$. So Theorem~\ref{4.1.5} $\not\Rightarrow$ Theorem~\ref{4.1.6}. However, based on Theorem~\ref{4.1.6}, we can restate Theorem~\ref{4.1.5} and give another proof as follows, in which Step 2 is of interest in itself.

\begin{4.1.5'}
If $X$ is a countably tight space that possesses a rich family of Baire subspaces, then $X$ is a Baire space.
\end{4.1.5'}

\begin{proof}
We shall divide our proof into three steps.
\item Step~1. Countable tightness of $X$ is hereditary to subsets of $X$.

\item Step~2. Let $\mathcal{F}$ be a rich family of Baire subspaces for $X$; then $\mathcal{F}\!\upharpoonright_G=\{F\cap G\colon F\in\mathcal{F}\}$, for all $G\in\mathscr{O}(X)$, is a rich family of Baire subspaces for $G$. 

Indeed, $F\cap G$ is Baire for all $F\in\mathcal{F}$. Next, we need verify that $\mathcal{F}\!\upharpoonright_G$ is a rich family for $G$. In fact, if $F_1\subseteq F_2\subseteq F_3\subseteq\dotsm$ in $\mathcal{F}$, then $\overline{\bigcup_{n=1}^\infty(F_n\cap G)}^{\,G}=\overline{(\bigcup_{n=1}^\infty F_n)\cap G}^{\,G}=\overline{\bigcup_{n=1}^\infty F_n}\cap G\in\mathcal{F}\!\upharpoonright_G$. Moreover, if $A\in\mathcal{S}_{\textrm{cl}}(G)$, then there exists a member $F\in\mathcal{F}$ such that $A\subseteq F$. So $A\subseteq F\cap G\in\mathcal{F}\!\upharpoonright_G$. \quad
\item Step~3. By Theorem~\ref{4.1.6}, every $G\in\mathscr{O}(X)$ is non-meager in $X$ so that $X$ is Baire. The proof is complete.
\end{proof}

\begin{sthm}[{cf.~\cite[Thm.~4.7]{LM08}}]\label{4.1.8}
Suppose that $X$ is a countably tight Hausdorff space that possesses a rich family of Baire subspaces. Then $X$ is an $\mathcal{N}$-space.
\end{sthm}

Using Theorem~\ref{2.5} and slightly modifying the proof of Lin-Moors (2008) \cite[Thm.~4.7]{LM08}, we can slightly modify Theorem~\ref{4.1.8} by removing condition ``Hausdorff'' on $X$ as follows:

\begin{4.1.8'}
Let $X$ be a space having countable tightness and a rich family of Baire subspaces. Let $f\colon X\times Y\rightarrow Z$ be a separately continuous mapping, where $Y\times Y$ is countably compact and $Z$ a pseudo-metric space. Then there exists a dense set $J\subseteq X$ such that $f$ is jointly continuous at each point of $J\times Y$. (So $X$ is an $\mathcal{N}$-space.)
\end{4.1.8'}

\begin{proof}
Considering members of $\mathscr{O}(X)$ if necessary, it suffices to prove that there exists a point $x\in X$ such that  $f$ is jointly continuous at each point of $\{x\}\times Y$. For that, suppose to the contrary that there exists no point $x\in X$ such that $f$ is jointly continuous at each point of $\{x\}\times Y$. Let $\rho$ be a pseudo-metric for $Z$.
Let $\mathcal{F}$ be a rich family of Baire subspaces for $X$. Firstly
for all $n\in\mathbb{N}$, define a set
\begin{enumerate}
    \item[] $E_n=\left\{x\in X\,|\,\exists y\in Y\textrm{ s.t. }\rho\textrm{-diam}(f(U\times V))>\frac1n \ \forall (U,V)\in \mathfrak{N}_{x}(X)\times \mathfrak{N}_{y}(Y)\right\}$.
\end{enumerate}
Then $X=\bigcup_{n=1}^\infty E_n$, $E_n\subseteq E_{n+1}$. Since $X$ is Baire by Theorem~\ref{4.1.5}$^\prime$, $X=\overline{\bigcup_{n=1}^\infty \textrm{int}_X\overline{E}_n}$ and there exists some $k_0\in \mathbb{N}$ such that
$\text{int}_X\overline{E}_k\not=\emptyset$ for all $k\ge k_0$.

For all $k\ge k_0$ and each $x\in X$, let
$X_k[x]=\left\{x^\prime \in X\colon \|f_x-f_{x^\prime}\|>1/3k\right\}$,
where $\|\cdot\|$ is the sup-norm in $C(Y,Z)$. Then $x\notin X_k[x]$ but $x\in \overline{X_k[x]}$ for each $x\in\textrm{int}_XE_k$. Moreover, since $X$ has countable tightness, there exists for each $x\in E_k\cap\textrm{int}_X\overline{E}_k$ a countable set $C_k[x]\subseteq X_k[x]\cap E_k\cap\textrm{int}_X\overline{E}_k$ with $x\in \overline{C_k[x]}$. 
Next, for all $k\ge k_0$ we can inductively define an increasing sequence $\{F_{k,n}\}_{n\in\mathbb{N}}$ in $\mathcal{F}$ such that $F_{k,1}\cap\text{int}_X\overline{E}_k\not=\emptyset$ and $\bigcup\{C_k[x]\,|\,x\in D_{k,n}\cap E_k\cap\text{int}_X\overline{E}_k\}\cup F_{k,n}\subseteq F_{k,n+1}$ for all $ n\in\mathbb{N} $, where $ D_{k,n}$ is any countable dense subset of $F_{k,n}$.
Let $F_k=\overline{\bigcup_{n\in\mathbb{N}}F_{k,n}}$ and $D_k=\bigcup_{n\in\mathbb{N}}D_{k,n}$. Then $\overline{D}_k=F_k\in\mathcal{F}$ for $\mathcal{F}$ is a rich family for $X$; and moreover, $\|\cdot\|\textrm{-diam}(\{f_x\,|\,x\in U\})\geq 1/3k$ for every $U\in\mathscr{O}(F_k\cap \text{int}_X\overline{E}_k)$.

Note that $F_k\cap \text{int}_X\overline{E}_k$ is a separable Baire space.
However, there is no point $x\in F_k\cap \text{int}_X\overline{E}_k$ such that $f\!\upharpoonright_{(F_k\cap \text{int}_X\overline{E}_k)\times Y}\colon (F_k\cap \text{int}_X\overline{E}_k)\times Y\rightarrow Z$ is jointly continuous at each point of $\{x\}\times Y$, contrary to Theorem~\ref{2.5}. The proof is complete.
\end{proof}

%

We need to note that a countably tight space that only contains a separable non-meager subspace need not be a g$\mathcal{N}$-space. For instance, a first countable $T_1$-space is not necessarily g$\mathcal{N}$, but it always contains separable non-meager subspaces.

\begin{srem}\label{4.1.9}
Here ``$\mathcal{F}$ being a rich family of non-meager subspaces for $X$'' $\not\Rightarrow$ ``$\mathcal{F}\!\upharpoonright_G$ being a rich family of non-meager subspaces for $U$'', for all $G\in\mathscr{O}(X)$ with $G\not=X$.
\end{srem}

\begin{srem}\label{4.1.10}
Comparing with Theorems~\ref{4.1.6} and \ref{4.1.8}$^\prime$, we naturally expect the following statement which implies Theorem~\ref{4.1.8}$^\prime$:
\textit{$X$ is a g$\mathcal{N}$-space if it has countable tightness and possesses a rich family of non-meager subspaces (?)}.
\end{srem}

\begin{sthm}\label{4.1.11}
Let $G$ be a right-topological group, which has countable tightness and a rich family of Baire subspaces. Let $X\times X$ be countably compact completely regular. If $G\curvearrowright_\pi X$ is separately continuous, then $G\curvearrowright_\pi X$ is a topological flow and $G$ is Baire.
\end{sthm}

\begin{proof}
First, by Corollary~\ref{4.1.7}, $G$ is a Baire space. Let $\rho$ be any uniformly continuous pseudo-metric for $X$ and $X_\rho=(X,\rho)$. Let
$f=\textit{id}_X\circ\pi\colon G\times X\xrightarrow{\pi}X\xrightarrow{\textit{id}_X}X_\rho$,
which is separately continuous. Then by Theorem~\ref{4.1.8}$^\prime$, there exists an element $g\in G$ such that $f$ is jointly continuous at each point of $\{g\}\times X$. Now, for nets $t_i\to t$ in $G$ and $x_i\to x$ in $X$, we have that $t_it^{-1}g\to g$ in $G$ and $g^{-1}tx_i\to g^{-1}tx$ in $X$. Thus, 
$t_ix_i=(t_it^{-1}g)(g^{-1}tx_i)=f(t_it^{-1}g, g^{-1}tx_i)\to f(g,g^{-1}tx)=tx$
in $X_\rho$. This shows that $f$ is jointly continuous. Since $\rho$ is arbitrary and the topology for $X$ is determined by all such $\rho$, $\pi\colon G\times X\rightarrow X$ is jointly continuous. The proof is complete.
\end{proof}

\subsection{Hereditarily Baire space}\label{s4.2}
We begin with recalling that a subset of a topological space $X$ is called a \textit{perfect set}, if it is non-void, closed, and without isolated points as a subspace of $X$.

\begin{sse}[Hereditarily Baire space]\label{4.2.1}
A space $X$ is \textit{hereditarily Baire} if all closed non-void subsets of $X$ are Baire spaces. For example, a locally compact regular space is hereditarily Baire; the space $Y$ is Example~\ref{3.1}E is compact, $T_1$, first countable, and Baire, but not hereditarily Baire.

\begin{4.2.1A} 
If $X$ is hereditarily Baire and $T_1$, then all perfect sets in $X$ are uncountable.
\end{4.2.1A}

\begin{4.2.1B}[Hurewicz (1928) \cite{H28}]
A metric space $X$ is hereditarily Baire if and only if all perfect sets in $X$ are uncountable.
\end{4.2.1B}

\item\textbf{\ref{4.2.1}C.} If $X$ is hereditarily Baire, then $\mathcal{S}_\textrm{cl}(X)$ is a rich family of Baire subspaces for $X$; and each $U\in\mathscr{O}(X)$ with $U\not=X$ and $X\setminus U$ are hereditarily Baire.
\end{sse}

\begin{sthm}[{cf.~\cite[Thm.~1.1]{CP05}}]\label{4.2.2}
Let $X_i$, $i\in I$, be metrizable hereditarily Baire spaces. Then $\prod_{i\in I}X_i$ is Baire; and moreover, it has the $\mathcal{N}$-property.
\end{sthm}

We shall reprove and slightly improve Theorem~\ref{4.2.2} in $\S$\ref{s5} using approaches different with Chaber-Pol 2005~\cite{CP05} (Thm.~\ref{5.3.7}).

\begin{sse}[Hereditarily non-meager space]\label{4.2.3}
Naturally, we say that $X$ is \textit{hereditarily non-meager}, if all closed non-void subsets of $X$ are non-meager in themselves. In that case, $X$ has a rich family of non-meager subspaces; and moreover, if $U\in\mathscr{O}(X)$ is dense in $X$ and $U\not=X$, then $F=X\setminus U$ is a meager subset in $X$, but $F$ is a non-meager subspace.
\end{sse}

However, `hereditarily Baire' coincides with `hereditarily non-meager' from the following simple observation.

\begin{slem}\label{4.2.4}
A topological space is hereditarily Baire if and only if it is hereditarily non-meager.
\end{slem}

\begin{proof}
Since a Baire space must be non-meager, hence necessity is obvious. Now conversely, assume $X$ is hereditarily non-meager. To prove that
$X$ is hereditarily Baire, it is enough to prove that $X$ is Baire. However, for that, we need only prove that every $U\in\mathscr{O}(X)$ is non-meager in $X$. Indeed, for all $U\in\mathscr{O}(X)$, since $\overline{U}$ is a non-meager space and $\overline{U}=U\cup(\overline{U}\setminus U)$ such that $\overline{U}\setminus U$ is meager in $\overline{U}$, it follows that $U$ is non-meager in $\overline{U}$. Thus, $U$ is a non-meager space; and so, $U$ of non-meager in $X$.
The proof is complete.
\end{proof}

Therefore, Theorem~\ref{4.2.2} (\cite[Thm.~1.1]{CP05}) can be stated as follows: The product of metrizable hereditarily non-meager spaces is a Baire $\mathcal{N}$-space.

Hurewicz's theorem \cite{H28}---Theorem~\ref{4.2.1}B had been extended as follows: \textit{If a meager space is embeddable in $C_p(K)$ for some compact Hausdorff space $K$, then $X$ contains a countable perfect set} (cf.~\cite[Prop.~6.1]{CP05}). Here we can generalize Hurewicz's theorem as follows:

\begin{sthm}\label{4.2.5}
Let $X$ be a regular, first countable, $T_1$-space. If $X$ is meager, then $X$ contains a countable perfect set that is homeomorphic to the space $\mathbb{Q}$ of rational numbers.
\end{sthm}

\begin{proof}
Let $X=\bigcup_{n=1}^\infty F_n$, where $F_n$, for each $n\in\mathbb{N}$, is a closed nowhere dense set in $X$.
Since $X$ is first countable and Hausdorff, it follows that for all $x\in X$, we can choose $V_n(x)\in\mathfrak{N}_x^\textrm{o}(X)$, for each $n\in\mathbb{N}$, satisfying $\bigcap_nV_n(x)=\{x\}$ and $V_1(x)\supseteq V_2(x)\supseteq\dotsm$. We shall inductively define finite sets $A_1\subset A_2\subset A_3\subset\dotsm$ in $X$ and $U_n(x)\in\mathfrak{N}_x^\textrm{o}(X)$, for each $x\in A_n$, such that:
\begin{enumerate}
\item[(1)] $U_n(x)\cap U_n(y)=\emptyset\ \forall x,y\in A_n$ with $x\not=y$,
\item[(2)] $U_n(x)\subseteq V_n(x)\ \forall x\in A_n$,
\end{enumerate}
and setting
\begin{enumerate}
\item[(3)] $\mathcal{U}_n=\{U_n(x)\colon x\in A_n\}$ and $\overline{\mathcal{U}}_n=\{\overline{U_n(x)}\colon x\in A_n\}$,
\end{enumerate}
we have
\begin{enumerate}
\item[(4)] $\overline{\mathcal{U}}_{n+1}\preceq\mathcal{U}_n$.
\end{enumerate}
For that, we start with $A_1=\{x\}$ and $\mathcal{U}_1=\{U_1(x)=V_1(x)\}$, where $x\in X$ is arbitrarily given. Assume that $A_n$ and $\mathcal{U}_n$ are defined. Since $\{x\}$ and $F_1, \dotsc, F_n$ are closed nowhere dense in $X$, we can choose, for each $x\in A_n$, a point $b_n(x)\in U_n(x)\setminus(\{x\}\cup F_1\cup\dotsm\cup F_n)$. Then we put
\begin{enumerate}
\item[(5)] $A_{n+1}=A_n\cup\{b_n(x)\colon x\in A_n\}$;
\end{enumerate}
and we end the inductive step by choosing $U_{n+1}(x)$, for each $x\in A_{n+1}$, so that conditions (1) $\thicksim$ (4) are satisfied together with the condition
\begin{enumerate}
\item[(6)] $U_{n+1}(x)\cap(F_1\cup\dotsm\cup F_n)=\emptyset\ \forall x\in A_{n+1}\setminus A_n$.
\end{enumerate}
Now let
$A={\bigcap}_{n=1}^\infty\bigcup\mathcal{U}_n={\bigcap}_{n=1}^\infty\bigcup\overline{\mathcal{U}}_n$.
It is obvious that $A$ is closed with $\bigcup_{n=1}^\infty A_n\subseteq A$. On the other hand, for every $y\in A$, we have that $y\in U_1(x_1)\cap U_2(x_2)\cap\dotsm\cap U_n(x_n)\cap\dotsm$, where $x_n\in A_n$. We can fix an $m\in\mathbb{N}$ such that $y\in F_m$. By (6), $x_n\in A_m$ for all $n\ge m$, which implies by (1) that $x_m=x_{m+1}=\dotsm=x\in A_m$. Then by $\bigcap_nV_n(x)=\{x\}$ and (2), it follows that $y=x$. Thus, $A=\bigcup_{n=1}^\infty A_n$. This implies that $A$ is a countable perfect set in $X$. Since $A$ is regular, $T_1$ and countable, it is metrizable. Thus, $A$ is homeomorphic to $\mathbb{Q}$.
The proof is complete.
\end{proof}

Now by Lemma~\ref{4.2.4} and Theorem~\ref{4.2.5}, we can provide a characterization of the regular first countable hereditarily Baire $T_1$-spaces, which contains Hurewicz's theorem ($\because$ a metric space is always a regular first countable $T_1$-space; see Theorem~\ref{5.2.9} for a more general extension).

\begin{scor}\label{4.2.6}
If $X$ is a regular first countable $T_1$-space, then $X$ is hereditarily Baire if and only if all perfect sets in $X$ are uncountable.
\end{scor}

Consequently, any perfect subset of a locally countably compact metric space is uncountable by Lemma~\ref{3.2}.
\section{Cartesian product and $\Sigma$-product of non-meager spaces}\label{s5}
In this section we shall further extend Hurewicz's theorem (Thm.~\ref{4.2.1}B) based on Theorem~\ref{4.2.5} and the concept of $W$-space of G (Def.~\ref{5.1.1} and Thm.~\ref{5.2.9}).
In addition, we will prove Theorems~\ref{1.3}-(5) and \ref{1.3}-(6) (Thm.~\ref{5.4.4} and Thm.~\ref{5.4.5}). 

\subsection{$\mathscr{G}_x(X)$-game}\label{s5.1}
Let $X$ be a topological space and $x\in X$. By a \textit{$\mathscr{G}_x(X)$-play} played by Player $\beta$ and Player $\alpha$, it means a sequence $\{(x_i,W_i)\}_{i=1}^\infty$ with $x_1=x$, $W_i\in\mathfrak{N}_{x}^\textrm{o}(X)$ and $x_{i+1}\in W_i$ for all $i\in\mathbb{N}$, where $x_i$ and $W_i$ are picked alternately by Player $\beta$ and Player $\alpha$, and moreover, Player $\beta$ is always granted the privilege of the first move here. Our winning condition is defined as follows:
\begin{enumerate}[\textbf{W.C.}:]
    \item Player $\alpha$ \textit{wins} the $\mathscr{G}_x(X)$-play $\{(x_i,W_i)\}_{i=1}^\infty$, if $x$ is a cluster point of $\{x_i\}_{i=1}^\infty$. Otherwise, Player $\beta$ \textit{wins} this play.    
\end{enumerate}
\begin{sse}[{$W$-spaces of Gruenhage 1976~\cite{G76}}]\label{5.1.1}
The point $x\in X$ is called a \textit{$W$-point of G}, if Player $\alpha$ has a winning strategy $\sigma$ in the $\mathscr{G}_x(X)$-game; that is, if $\{(x_i,W_i)\}_{i=1}^\infty$ is a $\sigma$-play in the $\mathscr{G}_x(X)$-game, then Player $\alpha$ wins this play. If every point $x$ of $X$ is a $W$-point of G in the $\mathscr{G}_x(X)$-game, then $X$ is called a \textit{$W$-space of G}. In other words, $X$ is a $W$-space of G if and only if it is $\alpha$-favorable of G. In addition, if the $W$-points of G are only dense in $X$, then $X$ will be called an \textit{almost $W$-space of G}.

It is readily seen that if $X$ is a $W$-space of G and $\emptyset\not=A\subset X$, then $A$ is a $W$-subspace of G (cf.~\cite[Thm.~3.1]{G76}). Note that Gruenhage's game has been generalized by requiring only that  $\{x_i\}_{i=1}^\infty$ in the $\mathscr{G}_x(X)$-play $\{(x_i,W_i)\}_{i=1}^\infty$ has a cluster point in $X$ but $x$ itself need not be a cluster point of $\{x_i\}_{i=1}^\infty$ (cf.~Bouziad 1993 \cite{B93}).

As a generalization of the first countable spaces, a first countable space is of course a $W$-space of G. However, a $W$-space of G is not necessarily first countable (see, e.g., \cite[Ex.~2.7]{LM08}).
In fact, the one-point compactification $X^*$ of a discrete space $X$ is always a $W$-space of G. Thus, if $X$ is a discrete uncountable space, then $X^*$ is a $W$-space of G; but it is not a first countable space.
\end{sse}

Since a metrizable hereditarily Baire space is a Hausdorff regular $W$-space of G with a rich family of Baire subspaces, the first part of Theorem~\ref{4.2.2} (\cite[Thm.~1.1]{CP05}) has already been improved by Lin and Moors 2008 in \cite{LM08} as follows:

\begin{sthm}[{cf.~\cite[Cor.~4.6]{LM08}}]\label{5.1.2}
Let $\{X_i\}_{i\in I}$ be a family of Hausdorff regular $W$-spaces of G, each of which possesses a rich family of Baire subspaces. Then $\prod_{i\in I}X_i$ is a Baire space.
\end{sthm}

Theorem~\ref{5.3.7} is a further improvement of Theorem~\ref{5.1.2}.
First of all, we can slightly improve the Chaber-Pol theorem~\cite[Thm.~1.1]{CP05} mentioned in $\S$\ref{s4.2} (Thm.~\ref{4.2.2}) as follows:

\begin{sthm}\label{5.1.3}
Let each $X_i$, $i\in I$, be a pseudo-metrizable hereditarily non-meager space. Then $\prod_{i\in I}X_i$ is a Baire space; and moreover, it is an $\mathcal{N}$-space.
\end{sthm}

\begin{proof}
    By Lemma~\ref{4.2.4} and Theorem~\ref{5.1.2}.
\end{proof}

\subsection{$W$-spaces, $\pi$-base, $\Sigma$-products, and Hurewicz's Theorem}\label{s5.2}
The property to be a Baire space is hereditary to open subspaces and to dense $G_\delta$-subspaces (cf.~\cite[3.9J-(a)]{E89}).
We note that if $X_0$ is a dense subset of a space $X$ such that $X_0$, as a subspace, is Baire, then $X$ is Baire itself (cf.~\cite[3.9J-(b)]{E89}). In fact, we have the following more general fact:

\begin{slem}\label{5.2.1}
If $X_0$ is a dense subset of a space $X$ such that $X_0$, as a subspace, is non-meager, then $X$ is non-meager itself.
\end{slem}

\begin{proof}
Otherwise, $X=\bigcup_{n=1}^\infty F_n$, where each $F_n$ is closed nowhere dense. So $X_0=\bigcup_{n=1}^\infty(X_0\cap F_n)$. If $V=\textrm{int}_{X_0}(X_0\cap F_n)\not=\emptyset$ for some $n\in\mathbb{N}$, then there exists $U\in\mathscr{O}(X)$ such that $V=U\cap X_0$ and $U\subseteq\overline{V}\subseteq F_n$, which is impossible.
\end{proof}

\begin{sse}[$\pi$-base]\label{5.2.2}
A family $\mathcal{B}\subseteq\mathscr{O}(X)$ is referred to as a \textit{$\pi$-base} (also known as a pseudo-base) for $X$ \cite{O60, W70, M75}, if every $U\in\mathscr{O}(X)$ contains some member of $\mathcal{B}$.
A $\pi$-base $\mathcal{B}$ is called \textit{locally countable}, if each member of $\mathcal{B}$ contains only countably many members of $\mathcal{B}$. If 
\begin{enumerate}
    \item[] $X_o[\mathcal{B}]:=\bigcup\{B\,|\,B\in\mathcal{B}\}$,
\end{enumerate} 
then $X_o$ is dense open in $X$. If a space $X$ has a locally countable $\pi$-base, then for each $x\in X_o[\mathcal{B}]$ there exists $U\in\mathfrak{N}_x^\textrm{o}(X)$ such that $U$ has a countable $\pi$-base.
\end{sse}

If a space is second countable, then it has a countable $\pi$-base; but not vice versa. For example,
$\beta\mathbb{N}$ is not a second countable space, but it has a countable $\pi$-base $\mathcal{B}=\{\{n\}\colon n\in\mathbb{N}\}$ \cite{P37}. 
If $X$ has a countable $\pi$-base, then $X$ is a separable space; but not vice versa. In fact, $X$ has a countable $\pi$-base if and only if $X$ is separable with a locally countable $\pi$-base.

\begin{slem}\label{5.2.3}
Let $X$ be a quasi-regular, separable, $W$-space of G. Then $X$ has a countable $\pi$-base.
\end{slem}

\begin{proof}
Let $D=\{x_{n}\}_{n=1}^{\infty}$ be a dense sequence of points of $X$. For each $x\in X$, let $\sigma_{\!x}$ be a winning strategy for Player $\alpha$ in the $\mathscr{G}_x(X)$-game; and for $x\in X$ define
\begin{enumerate}
    \item[] $\mathcal{E}(x)=\{\sigma_{\!x}(x_1,x_{i_1},\dotsc,x_{i_k})\in\mathscr{O}(X)\,|\,k\in\mathbb{N}\ \&\ (x_1,x_{i_1},\dotsc,x_{i_k})\in \{x\}\times D^k\textrm{ is a partial }\sigma_{\!x}\textrm{-string}\}$.
\end{enumerate}
and let $\mathcal{B}=\bigcup_{n=1}^\infty\mathcal{E}(x_n)$. Then $\mathcal{B}\subseteq\mathscr{O}(X)$ is a countable collection. Next we claim that $\mathcal{B}$ is a countable $\pi$-base for $X$. Indeed, for each $U\in\mathscr{O}(X)$, by quasi-regularity there exists a set $U_1\in\mathscr{O}(X)$ such that $U_1\subseteq\overline{U}_1\subseteq U$. Then $x_{n}\in U_1$ for some $n\in\mathbb{N}$. If $W\setminus\overline{U}_1\not=\emptyset$ for every $W\in \mathcal{E}(x_n)$, then based on $\sigma_{\!x_n}$ there is a $\mathscr{G}_{x_n}(X)$-play $\{(y_i,W_i)\}_{i=1}^\infty$ such that $y_i\notin \overline{U}_1$ for each $i\in\mathbb{N}$, contrary to $x_{n}\in\bigcap_{k\in\mathbb{N}}\overline{\{y_i\,|\,i\ge k\}}$. The proof is complete.
\end{proof}

In fact, if $X$ is a regular separable $W$-space of G then $\mathcal{E}(x)$, defined as in Proof of Lemma~\ref{5.2.3}, is a countable base at $x\in X$ so that $X$ is first countable. This also proves the following

\begin{slem}[{cf.~\cite{G76}}]\label{5.2.4}
A regular separable $W$-space of G is first countable.
\end{slem}

A separable first countable space has obviously a countable $\pi$-base. Then by Lemma~\ref{5.2.4}, it follows that every regular separable $W$-space of G has a countable $\pi$-base. So Lemma~\ref{5.2.3} may be thought of as a generalization of Lemma~\ref{5.2.4}.

\begin{srem}[{cf.~\cite[Thm.~3.9]{G76}}]\label{5.2.5}
If there is a winning strategy $\sigma$ for Player $\alpha$ in the $\mathscr{G}_y(Y)$-game, then there is a strategy $\sigma^\prime$ for Player $\alpha$ such that $y_i\to y$ as $i\to\infty$ whenever $\{(y_i,W_i)\}_{i=1}^\infty$ is a $\sigma^\prime$-play in the $\mathscr{G}_y(Y)$-game. (Hence the countable product of $W$-spaces of G is a $W$-space of G \cite[Thm.~4.1]{G76}.)
\end{srem}

\begin{proof}
Indeed, let $y_1=y$ and $\sigma^\prime(y_1)=\sigma(y_1)$ and let $\sigma^\prime(y_1,\centerdot)\colon \sigma^\prime(y_1)\rightarrow\mathfrak{N}_{y}^\textrm{o}(Y)$ be defined by
$$
\sigma^\prime(y_1,y_2)=\sigma(y_1)\cap\sigma(y_1,y_2)\quad \forall y_2\in\sigma^\prime(y_1).
$$
Next, define $\sigma^\prime(y_1,y_2,\centerdot)\colon\sigma^\prime(y_1,y_2)\rightarrow\mathfrak{N}_{y}^\textrm{o}(Y)$ by
$$
\sigma^\prime(y_1,y_2,y_3)=\sigma(y_1)\cap\sigma(y_1,y_2)\cap\sigma(y_1,y_3)\cap\sigma(y_1,y_2,y_3)\quad \forall y_3\in\sigma^\prime(y_1,y_2).
$$
If $(y_1,\dotsc,y_n)$ is a partial $\sigma^\prime$-string and $y_{n+1}\in\sigma^\prime(y_1,\dotsc,y_n)$, then
\begin{equation*}\begin{split}
\sigma^\prime(y_1,\dotsc,y_{n+1})=\sigma(y_1)\cap\left(\bigcap\{\sigma(y_{i_1},\dotsc,y_{i_k})\,|\,1=i_1<\dotsm<i_k\le n+1\ \&\ 1\le k\le n+1\}\right).
\end{split}\end{equation*}
Clearly, if $\{y_i\}_{i=1}^\infty$ is a $\sigma^\prime$-sequence, then every subsequence of $\{y_i\}_{i=1}^\infty$ is a $\sigma$-sequence and so, $y_i\to y$ as $i\to\infty$. The proof is complete.
\end{proof}

\begin{slem}\label{5.2.6}
Let $p\in X$ be a $W$-point of G. If $A\subseteq X$ with $p\in\overline{A}$, then there exists a sequence $\{x_n\}_{n=1}^\infty$ in $A$ such that $x_n\to p$ as $n\to\infty$.
\end{slem}

\begin{proof}
Assume $p\notin A$; for otherwise, taking $x_n=p$ for all $n\in\mathbb{N}$. Let $\sigma^\prime$, as in Remark~\ref{5.2.5}, be a winning strategy for Player $\alpha$ in the $\mathscr{G}_p(X)$-game. Let $U_1=\sigma^\prime(p)\in\mathfrak{N}_p^\textrm{o}(X)$; then choose $x_2\in U_1\cap A$. Let $U_2=\sigma^\prime(p,x_2)\in\mathfrak{N}_p^\textrm{o}(X)$; then choose $x_3\in U_2\cap A$. Inductively, we can construct a $\mathscr{G}_p(X)$-play $\{(x_n,U_n)\}_{n=1}^\infty$ with $x_1=p$ based on $\sigma^\prime$. Then $x_n\in A\to p$ as $n\to\infty$. The proof is complete.
\end{proof}

\begin{sthm}\label{5.2.7}
Let $X$ be a $W$-space of G. Then the following two statements hold:
\begin{enumerate}[(1)]
\item If $Y$ is countably tight, then $X\times Y$ has countable tightness (cf.~\cite[Cor.~3.4 $\&$ Thm.~4.2]{G76}).
\item If $X$ and $Y$ are (locally) countable compact spaces, then $X\times Y$ is (locally) countably compact.
\end{enumerate}
\end{sthm}

\begin{proof}
(1) is \cite[Thm.~4.2]{G76}. To prove (2), let $\{(x_n,y_n)\colon n\in\mathbb{N}\}\subseteq X\times Y$ be arbitrarily given. Since $X$ is countably compact, it follows from Lemma~\ref{5.2.6} that there is a subsequence $\{x_{n(i)}\}_{i=1}^\infty$ of $\{x_n\}_{n=1}^\infty$ such that $x_{n(i)}\to x\in X$ as $i\to\infty$. Further, there exists a subnet $\{(x_{n(i(\alpha))},y_{n(i(\alpha))})\colon \alpha\in\Lambda\}$ of $\{(x_{n(i)},y_{n(i)})\}_{i=1}^\infty$ such that $(x_{n(i(\alpha))},y_{n(i(\alpha))})\to(x,y)\in X\times Y$. Thus, $X\times Y$ is countably compact. The proof is complete.
\end{proof}

Consequently, by Theorem~\ref{5.2.7}-(1) and Theorem~\ref{4.1.5}$^\prime$ (resp. Thm.~\ref{4.1.6}), a $W$-space of G with a rich family of Baire (resp. non-meager) subspaces is Baire (resp. non-meager) itself. In addition, Theorem~\ref{5.2.7}-(2) gives us a sufficient condition for the countable compactness of $X\times X$, which is useful via Theorems~\ref{2.5} and \ref{3.4} as follows:

\begin{scor}
Let $f\colon X\times Y\rightarrow \mathbb{R}$ be a separately continuous function, where $Y$ is a countably compact $W$-space of G. Let one of the following two conditions be satisfied:
\begin{enumerate}[(1)]
\item $X$ is a $\Pi$-separable space;
\item Player $\beta$ has no winning strategy $\tau$ with $\tau(X)\in\mathscr{O}(X)$ being non-meager in the $\mathcal{J}_{\!p}(X)$-game.
\end{enumerate}
Then there exists a residual set $R$ in $X$ such that $f$ is jointly continuous at each point of $R\times Y$.
\end{scor}

\begin{sthm}\label{5.2.9}
Let $X$ be a regular, $T_1$, $W$-space of G. Then $X$ is hereditarily Baire if and only if each perfect set in $X$ is uncountable.
\end{sthm}

\begin{proof}
Necessity is obvious. For sufficiency, assume all perfect sets in $X$ are uncountable. To prove that $X$ is hereditarily Baire, suppose to the contrary that $X$ is meager; and so, $X=\bigcup_{n\in\mathbb{N}}F_n$, where each $F_n$ is closed nowhere dense in $X$ and $F_n\subseteq F_{n+1}$. In addition, by Theorem~\ref{5.2.7}-(1), $X$ has countable tightness.
First, there exists a countable subspace $Y$ of $X$ such that $F_n\cap Y$ is nowhere dense in $Y$ for all $n\in\mathbb{N}$ (by \cite[Lem.~2.1]{CP05}). Indeed, we can define countable subsets $Y_0=\{y_0\}\subseteq Y_1\subseteq Y_2\subseteq\dotsm$ of $X$ as follows: Let $Y_0=\{y_0\}$ and suppose $Y_{j-1}$ is already defined and let $A_n=F_n\cap Y_{j-1}$ for all $n\in\mathbb{N}$. Then there exists a countable set $C_n\subseteq X\setminus F_n$ with $A_n\subseteq \overline{C}_n$. Set $Y_j=Y_{j-1}\cup(\bigcup_{n\in\mathbb{N}}C_n)$. Thus, no point of $F_n\cap Y_{j-1}$ is in the interior of $F_n\cap Y_j$ in the space $Y_j$, for all $n\in\mathbb{N}$. So, $Y=\bigcup_{j=0}^\infty Y_j$ is as desired ($\because V:=\textrm{int}_Y(F_n\cap Y)\not=\emptyset\Rightarrow \emptyset\not=V\cap(F_n\cap Y_{j-1})\subseteq V\cap(F_n\cap Y_j)\subseteq\textrm{int}_{Y_j}(F_n\cap Y_j)$ as $j$ sufficiently big). Next, note that $F_n\cap \overline{Y}$ is also nowhere dense in the closed subspace $\overline{Y}$ for all $n\in\mathbb{N}$ and $\overline{Y}=\bigcup_{n=1}^\infty (F_n\cap\overline{Y})$. This shows that $\overline{Y}$ is a meager, regular, $T_1$, separable $W$-space of G. By Lemma~\ref{5.2.4}, $\overline{Y}$ is a first countable space. Thus, by Theorem~\ref{4.2.5}, it follows that $\overline{Y}$ and so $X$ contain a countable perfect set, a contradiction. The proof is complete.
\end{proof}

\begin{sse}[$\bigoplus$-product and $\Sigma$-product of spaces]\label{5.2.10}
Let $\{X_i\}_{i\in I}$ be a family of topological spaces and let $\theta=(\theta_i)_{i\in I}\in\prod_{i\in I}X_i$ be any fixed point. Then: 
\begin{enumerate}[\textbf{a.}]
\item[\textbf{a.}] The \textit{$\bigoplus$-product} of $X_i$, $i\in I$, with base point $\theta$, denoted by $\bigoplus_{i\in I}X_i(\theta)$, is the subspace of $\prod_{i\in I}X_i$ consisting of points $x=(x_i)_{i\in I}\in\prod_{i\in I}X_i$ such that $x_i=\theta_i$ for all but finitely many indices $i\in I$; a \textit{cube} $E$ in $\bigoplus_{i\in I}X_i(\theta)$ is a product $\prod_{i\in I}E_i\subset\bigoplus_{i\in I}X_i(\theta)$, where $E_i\subseteq X_i$ is the $i$th-face of $E$ such that $E_i=\{\theta_i\}$ for all but finitely many indices $i\in I$.

\item[\textbf{b.}] The \textit{$\Sigma$-product} of $X_i$, $i\in I$, with base point $\theta$, denoted by $\Sigma_{i\in I}X_i(\theta)$, is the subspace of $\prod_{i\in I}X_i$ consisting of points $x=(x_i)_{i\in I}\in\prod_{i\in I}X_i$ such that $x_i=\theta_i$ for all but countably many indices $i\in I$; a \textit{cube} $E$ in $\Sigma_{i\in I}X_i(\theta)$ is a product $\prod_{i\in I}E_i\subset\Sigma_{i\in I}X_i(\theta)$, where $E_i\subseteq X_i$ is the $i$th-face of $E$ such that $E_i=\{\theta_i\}$ for all but countably many indices $i\in I$ (cf.~e.g., \cite{E89}).
\end{enumerate}
Clearly, $\overline{\bigoplus_{i\in I}X_i(\theta)}=\prod_{i\in I}X_i$, 
$\overline{\Sigma_{i\in I}X_i(\theta)}=\prod_{i\in I}X_i$, and 
$\Sigma_{i\in\mathbb{N}}X_i(\theta)=\prod_{i\in\mathbb{N}}X_i$.
\end{sse}

\begin{slem}[{cf.~\cite[Thm.~3.5]{LM08}}]\label{5.2.11}
Let $\{X_i\,|\,i\in I\}$ be a family of spaces and $\theta\in\prod_{i\in I}X_i$. If each $\mathcal{F}_i$, $i\in I$, is a rich family for $X_i$, then
\begin{enumerate}
    \item[]$
\Sigma_{i\in I}\mathcal{F}_i(\theta):=\left\{\left({\prod}_{i\in I_0}F_i\right)\times\{(\theta_i)_{i\in I\setminus I_0}\}\subseteq\Sigma_{i\in I}X_i(\theta)\,|\,I_0\subseteq I\textrm{ is countable}\ \&\ F_i\in\mathcal{F}_i\ \forall i\in I_0\right\}
$
\end{enumerate}
is a rich family for $\Sigma_{i\in I}X_i(\theta)$.
\end{slem}

Using Remark~\ref{5.2.5}, we can readily prove the following lemma.

\begin{slem}[{cf.~\cite[Thm.~4.6]{G76}}]\label{5.2.12}
If $\{X_i\,|\,i\in I\}$ is a family of $W$-spaces of G, then $\Sigma_{i\in I}X_i(\theta)$ (and so $\bigoplus_{i\in I}X_i(\theta)$) is a $W$-space of G for every $\theta\in\prod_{i\in I}X_i$.
\end{slem}

\subsection{Cartesian product of Baire $W$-spaces}\label{s5.3}
First of all we shall recall a classic theorem of Oxtoby (1960)~\cite{O60} on the products of  families of Baire spaces.

\begin{sthm}[{cf.~\cite[Thm.~3]{O60}}]\label{5.3.1}
The product of any family of Baire spaces, each of which has a countable $\pi$-base, is a Baire space.
\end{sthm}

\begin{scor}\label{5.3.2}
If $Y$ is a $\Pi$-separable Baire space and each $X_i$, $i\in\mathbb{N}$, is Baire and has a countable $\pi$-base, then $Y\times\prod_{i\in\mathbb{N}}X_i$ is a Baire $\mathcal{N}$-space.
\end{scor}

\begin{proof}
Let $Z=\prod_{i\in\mathbb{N}}X_i$. Then by Theorem~\ref{5.3.1}, $Z$ is a Baire space having a countable $\pi$-base (cf.~\cite[(2.5)]{O60}); and so, $Z$ is separable. Furthermore, $Y\times Z$ is $\Pi$-separable and Baire (by Thm.~A.3). Then by Theorem~\ref{2.5}, it follows that $Y\times Z$ is an $\mathcal{N}$-space. The proof is complete.
\end{proof}

\begin{sthm}\label{5.3.3}
Let $X$ be a Baire space and $Y$ an almost $W$-space of G with countable tightness and having a rich family of Baire subspaces. Then $X\times Y$ is Baire.
\end{sthm}

\begin{proof}
Using Lemma~A.8 and a modification of Proof of \cite[Thm.~4.4]{LM08} as follows: Let $R\subseteq X\times Y$ be any residual set and $U\times V$ any basic open set in $X\times Y$. We need prove $(U\times V)\cap R\not=\emptyset$. For that, let $y\in V$ and we can then choose a rich family $\mathcal{F}$ of Baire subspaces for $Y$ such that each member $F\in\mathcal{F}$ contains $y$. Then by Lemma~A.8,
$X_R=\{x\in X\,|\,\exists F\in\mathcal{F}\textrm{ s.t. }F\cap R_x\textrm{ is residual in }F\}$
is residual in $X$. Let $x\in X_R\cap U$ ($\not=\emptyset$ for $X$ is Baire). Since $F(x)\in\mathcal{F}$ is Baire and $y\in F(x)$, there is a net $y_\alpha(x)\in R_x$ with $(x,y_\alpha(x))\in R\to(x,y)\in U\times V$. Thus, $(U\times V)\cap R\not=\emptyset$.
\end{proof}

Theorem~\ref{5.3.3} is comparable with \cite[Thm.~4.4]{LM08} in which $Y$ is a $W$-space of G (so $Y$ has countable tightness by Thm.~\ref{5.2.7}) and $X, Y$ are in the class of Hausdorff spaces; moreover, since both of $X$ and $Y$ need not have a countable $\pi$-base, Theorem~\ref{5.3.1} or Corollary~A.1$^{\prime\prime}$ does not work here. Note that a separable $W$-space of G (not necessarily quasi-regular) need not have a countable $\pi$-base; and moreover, a space with a countable $\pi$-base need not be a $W$-space of G.
\begin{slem}[{cf.~\cite[Cor.~4.5]{LM08} in the class of Hausdorff regular spaces}]\label{5.3.4}
Let $\{X_i\}_{i\in I}$ be a family of $W$-spaces of G such that each of which has a rich family of Baire quasi-regular subspaces. Then $\Sigma_{i\in I}X_i(\theta)$ is a $W$-space of G having a rich family of Baire subspaces for every point $\theta\in\prod_{i\in I}X_i$. In particular, $\Sigma_{i\in I}X_i(\theta)$ is Baire.
\end{slem}

\begin{proof}
First, $\Sigma_{i\in I}X_i(\theta)$ is a $W$-space of G by Lemma~\ref{5.2.12}. Let $\mathcal{F}_i$, for each $i\in I$, be a rich family of Baire subspaces for $X_i$. Then by Lemma~\ref{5.2.11}, $\Sigma_{i\in I}\mathcal{F}_i(\theta)$ is a rich family for $\Sigma_{i\in I}X_i(\theta)$. In view of Theorem~\ref{4.1.5}$^\prime$ and Theorem~\ref{5.2.7}, it remains to prove that every member of $\Sigma_{i\in I}\mathcal{F}_i(\theta)$ is a Baire subspace of $\Sigma_{i\in I}X_i(\theta)$. In fact, if $F\in\Sigma_{i\in I}\mathcal{F}_i(\theta)$, then $F\cong\prod_{i\in I_0}F_i$, where $I_0$ is some countable subset of $I$ and each $F_i\in\mathcal{F}_i$ is a quasi-regular, separable, Baire $W$-space of G. Then by Lemma~\ref{5.2.3} and Theorem~\ref{5.3.1}, it follows easily that $F$ is a Baire space.
The proof is complete.
\end{proof}

Note that our proof of Lemma~\ref{5.3.4} is comparable with Lin and Moors' proof of \cite[Cor.~4.5]{LM08}. To employ Theorem~\ref{5.3.1}, \cite[Thm.~4.3]{LM08} and \cite[Thm.~3.6]{G76} (i.e., Lem.~\ref{5.2.4}), the involving spaces in \cite{LM08} must be in the class of Hausdorff regular spaces. However, we do not need those conditions and \cite[Thm.~4.3]{LM08} here. Moreover, we can improve Theorem~\ref{5.1.2} as follows:

\begin{sthm}\label{5.3.5}
Let $Y$ be a Baire space and $\{X_i\}_{i\in I}$ a family of $W$-spaces of G. If each $X_i$, $i\in I$, possesses a rich family of quasi-regular Baire subspaces, then $Y\times\prod_{i\in I}X_i$ is a Baire space.
\end{sthm}

\begin{proof}
Let $\theta\in\prod_{i\in I}X_i$ be arbitrarily given. Then by Lemma~\ref{5.3.4} and Theorem~\ref{5.3.3}, it follows that $Y\times\Sigma_{i\in I}X_i(\theta)$ is a Baire space. However, since $Y\times\Sigma_{i\in I}X_i(\theta)$ is dense in $Y\times\prod_{i\in I}X_i$, $Y\times\prod_{i\in I}X_i$ is Baire. The proof is complete.
\end{proof}

\begin{scor}[{cf.~\cite[Thm.~2]{M06}}]\label{5.3.6}
If $X$ is a Baire space and $Y$ is a hereditarily Baire metric space, then $X\times Y$ is Baire.
\end{scor}

It may be interesting to remove the technical condition ``quasi-regular'' in Theorem~\ref{5.3.5}. For example we can prove the following:

\begin{sthm}\label{5.3.7}
Let $Y$ be a Baire space and $\{X_i\}_{i\in I}$ a family of $W$-spaces of G. If each $X_i$, $i\in I$, possesses a rich family of Baire subspaces having a countable $\pi$-base, then $Y\times\prod_{i\in I}X_i$ is a Baire space.
\end{sthm}
\begin{proof}
    By Lemma~\ref{5.2.11} and Theorem~\ref{5.3.1}, it follows that $\Sigma_{i\in I}X_i(\theta)$ is a $W$-space of G and has a rich family of Baire subspaces. Then $Y\times\Sigma_{i\in I}X_i(\theta)$ is a Baire space by Theorem~\ref{5.3.3}. However, since $Y\times\Sigma_{i\in I}X_i(\theta)$ is dense in $Y\times\prod_{i\in I}X_i$, hence $Y\times\prod_{i\in I}X_i$ is Baire. The proof is complete.    
\end{proof}

\subsection{Cartesian product of non-meager spaces}\label{s5.4}
We need the following topological Fubini theorem, due to Lin-Moors (2008) \cite[Thm.~4.3]{LM08} that is for $Y$ in the class of Hausdorff $W$-spaces but their proof is still valid for the following general case (see Lem.~A.8 for a more general version), which is a variant of a classic Fubini theorem.

\begin{slem}[A special case of Lem.~A.8]\label{5.4.1}
Let $X$ be a space, $Y$ an almost $W$-space of G having countable tightness, and $R$ a residual subset of $X\times Y$. If $\mathcal{F}$ is any rich family for $Y$, then
\begin{enumerate}
    \item[] $X_R=\{x\in X\,|\,\exists F\in\mathcal{F}\textrm{ s.t. }F\cap R_x\textrm{ is residual in }F\}$
\end{enumerate}
is residual in $X$.
\end{slem}

The following is a variant of Theorem~\ref{5.3.3} (also Lin-Moors \cite[Thm.~4.4]{LM08}) where $Y$ has a rich family of Baire subspaces.

\begin{sthm}\label{5.4.2}
Let $X$ be a non-meager space and $Y$ a $W$-space of G having a rich family of non-meager subspaces. Then $X\times Y$ is non-meager.
\end{sthm}

\begin{proof}
Let $\{G_n\}_{n=1}^\infty$ be any sequence of open dense subsets of $X\times Y$. We need only prove that $R:=\bigcap_{n=1}^\infty G_n\not=\emptyset$. For that, we first take a rich family $\mathcal{F}$ of non-meager subspaces for $Y$. Then by Lemma~\ref{5.4.1}, it follows that $X_R\not=\emptyset$ for $X$ is non-meager. Now, for all $x\in X_R$, $F\cap R_x\not=\emptyset$ for some $F\in\mathcal{F}$ since $F$ is non-meager. Thus, $R\not=\emptyset$.
\end{proof}

We can generalize Oxtoby's theorem \cite[Thm.~3]{O60} (i.e., Thm.~\ref{5.3.1}). For that we shall need a lemma, which is contained in Oxtoby's proof of \cite[(2.7)]{O60} in the special case that each factor has a countable $\pi$-base and $\mathscr{F}$ consists of basic open sets:

\begin{slem}[{cf.~\cite{M47} or \cite[(2.7)]{O60}}]\label{5.4.3}
Let $X$ be a $\Pi$-separable space. Then any disjoint family $\mathscr{F}$ of open subsets of $X$ is countable.
\end{slem}

\begin{proof}
Let $X=\prod_{\alpha\in A}X_\alpha$, where each $X_\alpha$, $\alpha\in A$, is a separable space.
Let $D_\alpha$ be a countable dense set in $X_\alpha$. Assign positive weights with sum $1$ to the points of $D_\alpha$. For any Borel set $E\subseteq X_\alpha$, let $\mu_\alpha(E)$ be the sum of the weights of the points of $D_\alpha\cap E$. Then $\mu_\alpha$ is a measure defined for all Borel subsets of $X_\alpha$ such that $\mu_\alpha(X_\alpha)=1$ and $\mu_\alpha(U)>0$ for all $U\in\mathscr{O}(X_\alpha)$. Let $(X,\bigotimes_{\alpha\in A}\mathscr{B}_\alpha,\mu)$ denote the product of the Borel probability spaces $(X_\alpha,\mathscr{B}_\alpha,\mu_\alpha)$, $\alpha\in A$. As $\mu(X)=1$, it follows that any disjoint family of open sets in $X$ is countable. The proof is complete.
\end{proof}

\begin{sthm}\label{5.4.4}
    Let $X=\prod_{\alpha\in A}X_\alpha$, where each factor $X_\alpha$ is a separable space. If $\prod_{\alpha\in A^\prime}X_\alpha$ is a Baire space for each countable set $A^\prime\subset A$, then $X$ is a Baire $\mathcal{N}$-space.
\end{sthm}

\begin{proof}
    Since $X$ is a $\Pi$-separable space, in view of Theorem~\ref{1.3}-(1) or Corollary~\ref{2.5}A we need only prove that $X$ is a Baire space. For that, let $G_n$, $n\in\mathbb{N}$, be a dense open subset of $X$. To prove $X$ Baire, it suffices to prove that $\bigcap_nG_n$ is dense in $X$. For this, for each $n\in\mathbb{N}$, let $\{V_{n,m}\colon m=1,2,\dotsc\}$ be a maximal disjoint family of basic open sets contained in $G_n$ (by Lem.~\ref{5.4.3}). Then $H_n=\bigcup_{m}V_{n,m}$ is dense open in $G_n$; and moreover, there exists a countable set $A_n^\prime\subset A$ and a dense open set $H_n^\prime\subseteq\prod_{\alpha\in A_n^\prime}X_\alpha$ such that $H_n=H_n^\prime\times\prod_{\alpha\in A\setminus A_n^\prime}X_\alpha\subseteq G_n$. Let $A^\prime=\bigcup_{n=1}^\infty A_n^\prime$ and $H_n^{\prime\prime}=H_n^\prime\times\prod_{\alpha\in A^\prime\setminus A_n^\prime}X_\alpha$. Then $H_n^{\prime\prime}$ is dense open in $\prod_{\alpha\in A^\prime}X_\alpha$ for each $n\in\mathbb{N}$. Thus, $R:=\bigcap_{n=1}^\infty H_n^{\prime\prime}$ is dense in $\prod_{\alpha\in A^\prime}X_\alpha$. So, $R\times \prod_{\alpha\in A\setminus A^\prime}X_\alpha\subseteq\bigcap_{n=1}^\infty G_n$ is dense in $X$. The proof is complete.
\end{proof}

\begin{sthm}\label{5.4.5}
    Let $X=\prod_{\alpha\in A}X_\alpha$, where each factor $X_\alpha$ is a separable space. If $\prod_{\alpha\in A^\prime}X_\alpha$ is a non-meager space for each countable set $A^\prime\subset A$, then $X$ is a non-meager g$\mathcal{N}$-space.
\end{sthm}

\begin{proof}
    In view of Theorem~\ref{1.3}-(1) or Corollary~\ref{2.5}A, we need only prove that $X$ is non-meager. For that, let $G_n$, $n\in\mathbb{N}$, be a dense open subset of $X$. To prove $X$ non-meager, it suffices to prove that $\bigcap_nG_n\not=\emptyset$. For this, for each $n\in\mathbb{N}$, let $\{V_{n,m}\colon m=1,2,\dotsc\}$ be a maximal disjoint family of basic open sets contained in $G_n$ (by Lem.~\ref{5.4.3}). Then $H_n=\bigcup_{m}V_{n,m}$ is dense open in $G_n$; and moreover, there exists a countable set $A_n^\prime\subset A$ and a dense open set $H_n^\prime\subseteq\prod_{\alpha\in A_n^\prime}X_\alpha$ such that $H_n=H_n^\prime\times\prod_{\alpha\in A\setminus A_n^\prime}X_\alpha\subseteq G_n$. Let $A^\prime=\bigcup_{n=1}^\infty A_n^\prime$ and $H_n^{\prime\prime}=H_n^\prime\times\prod_{\alpha\in A^\prime\setminus A_n^\prime}X_\alpha$. Then $H_n^{\prime\prime}$ is dense open in $\prod_{\alpha\in A^\prime}X_\alpha$ for each $n\in\mathbb{N}$. Thus, $R:=\bigcap_{n=1}^\infty H_n^{\prime\prime}\not=\emptyset$ in $\prod_{\alpha\in A^\prime}X_\alpha$. So, $\emptyset\not=R\times \prod_{\alpha\in A\setminus A^\prime}X_\alpha\subseteq\bigcap_{n=1}^\infty G_n$. The proof is complete.
\end{proof}
\section{Category analogues of Kolmogoroff's Zero-One Law}\label{s6}
We shall prove two category analogues\,(Thm.~\ref{6.2.4} and Thm.~\ref{6.2.6}) of the classic zero-one law of Kolmogoroff in the theory of probability. Given $A,B\subseteq X$, $A\vartriangle B:=(A\setminus B)\cup (B\setminus A)$ is called the symmetric difference of $A$ and $B$ in $X$. Then $A\vartriangle B=A^c\vartriangle B^c$, where $A^c=X\setminus A$ and $B^c=X\setminus B$.
\subsection{Ergodicity of shifts and finite permutations}
For our convenience we shall introduce the classic Kolmogoroff and Hewitt-Savage zero-one laws.
Let $I$ be an infinite index set, denumerable or non-denumerable. For each $i\in I$, let $(\Omega_i,\mathscr{F}_i,P_i)$ be a probability space. Let
\begin{enumerate}
    \item[] $X={\prod}_{i\in I}\Omega_i=\{x=(x_i)_{i\in I}\colon x_i\in \Omega_i\ \forall i\in I\}$.
\end{enumerate}
On $X$ we have the canonical product $\sigma$-field $\bigotimes_{i\in I}\mathscr{F}_i$, the smallest $\sigma$-field on $X$ making each coordinate projection $\pi_i\colon X\rightarrow\Omega_i$ measurable, and the product probability $\bigotimes_{i\in I}P_i$ given by
\begin{enumerate}
    \item[] $
{\bigotimes}_{i\in I}P_i(A_{i_1}\times\dotsm\times A_{i_n})=P_{i_1}(A_{i_1})\dotsm P_{i_n}(A_{i_n})$ 
\end{enumerate}
for all $n\in\mathbb{N}$, $i_1,\dotsc, i_n\in I$, and $A_{i_1}\in\mathscr{F}_{i_1},\dotsc,A_{i_n}\in\mathscr{F}_{i_n}$,
where
\begin{enumerate}
    \item[] $
A_{i_1}\times\dotsm\times A_{i_n}=\left\{x=(x_i)_{i\in I}\in X\,|\,x_{i_1}\in A_{i_1}, \dotsc, x_{i_n}\in A_{i_n}\right\}\in{\bigotimes}_{i\in I}\mathscr{F}_i$. 
\end{enumerate}
Note that the collection of all cylindrical sets $A_{i_1}\times\dotsm\times A_{i_n}$ of finite length is an algebra, which may generates $\bigotimes_{i\in I}\mathscr{F}_i$.
Given any finite set $J\subset I$, we can define $\sigma$-subfields of $\bigotimes_{i\in I}\mathscr{F}_i$ as follows:
\begin{enumerate}
    \item[] $
\left({\bigotimes}_{j\in J}\mathscr{F}_j\right)\times\left({\prod}_{i\in I\setminus J}\Omega_i\right)\quad \textrm{and}\quad\left({\prod}_{j\in J}\Omega_j\right)\times\left({\bigotimes}_{i\in I\setminus J}\mathscr{F}_i\right)$.
\end{enumerate}
Let
\begin{enumerate}
    \item[] $
\mathscr{F}^{(\infty)}={\bigcap}_J\left({\prod}_{j\in J}\Omega_j\right)\times\left({\bigotimes}_{i\in I\setminus J}\mathscr{F}_i\right)
$
\end{enumerate}
where $J$ varies in the collection of all finite subsets of $I$, which is of course a $\sigma$-subfield of $\bigotimes_{i\in I}\mathscr{F}_i$.

\begin{sse}[Tail events]\label{6.1.1}
An event $A\in\bigotimes_{i\in I}\mathscr{F}_i$ is called a \textit{tail event} if $A\in\mathscr{F}^{(\infty)}$. See, e.g., \cite[p.~53]{K02} for the case that $I$ is denumerable.
\end{sse}

\begin{sthm}[{Kolmogoroff's 0-1 Law; cf.~\cite[Thm.~3.13]{K02} or \cite[Thm.~21.3]{O80} for $I=\mathbb{Z}_+$}]\label{6.1.2}
Let $(\Omega_i,\mathscr{F}_i,P_i)$, $i\in I$, be any family of probability spaces. Then, $\bigotimes_{i\in I}P_i(A)=0$ or $1$ for all $A\in\mathscr{F}^{(\infty)}$.
\end{sthm}

\begin{proof}
Let $A\in\mathscr{F}^{(\infty)}$ be any tail event. Then for all $n\in\mathbb{N}$, there exists a finite set $J_n\subset I$ and an event $B_n\in\left({\bigotimes}_{j\in J_n}\mathscr{F}_j\right)\times\left({\prod}_{i\in I\setminus J_n}\Omega_i\right)$ such that $\bigotimes_{i\in I}P_i(A\vartriangle B_n)<1/n$. By $A\in\mathscr{F}^{(\infty)}$, there exists an event $C_n\in{\bigotimes}_{i\in I\setminus J_n}\mathscr{F}_i$ such that $A=\left(\prod_{j\in J_n}\Omega_j\right)\times C_n$. Thus,
\begin{equation*}\begin{split}
{\bigotimes}_{i\in I}P_i(A)&=\lim_{n\to\infty}{\bigotimes}_{i\in I}P_i(B_n)=\lim_{n\to\infty}{\bigotimes}_{i\in I}P_i(A\cap B_n)=\lim_{n\to\infty}{\bigotimes}_{i\in I}P_i(A)\cdot{\bigotimes}_{i\in I}P_i(B_n)\\
&={\bigotimes}_{i\in I}P_i(A)\cdot{\bigotimes}_{i\in I}P_i(A).
\end{split}\end{equation*}
So $\bigotimes_{i\in I}P_i(A)=0$ or $1$. The proof is complete.
\end{proof}

\begin{sse}[$G$-shift]
Let $G$ be an infinite group. We now consider the special case where all $(\Omega_i,\mathscr{F}_i,P_i)$, $i\in G$, are copies of a probability space $(\Omega,\mathscr{F},P)$. In this case let
\begin{enumerate}
    \item[]$
\Omega^G={\prod}_{i\in G}\Omega_i$,\quad $\mathscr{F}^G={\bigotimes}_{i\in G}\mathscr{F}_i$,\quad $P^G={\bigotimes}_{i\in G}P_i$.
\end{enumerate}
Given $t\in G$ and $x=(x_i)_{i\in G}\in\Omega^G$, put $tx=(x_{it})_{i\in G}$. Then $tx\in\Omega^G$. Let
\begin{enumerate}
    \item[]$
\sigma\colon G\times\Omega^G\rightarrow\Omega^G,\quad (t,x)\mapsto tx
$.
\end{enumerate}
Clearly, $P^G=t_*P^G$ for all $t\in G$. Thus, $G\curvearrowright_\sigma\!\left(\Omega^G,\mathscr{F}^G, P^G\right)$ is a measure-preserving flow. Note that a $G$-invariant event $A\in\mathscr{F}^G$ (i.e., $tA=A\ \forall t\in G$) is not necessarily a tail event.
\end{sse}

\begin{sthm}[Ergodicity of $G$-shift]\label{6.1.4}
The $G$-shift flow $G\curvearrowright_\sigma\!\left(\Omega^G,\mathscr{F}^G, P^G\right)$ is ergodic; that is, if $A\in\mathscr{F}^G$ is $G$-invariant, then $P^G(A)=0$ or $1$.
\end{sthm}

\begin{proof}
Let $A\in\mathscr{F}^G$ be any $G$-invariant event. For all $n\in\mathbb{N}$, there exists a finite set $J_n\subset G$ and an event $B_n\in\mathscr{F}^{J_n}\times\Omega^{G\setminus J_n}$ such that $P^G(A\vartriangle B_n)<1/n$. As $J_n$ is finite and $G$ is an infinite group, it follows that one can choose an element $t_n\in G$ such that $J_nt_n\cap J_n=\emptyset$. Then
$$
P^G(t_nB_n)=P^G(B_n)\quad \textrm{and}\quad P^G(A\vartriangle B_n)=P^G(t_n(A\vartriangle B_n))=P^G(A\vartriangle t_nB_n)\to 0\ \textrm{as }n\to\infty.
$$
So
$$
P^G(A)=\lim_{n\to\infty}P^G(A\cap B_n)=\lim_{n\to\infty}P^G(B_n\cap t_nB_n)=\lim_{n\to\infty}P^G(B_n)\cdot P^G(t_nB_n)=P^G(A)\cdot P^G(A).
$$
$P^G(A)=0$ or $1$. The proof is complete.
\end{proof}

\begin{sse}[Symmetric events]
Let $I$ be an infinite index set, $(\Omega,\mathscr{F},P)$ a probability space, and $\mathcal{P}_{\!I}$ the group of all finite permutations of $I$.
A set $A\subseteq\Omega^I$ is called \textit{symmetric}, if $px=(x_{p(i)})_{i\in I}\in A$ for all $x=(x_i)_{i\in I}\in A$ and all $p\in\mathcal{P}_{\!I}$. Let
\begin{enumerate}
    \item[] $
\rho\colon\mathcal{P}_{\!I}\times\Omega^I\rightarrow\Omega^I,\quad (p,x)\mapsto px
$.
\end{enumerate}
Then, $P^I=p_*P^I$ for all $p\in\mathcal{P}_{\!I}$; and so, $\mathcal{P}_{\!I}\curvearrowright_\rho\!\left(\Omega^I,\mathscr{F}^I,P^I\right)$ is a measure-preserving flow. Moreover, $A\in\mathscr{F}^I$ is symmetric if and only if $A$ is $\mathcal{P}_{\!I}$-invariant (cf.~\cite[p.~53]{K02} for $I=\mathbb{N}$).
\end{sse}

\begin{sthm}[{Hewitt-Savage 0-1 Law; cf.~\cite[Thm.~11.3]{HS55} or \cite[Thm.~3.15]{K02} for $I=\mathbb{Z}_+$}]\label{6.1.6}
Let $I$ be an infinite index set and $(\Omega,\mathscr{F},P)$ a probability space. Then $\mathcal{P}_{\!I}\curvearrowright_\rho\!\left(\Omega^I,\mathscr{F}^I,P^I\right)$ is ergodic; i.e., $P^I(A)=0$ or $1$ for all symmetric event $A\in\mathscr{F}^I$.
\end{sthm}

\begin{proof}
Let $A\in\mathscr{F}^I$ be any symmetric event. For all $n\in\mathbb{N}$, there exists a finite set $J_n\subset I$ and an event $B_n\in\mathscr{F}^{J_n}\times\Omega^{I\setminus J_n}$ such that $P^I(A\vartriangle B_n)<1/n$. As $J_n$ is finite and $I$ is infinite, it follows that one can choose an element $p_n\in \mathcal{P}_{\!I}$ such that $p_n(J_n)\cap J_n=\emptyset$. Thus, $B_n$ and $p_nB_n$ are independent in $(\Omega^I,\mathscr{F}^I,P^I)$. Noting that
$$
P^I(p_nB_n)=P^I(B_n)\quad \textrm{and}\quad P^I(A\vartriangle B_n)=P^I(p_n(A\vartriangle B_n))=P^I(A\vartriangle p_nB_n)\to 0\ \textrm{as }n\to\infty,
$$
it follows that
$$
P^I(A)=\lim_{n\to\infty}P^I(A\cap B_n)=\lim_{n\to\infty}P^I(B_n\cap p_nB_n)=\lim_{n\to\infty}P^I(B_n)\cdot P^I(p_nB_n)=P^I(A)\cdot P^I(A).
$$
Thus, $P^I(A)=0$ or $1$. The proof is complete.
\end{proof}

\subsection{Category analogues}
In this subsection we will consider two category analogues of Kolmogoroff's zero-one law. Meanwhile, we shall improve a classic theorem of Oxtoby (1960) \cite{O60}.
\begin{sse}[Tail sets]\label{6.2}
Let $X=\prod_{\alpha\in A}X_\alpha$. A set $E\subset X$ will be called a \textit{tail set}~\cite{O80}, if whenever $x=(x_\alpha)_{\alpha\in A}$ and $y=(y_\alpha)_{\alpha\in A}$ are points of $X$, and $x_\alpha=y_\alpha$ for all but finite number of $\alpha\in A$, then $E$ contains both $x$ and $y$ or neither.
\end{sse}

For any set $J\subset A$, finite or infinite, we shall write $X_J=\prod_{j\in J}X_j$. Then Definition~\ref{6.2} can be cast in a more convenient form as follows:
\begin{enumerate}[$\bullet$]
\item $E\subset\prod_{\alpha\in A}X_\alpha$ is a tail set if and only if for each finite set $J\subseteq A$ there is a set $B_J\subset X_{A\setminus J}$ such that $E=X_J\times B_J$.
\end{enumerate}
\begin{proof}
Indeed, sufficiency is obvious. Now conversely, suppose $E$ is a tail set and $J\subseteq A$ is a finite set. Let $B_J=\{y\in X_{A\setminus J}\,|\,\exists x_J\in X_J\textrm{ s.t. }(x_J,y)\in E\}$. Then $E=X_J\times B_J$.
\end{proof}

Subsequently, a tail event (Def.~\ref{6.1.1}) is a tail set.

\begin{sse}[Property of Baire]
A subset $E$ of a topological space is said to have the \textit{property of Baire}~\cite{O60, O80}, iff $E=G\!\vartriangle\!P$ where $G$ is open and $P$ is meager, iff $E=F\!\vartriangle\!Q$ where $F$ is closed and $Q$ is meager.

Note that a meager set has the property of Baire. Open set and closed set both have the property of Baire. In particular, if $A$ has the property of Baire, then so does its complement. In fact, the class of sets having the property of Baire is a $\sigma$-field generated by the open sets together with the meager sets \cite[Thm.~4.3]{O80}. Thus, every Borel subset of a topological space has the property of Baire.
\end{sse}

\begin{sthm}[{cf.~\cite[Thm.~4]{O60}}]\label{6.2.3}
Let $X$ be the product of a family of Baire spaces, each of which has a countable $\pi$-base. Then $X$ is a Baire space; and moreover, any tail set having the property of Baire in $X$ is either meager or residual.
\end{sthm}

Now we can generalize Theorem~\ref{6.2.3} from the class of Baire spaces to the class of arbitrary topological spaces as follows:

\begin{sthm}\label{6.2.4}
Let $X=\prod_{\alpha\in A}X_\alpha$ where each factor has a countable $\pi$-base. Then any tail set having the property of Baire in $X$ is either meager or residual.
\end{sthm}

\begin{proof}
Let $E$ be any tail set having the property of Baire in $X$. Suppose $E$ is not residual in $X$; and so, $X\setminus E$ is non-meager and has the property of Baire. Then there exists an open non-void set $G$ non-meager and a set $P$ meager in $X$ such that $X\setminus E=G\!\vartriangle\!P$. Let $\{G_i\}$ be a maximal disjoint family (countable by Lem.~\ref{5.4.3}) of basic open sets contained in $G$. Then $\bigcup_iG_i$ is dense open in $G$ so that $G\setminus \bigcup_iG_i$ is nowhere dense. Since $G$ is non-meager, $\bigcup_iG_i$ is non-meager so that at least one of the sets $G_i$ is non-meager, say $G_i=U\times X_{A\setminus J}$, where $J\subseteq A$ is some finite set and $U\in\mathscr{O}(X_J)$. So, $U$ is non-meager in $X_J$. By Definition~\ref{6.2}, $E=X_J\times B$ for some set $B\subset X_{A\setminus J}$. Hence $E\cap G_i=(U\cap X_J)\times(X_{A\setminus J}\cap B)=U\times B$. As $E\cap G_i\subseteq E\cap G=G\cap P\subseteq P$ and $P$ is meager, it follows that $U\times B$ is meager; and so, $B$ is meager in $X_{A\setminus J}$ by Theorem~A.3. Thus, $E$ is meager by Theorem~A.3 again. The proof is complete.
\end{proof}

In view of Lemma~\ref{5.2.3}, Lemma~\ref{6.2.5} below may be thought of as a variant of the Kuratowski-Ulam-Sikorski theorem (Thm.~A.3), which gives us an equivalent description of $A\times B$ being meager in $X\times Y$.

\begin{slem}\label{6.2.5}
Let $X$ and $Y$ be spaces at least one of which is a separable $W$-space of G. Let $A\subseteq X$ and $B\subseteq Y$. Then $A\times B$ is meager in $X\times Y$ if and only if either $A$ or $B$ is meager in $X$ or $Y$.
\end{slem}

\begin{proof}
Letting $\mathcal{F}=\{Y\}$ be a rich family for $Y$ if $Y$ is a separable $W$-space of G, by Lemma~A.8 and a modification of Proof of Theorem~A.3, it follows that if $A\times B$ is meager in $X\times Y$ and $A$ is non-meager in $X$, then $B$ must be meager in $Y$.
\end{proof}

\begin{sthm}\label{6.2.6}
Let $X=\prod_{\alpha\in A}X_\alpha$ such that each factor is a separable $W$-space of G. Then any tail set having the property of Baire is either meager or residual in $X$.
\end{sthm}

\begin{proof}
By Lemma~\ref{6.2.5} in place of Theorem~A.3, the statement follows easily from Proof of Theorem~\ref{6.2.4}.
\end{proof}

We note that neither of Theorems~\ref{6.2.4} and \ref{6.2.6} includes the other because of the lack of the quasi-regularity (see Lem.~\ref{5.2.3}).

\begin{srem}\label{6.2.7}
Let $(X,\mathscr{B},P)$ be a Borel probability space such that $P(U)>0$ for all $U\in\mathscr{O}(X)$ and $I$ an infinite index set. \textit{If $E\in\mathscr{B}^I$ is a symmetric set and has the property of Baire, is $E$ either meager or residual in $X^I$ and $P^I(E)=1$ $\Leftrightarrow$ $E$ being residual?}
\end{srem}
\section{Non-meagerness of g$\mathcal{N}$-spaces}\label{s7}
This section will be mainly devoted to proving the sufficiency part of Theorem~\ref{1.3}-(2b) and Theorem~\ref{1.3}-(7) stated in $\S$\ref{1.3}. See Theorems~\ref{7.8}, \ref{7.3} and \ref{7.9}.

Recall that $X$ is a completely regular space (or a uniform space \cite{K55}) iff for all $x\in X$ and $U\in\mathfrak{N}_x(X)$, there exists a continuous function $f\colon X\rightarrow[0,1]$ such that $f(x)=0$ and $f\!\upharpoonright_{X\setminus U}\equiv1$.
In (1983) \cite{Ch83} Christensen conjectured that any metrizable $\mathcal{N}$-space is Baire. In fact, it is true in the category of completely regular spaces.

\begin{thm}[{cf.~\cite[Thm.~3]{SR83}}]\label{7.1}
Let $X$ be a completely regular space. If $X$ is an $\mathcal{N}$-space, then $X$ is Baire.
\end{thm}

\begin{lem}[{cf.~\cite[Lem.~4]{SR83}}]\label{7.2}
Let $X$ be completely regular and $F\subset X$ a nowhere dense set. Then there exists a compact Hausdorff space $Y$ and a separately continuous function $f\colon X\times Y\rightarrow[0,1]$ such that for each $x\in F$, there is a point $y\in Y$ such that $f$ is discontinuous at $(x,y)$.
\end{lem}

Lemma~\ref{7.2} plays an important role in Saint-Raymond's proof of Theorem~\ref{7.1}. It will be still useful for us to prove Theorem~\ref{7.3} below; and so, we shall present its proof in \ref{B} for reader's convenience.

\begin{thm}\label{7.3}
Let $X$ be completely regular. If $X$ is a g$\mathcal{N}$-space, then $X$ is non-meager in itself.
\end{thm}

\begin{proof}
Suppose to the contrary that $X$ is meager in itself. Then there exists a sequence of nowhere dense closed sets, $\{F_n\}_{n=1}^\infty$, such that $X=\bigcup_nF_n$. By Lemma~\ref{7.2}, we have for each $n\in\mathbb{N}$ that there is a separately continuous function $f_n\colon X\times Y_n\rightarrow[0,1]$ such that $Y_n$ is a compact Hausdorff space and that for each $x\in F_n$ there exists a point $y\in Y_n$ such that $f_n$ is discontinuous at $(x,y)$.
Let $Y=\prod_nY_n$ be the product topological space. Then $Y\times[0,1]^\mathbb{N}$ is a compact Hausdorff space. Now we can define a separately continuous function
$f\colon X\times Y\rightarrow[0,1]^\mathbb{N}$ by $(x,(y_i)_{i\in\mathbb{N}})\mapsto(f_i(x,y_i))_{i\in\mathbb{N}}$.
Then for all $x\in X$, there exists some $n\in\mathbb{N}$ with $x\in F_n$; and so, there exists a point in $\{x\}\times Y$ at which $f$ is not jointly continuous. Let $d$ be a compatible metric for $[0,1]^\mathbb{N}$. Next, define a separately continuous function
$\tilde{f}\colon X\times\left(Y\times[0,1]^\mathbb{N}\right)\rightarrow\mathbb{R}$ by $(x,(y,w))\mapsto d(f(x,y),w)$.
Let $x_0\in X$ be arbitrary. Then there is a point $y_0=y(x_0)\in Y$ such that $f$ is discontinuous at $(x_0,y_0)$. Let $w_0=f(x_0,y_0)$. Then $\tilde{f}$ is discontinuous at $(x_0,(y_0,w_0))$,
contrary to that $X$ is a g$\mathcal{N}$-space.
\end{proof}

It turns out that if $X$ is a completely regular $T_1$-space, then the Stone-\v{C}ech compactification $\beta X$ is well defined (cf.~\cite[Thm.~5.24]{K55}); and further, Theorem~\ref{7.3} follows readily from the following:

\begin{thm}[{cf.~\cite[Prop.~4.1]{BP05}}]
Let $X$ be a meager, completely regular, $T_1$-space. Then there exists a separately continuous function $\phi\colon X\times\beta X\rightarrow[0,1]$ such that $\phi\!\upharpoonright_\varDelta\colon\varDelta\rightarrow[0,1]$ is discontinuous at each point of $\varDelta=\{(x,x)\,|\,x\in X\}$.
\end{thm}

It is well known that even in the realm of completely regular $T_1$-spaces, a Choquet space need not be an $\mathcal{N}$-space; see Talagrand (1985) \cite[Thm.~2]{T85} that solves a question of Namioka (\cite[Remark~1.3-(b)]{N74}). Haydon 1999 proved that there are Choquet spaces $B$ and compact scattered spaces $K$ such that $\langle B,K\rangle$ are not Namioka pairs. In addition, Burke-Pol (2005) \cite[Thm.~1.1]{BP05} showed that there is a Choquet completely regular $T_1$-space $B$ and a separately continuous function $f\colon B\times\beta B\rightarrow\mathbb{R}$ such that the set of points of continuity of $f\!\upharpoonright_\varDelta\colon\varDelta\rightarrow\mathbb{R}$ is not dense in the space $\varDelta=\{(b,b)\,|\,b\in B\}$; and so, $\langle B,\beta B\rangle$ is not a weak-Namioka pair.

In fact, a Choquet space and so a non-meager space, need not be a g$\mathcal{N}$-space as shown by the following example which is due to Talagrand, but our new ingredient is \ref{7.5}-(3).

\begin{ex}\label{7.5}
Let $T$ be an uncountable discrete space, $\mathscr{J}$ the family of countable non-void subsets of $T$, and $\beta T$ the Stone-\v{C}ech compactification of $T$.
Let $Y=\beta T\setminus T$ and we define
$$
\varPsi=\{p\in\beta T\,|\,T\cap U\notin\mathscr{J}\ \forall U\in\mathfrak{N}_p(\beta T)\textrm{ clopen}\}.
$$
Then $\varPsi\not=\emptyset$ is closed. Indeed, if $\varPsi=\emptyset$, then for all $p\in\beta T$ there is a clopen set $U_p\in\mathfrak{N}_p(\beta T)$ such that $T\cap U_p$ is countable dense in $U_p$; however, since $\beta T$ is compact, there is a countable set $J\subset T$ with $\bar{J}=\beta T$, contrary to that $T$ is uncountable, discrete, and open in $\beta T$.
Let
$$
X=\left\{x\in\{0,1\}^T\,|\,\{t\in T\colon x(t)=1\}\in\mathscr{J}\right\}.
$$
Given $x\in X$, let $x^\beta\colon \beta T\rightarrow\{0,1\}$ be the unique continuous extension of $x\colon T\rightarrow\{0,1\}$.
Let
$$
f\colon X\times\beta T\rightarrow\{0,1\},\quad (x,y)\mapsto f(x,y)=x^\beta(y)
$$
be the canonical evaluation map. Let
$U(x,J)=\{x^\prime\in X\,|\,x\!\upharpoonright_J={x^\prime}\!\upharpoonright_J\}$ for all $x\in X$ and $ J\in\mathscr{J}$.
Then $\{U(x,J)\,|\,x\in X\ \&\ J\in\mathscr{J}\}$ forms a base of some topology $\mathfrak{T}$ for $X$ (cf.~\cite[Thm.~1.11]{K55}). Then, under the topology $\mathfrak{T}$:
\begin{enumerate}[(1)]
\item[(1)] $X$ is completely regular, Hausdorff, $\alpha$-favorable of BM (so Baire);
\item[(2)] $f\colon X\times\beta T\rightarrow\{0,1\}$ is separately continuous;
\item[(3)] $f\colon X\times Y\rightarrow\{0,1\}$ is separately continuous but discontinuous at any point of $X\times\varPsi$.
\end{enumerate}
Consequently, $X$ is not a g$\mathcal{N}$-space (cf.~Talagrand 1985 \cite[Prob.~3]{T85}).
\end{ex}

\begin{proof}
(1): Since $U(x,J)$ is clopen in $X$ for all $x\in X$ and $J\in\mathscr{J}$, $X$ is completely regular. Given $x\not=y$ in $X$ there is an element $j\in T$ such that $x(j)\not=y(j)$. Let $J=\{j\}$ then $x\in U(x,J)$, $y\in U(y,J)$ and $U(x,J)\cap U(y,J)=\emptyset$. Thus, $X$ is a Tychonoff (completely regular Hausdorff) space.
Next, we claim that $X$ is $\alpha$-favorable of BM. Indeed, assume Player $\beta$ firstly plays $U_1$, then we can choose a set $J_1\in\mathscr{J}$ and $x_1\in X$ with $U(x_1,J_1)\subseteq U_1$ and Player $\alpha$ plays $V_1=U(x_1,J_1)$. At the $n$th-stroke, when Player $\beta$ has played $\{U_k\}_{k=1}^n$, we can choose a set $J_n\in\mathscr{J}$ and a point $x_n\in X$ such that $U(x_n,J_n)\subseteq U_n$ and then Player $\alpha$ plays $V_n=U(x_n,J_n)$. Inductively, we have constructed a BM($X$)-play $\{(U_i,V_i)\}_{i=1}^\infty$. Let $J=\bigcup_{n=1}^\infty J_n$; then $\{0,1\}^J$ is compact Hausdorff. Since $U(x_n,J_n)|_J\cap\{0,1\}^J$ is a closed set in $\{0,1\}^J$, $\bigcap_nU_n=\bigcap_nV_n\not=\emptyset$. Thus, $X$ is $\alpha$-favorable of BM so that $X$ is Baire.

(2): Clearly, $f_x=x^\beta\colon\beta T\rightarrow\{0,1\}$ is a continuous function for each $x\in X$. If $y\in T$, then $f(x,y)=x(y)$ is obviously continuous in $x\in X$. If $y\in\beta T\setminus T$ and $\{x_\lambda\}$ a net with $x_\lambda\to x$ in $X$, then there is a net $\{t_\alpha\,|\,\alpha\in D\}$ in $T$ such that $t_\alpha\to y$ in $\beta T$ and
$f(x,y)=\lim_\alpha x(t_\alpha)$ and $f(x_\lambda,y)=\lim_\alpha x_\lambda(t_\alpha)$. Further, $x_\lambda\in U(x,J)$ and so $x^\beta(y)=x_\lambda^\beta(y)$ eventually if $\exists\,\alpha_1\in D$ s.t. $\{t_\alpha\,|\,\alpha\ge \alpha_1\}\in\mathscr{J}$; and moreover, $f(x,y)=0=f(x_\lambda,y)$ for all $\lambda$ if $\{t_\alpha\,|\,\alpha\ge\alpha_1\}\notin\mathscr{J}$ for all $\alpha_1\in D$. Anyway, $f^y$ is continuous for all $y\in\beta T$. Thus, $f$ is separately continuous. (It should be noted that if $x\in X$ such that $J=\{t\in T\colon x(t)=1\}$ is not a finite set, then ${x^\beta}\!\upharpoonright_{\beta T\setminus T}\not\equiv0$. In fact, if $j_\alpha\in J\to y\in\beta T\setminus T$, then $x^\beta(y)=1$.)

(3): Let $(x,y)\in X\times\varPsi$ and assume $f\colon X\times Y\rightarrow\{0,1\}$ is jointly continuous at $(x,y)$.
Then there exists a set $U\in\mathfrak{N}_x(X)$ and a clopen set $V\in\mathfrak{N}_y(\beta T)$ such that $f(U\times(V\cap Y))=\{c\}$ for some point $c\in\{0,1\}$. Choose $J\in\mathscr{J}$ such that $U(x,J)\subseteq U$.
Let $I\subset(V\cap T)\setminus J$ be a countable set, and so $\overline{I}\subseteq V$; and let $x_1,x_2\in U(x,J)$ such that $x_1(t)\not=x_2(t)$ for all $t\in I$. Now we can take a net $t_i\in I$ and a point $q\in V\cap Y$ such that $t_i\to q$. Then
$$
{\lim}_ix_1(t_i)={\lim}_if(x_1,t_i)=f(x_1,q)=c=f(x_2,q)={\lim}_if(x_2,t_i)={\lim}_ix_2(t_i),
$$
which is impossible. Our construction of Example~\ref{7.5} is complete.
\end{proof}

\begin{thm}\label{7.6}
Let $X$ be an open subspace of a completely regular $\Pi$-separable space. Then:
\begin{enumerate}[(1)]
\item $X$ is a Baire space if and only if $X$ is an $\mathcal{N}$-space (cf.~\cite[Thm.~6]{SR83} for $X$ a separable space).
\item $X$ is non-meager in itself if and only if $X$ is a g$\mathcal{N}$-space.
\end{enumerate}
\end{thm}

\begin{proof}
Necessity of (1) and (2) follows from Theorem~\ref{2.5}. Sufficiency of (1) and (2) follows from Theorems~\ref{7.1} and \ref{7.3}, respectively. 
\end{proof}

\begin{rem}\label{7.7}
Let $X$ be a completely regular $\Pi$-separable space. Then $X$ is a g$\mathcal{N}$-space if and only if it has an open non-void subspace which is an $\mathcal{N}$-space.
\end{rem}

\begin{proof}
Sufficiency is obvious by Theorem~\ref{7.6}. Now, if $X$ is a  g$\mathcal{N}$-space, then by Theorem~\ref{7.3} it is non-meager in itself. So by Remark~\ref{2.10}, $X$ contains an open non-void $\mathcal{N}$-subspace.
\end{proof}

\begin{thm}\label{7.8}
Let $X$ be an open subspace of a $\Pi$-pseudo-metric space. Then:
\begin{enumerate}[(1)]
\item $X$ is a Baire space if and only if $X$ an $\mathcal{N}$-space (cf.~\cite[Thm.~7]{SR83} for $X$ a metric space and \cite[Cor.~1.3]{CP05} for $X$ a $\Pi$-metric space).
\item $X$ is non-meager in itself if and only if $X$ a g$\mathcal{N}$-space.
\end{enumerate}
\end{thm}

\begin{proof}
Necessity of (1) and (2) follows from Theorem~\ref{3.4}.
Sufficiency of (1) and (2) follows from Theorems~\ref{7.1} and \ref{7.3}, respectively. 
\end{proof}

\begin{thm}\label{7.9}
Let $X\times Y$ be such that each factor is either a completely regular separable space or a pseudo-metrizable space. Then: 
\begin{enumerate}[(1)]
\item $X\times Y$ is a Baire space if and only if it is an $\mathcal{N}$-space.
\item $X\times Y$ is a non-meager space if and only if it is a g$\mathcal{N}$-space.
\end{enumerate}
\end{thm}

\begin{proof}
    (1) follows easily from Theorems~\ref{7.1} and \ref{3.11}; and (2) from  Theorems~\ref{7.3} and \ref{3.11}.
\end{proof}

\begin{rem}\label{7.11}
Let $X$ is a $\Pi$-pseudo-metric space. Then $X$ is a g$\mathcal{N}$-space if and only if it has an open non-void subspace which is an $\mathcal{N}$-space.
\end{rem}

\begin{rem}\label{7.10}
    If $G$ is a right/left/semi-topological group such that $G$ is a completely regular g$\mathcal{N}$-space, then $G$ is Baire; but we cannot guarantee that $G$ has the $\mathcal{N}$-property, this is because for all $g, g_0\in G$, there is no commutative diagram of maps in general:
$$
\begin{tikzcd}
G\times Y \arrow[r, "f"] \arrow[d, "\tau"'] & \mathbb{R}  \\
G\times Y \arrow[r, "f"] &\arrow[u, "\psi"]\mathbb{R}
\end{tikzcd}
\quad \textrm{where }
\begin{cases}&\tau\colon G\times Y\rightarrow G\times Y\textrm{ is continuous s.t. } \tau(\{g\}\times Y)\subseteq\{g_0\}\times Y\\
&\textrm{\quad and}\\
&\psi\colon \mathbb{R}\rightarrow \mathbb{R} \textrm{ is continuous}, 
\end{cases}$$
for any separately continuous function $f\colon G\times Y\rightarrow\mathbb{R}$ where $Y$ is a compact Hausdorff space. However, if $G$ is $\Pi$-separable or $\Pi$-pseudo-metric, then $G$ is a g$\mathcal{N}$-space if and only if it is an $\mathcal{N}$-space by Theorem~\ref{7.6} or \ref{7.8}.
\end{rem}

\begin{appendix}
\section{Topological Fubini theorems and category theorems}\label{A}
Fubini's theorem says that if $E\subset\mathbb{R}^2$ is a plane set of measure zero, then $E_x=\{y\,|\,(x,y)\in E\}$ is a linear null set for all $x$ except a set of linear measure zero in $\mathbb{R}$ (cf., e.g.,~\cite[Thm.~14.2]{O80}). For reader's convenience and for the self-closeness, we will present two topological Fubini theorems (Lem.~A.1 and Lem.~A.8). In fact, Lemma~A.8 is a slight modification of Lemma~\ref{5.4.1}.

The first topological Fubini theorem, Lemma~A.1$^\prime$ below, is due to Brouwer (1919)~\cite{B19} in the case that $X,Y$ are intervals, to Kuratowski and Ulam (1932) \cite{KU32} (also \cite[Thm.~15.1]{O80}) for the case that $X,Y$ are separable metric spaces, and to Oxtoby (1960)~\cite[(1.1)]{O60} for the general case. Here we will give a different formulation and simple proof as follows.

\begin{A.1}[Topological Fubini Theorem~I]
Let $X$ and $Y$ be spaces, where $Y$ has a countable $\pi$-base. If $G\subseteq X\times Y$ is dense open, then $X_G=\{x\in X\,|\,G_x\textrm{ is dense open in }Y\}$
is residual in $X$. In particular, if $K\subseteq X\times Y$ is residual, then
$X_K=\{x\in X\,|\,K_x\textrm{ is residual in }Y\}$
is residual in $X$.
\end{A.1}

\begin{proof}
$(X\times Y)\setminus G=F$ is a closed nowhere dense set in $X\times Y$. Then $Y\setminus G_x=F_x\ \forall x\in X$. Let
$B=\left\{x\in X\,|\,\mathrm{int}_YF_x\not=\emptyset\right\}$.
So if $x\notin B$, then $G_x$ is open dense in $Y$. Thus, $X\setminus B\subseteq X_G$ and we need only prove that $B$ is meager in $X$.
For that, let $\{U_n\}_{n=1}^\infty$ be a countable $\pi$-base for $Y$. If $x\in B$, then $U_n\subseteq F_x$ for some $n\in\mathbb{N}$.
Put
$C_n=\{x\in B\,|\,U_n\subseteq F_x\}$ and $D_n=\mathrm{int}_X \overline{C}_n$
for all $n\in\mathbb{N}$. Then $B=\bigcup_{n=1}^\infty C_{n}$, and $B$ is meager in $X$ if each $D_{n}=\emptyset$. Indeed, if $D_{n}\not=\emptyset$, then $U_n\subseteq F_x$ for all $x\in D_n\cap C_n$ and $D_n\cap C_n$ is dense in $D_n$. So $(D_n\cap C_n)\times U_n\subseteq F$ so that $\emptyset\not=D_n\times U_n\subseteq \overline{F}=F$, contrary to $F$ being nowhere dense in $X\times Y$. The proof is complete.
\end{proof}

If $E\subset X\times Y$ is nowhere dense (i.e., $\mathrm{int}\,\overline{E}=\emptyset$), then $G=X\times Y\setminus\overline{E}$ is dense open in $X\times Y$, $G_x=Y\setminus\overline{E}_x$ and $G_x\subseteq Y\setminus E_x$ for all $x\in X$. Thus, Lemma~A.1 is equivalent to the following

\begin{A.1'}[{Topological Fubini Theorem~I$^\prime$; cf.~\cite[(1.1)]{O60}}]
Let $X$ and $Y$ be spaces, where $Y$ has a countable $\pi$-base. If $E$ is nowhere dense (resp.~meager) in $X\times Y$, then $E_x$ is nowhere dense (resp.~meager) in $Y$ for all $x$ except a meager set in $X$.
\end{A.1'}

\begin{A.1''}[{cf.~\cite[Thm.~2]{O60}}]
If $X$ is a Baire space and $Y$ a Baire space having a locally countable $\pi$-base, then $X\times Y$ is a Baire space.
\end{A.1''}

It should be mentioned that in Lemma~A.1 or Lemma~A.1$^\prime$, the hypothesis that $Y$ has a countable $\pi$-base cannot be relaxed even to a locally countable $\pi$-base (Def.~\ref{5.2.2}), as Kuratowski and Ulam showed by an example in \cite{KU32}.

\begin{A.2}[{Banach's Category Theorem~\cite{B30}; cf.~\cite[Thm.~6.35]{K55}\,$\&$\,\cite[Thm.~16.1]{O80}}]
Let $A$ be a subset of a space $X$ and $M(A)$ the union of all open sets $V$ such that $V\cap A$ is meager in $X$. Then $A\cap \overline{M(A)}$ is meager in $X$.
\end{A.2}


Consequently, by $A=X$ in any topological space $X$ the closure of the union of any family of meager open sets is meager (cf.~\cite[Thm.~16.1]{O80}).

\begin{A.3}[{Kuratowski-Ulam-Sikorski Theorem; cf.~\cite{KU32, S47} and \cite[Thm.~1]{O60}}]\label{A4}
Let $X$ and $Y$ be spaces at least one of which has a locally countable $\pi$-base. Let $A\subseteq X$ and $B\subseteq Y$. Then $A\times B$ is meager in $X\times Y$ if and only if either $A$ or $B$ is meager in $X$ or $Y$.
\end{A.3}

\begin{proof}
Sufficiency is obvious. Now, for necessity, assume $A\times B$ is meager in $X\times Y$. Suppose that $A$ is non-meager in $X$, and that $Y$ has a locally countable $\pi$-base $\mathcal{B}$. Let $Y_o=\bigcup\{V\,|\,V\in\mathcal{B}\}$. Then $Y_o$ is dense open in $Y$ so that $Y\setminus Y_o$ is meager in $Y$. Thus, to prove that $B$ is meager in $Y$, we may assume that $B\subseteq Y_o$. So, for each $b\in B$, there exists a member $V\in\mathcal{B}$ with $b\in V$ such that $V$ has a countable $\pi$-base. As $A\times(B\cap V)$ is meager in $X\times V$, it follows from Lemma~A.1$^\prime$ that $B\cap V=A\times(B\cap V)_x\ \forall x\in A$ is meager in $V$ and therefore in $Y$. Then by Theorem~A.2, $B=B\cap\overline{M(B)}$ is meager in $Y$.
\end{proof}

Corollary~A.1$^{\prime\prime}$ and Theorem~A.3 generalize easily to product of finitely many spaces each of which has a locally countable $\pi$-base. But Theorem~A.3 does not generalize to infinite products, even when each space has a countable base. For example, let $X=[0,1]$ and $A=[0,1/2]$; then $A^\infty$ is nowhere dense in $X^\infty$, but $A$ is non-meager in $X$ \cite{KU32}. In addition, if neither of $X$ and $Y$ has a locally countable $\pi$-base, then Theorem~A.3 might be false, even when each space is metrizable (see, e.g., \cite{C76, P79} for counterexamples).

\begin{A.4}[BM$_R$-game; cf.~\cite{O57, LM08}]\label{A5}
Let $R\subseteq X$. By a \textit{BM$_R(X)$-play}, we mean a sequence $\{(B_i,A_i)\}_{i=1}^\infty$ of ordered pairs such that $B_i,A_i\in\mathscr{O}(X)$ and $B_i\supseteq A_i\supseteq B_{i+1}$ for all $i\in\mathbb{N}$, where $B_i$ and $A_i$ are picked alternately by Player $\beta$ and Player $\alpha$, respectively; and moreover, Player $\beta$ is always granted the privilege of the first move. In fact, $\{(B_i,A_i)\}_{i=1}^\infty$ is a BM($X$)-play (Def.~\ref{2.1}A). 
\begin{enumerate}[\textbf{W.C.}:]
    \item Player $\alpha$ \textit{wins} the BM$_R(X)$-play $\{(B_i,A_i)\}_{i=1}^\infty$, if $\bigcap_{i=1}^\infty A_i\subseteq R$; otherwise, Player $\beta$ \textit{wins} this play.
\end{enumerate}
\end{A.4}

\begin{A.5}[{cf.~Oxtoby 1957 \cite{O57}}]
Let $R$ be a subset of a space $X$. Then $R$ is residual in $X$ if and only if there exists a winning strategy for Player $\alpha$ in the BM$_R(X)$-game.
\end{A.5}

\begin{proof}
\textsl{Necessity}: Suppose $R$ is residual in $X$. Then there exists a sequence $\{G_n\}_{n=1}^\infty$ of open dense subsets of $X$ with $R \supseteq \bigcap_{n=1}^{\infty} G_n$. We can define a strategy $\sigma$ for Player $\alpha$ in the BM$_R(X)$-game as follows:
If Player $\beta$ chooses $U_1 \in \mathscr{O}(X)$, then Player $\alpha$ responds $\sigma(U_1):= V_1=U_1 \cap G_1$. Next, if Player $\beta$ chooses $U_2 \in \mathscr{O}(V_1)$, then Player $\alpha$ responds $\sigma(U_1, U_2):=V_2=U_2 \cap G_2$.
Inductively,
$\sigma(U_1, \dotsc, U_n) := V_n=U_n \cap G_n\ \forall n \in \mathbb{N}$ such that
$\bigcap_{n=1}^{\infty} U_n = \bigcap_n(U_n \cap G_n) \subseteq \bigcap_nG_n \subseteq R$. Thus, $\sigma$ is a winning strategy for Player $\alpha$ in the BM$_R(X)$-game.

\textsl{Sufficiency}: Let $\sigma$ be a winning strategy for Player $\alpha$ in the BM$_R(X)$-game.
For each $n \in \mathbb{N}$, define $\mathscr{P}_n$ as a maximal family $\{(U_{n,i},V_{n,i})\}_{i \in I_n}$ satisfying:
\begin{enumerate}
    \item $\{V_{n,i}\}_{i \in I_n}$ are pairwise disjoint, and $U_{n,i}, V_{n,i} \in \mathscr{O}(X)$ with $U_{n,i}\supseteq V_{n,i}$ $\forall i\in I_{n}$;
    \item $\forall i \in I_n, \exists j \in I_{n-1}$ s.t. $V_{n-1,j}\supseteq U_{n,i}$, and $U_{0,j}=V_{0,j}=X$ $\forall j \in I_{0}$;
    \item If $(i_1, \dotsc, i_n) \in I_1 \times \cdots \times I_n$ with $V_{1,i_1} \supseteq \cdots \supseteq V_{n,i_n}$, then $V_{n,i_n} =\sigma(U_{1,i_{1}}, \dotsc, U_{n,i_{n}})$.
\end{enumerate}
Let $\Omega_n = \bigcup_{i \in I_n} V_{n,i}$. Then $\Omega_n$ is open dense in $X$ for all $n \in \mathbb{N}$.
Indeed, for $n=1$, if $\Omega_1$ were not dense, then take $G_{1} = X \setminus \overline{\Omega}_1 \in \mathscr{O}(X)$ so that $\mathscr{P}_1 \cup \{(G_{1}, \sigma(G_{1}))\}$ contradicts the maximality of $\mathscr{P}_1$.
Assume $\Omega_n$ is dense, then $\Omega_{n+1}$ is dense. Indeed,
suppose $\overline{\Omega}_{n+1}\neq X$, then $G_{n+1} := X \setminus \overline{\Omega}_{n+1}\in \mathscr{O}(X)$. Since $\Omega_n$ is dense, $G_{n+1} \cap \Omega_n \neq \emptyset$. Thus there exists some $i^* \in I_n$ with $G_{n+1} \cap V_{n,i^*} \neq \emptyset$. Let $U^* = G_{n+1} \cap V_{n,i^*}$.
For $(i_1, \dotsc, i_{n-1},i^*) \in I_1\times \dotsm \times I_{n-1} \times I_{n}$, $U_{1,i_1}\supseteq V_{1,i_1} \supseteq \dotsm \supseteq U_{n,i^*}\supseteq V_{n,i^*}$, let $V^* = \sigma(U_{1,i_1},\dotsc,U_{n-1,i_{n-1}},U_{n,i^*},U^*)$.
Then $\mathscr{P}_{n+1}^* := \mathscr{P}_{n+1} \cup \{(U^*, V^*)\}$ satisfies the above three conditions.
This contradicts the maximality of $\mathscr{P}_{n+1}$.

If $\bigcap_{n=1}^{\infty} \Omega_n = \emptyset$, then $X$ is meager; and so, $R$ is residual in $X$. Otherwise, for every point $x\in \bigcap_{n=1}^{\infty} \Omega_n = \bigcap_{n=1}^{\infty} \left( \bigcup_{i \in I_n} V_{n,i} \right)$, then there exists $i_n \in I_n$ for all $n\in\mathbb{N}$ such that $x \in \bigcap_{n=1}^{\infty} V_{n,i_n}$.
By the construction of $\mathscr{P}_n$, the sequence $\{(U_{n,i_n}, V_{n,i_n})\}_{n=1}^{\infty}$ defines a BM$_R(X)$-play. Since $\sigma$ is a winning strategy for Player $\alpha$,  $\bigcap_{n=1}^{\infty} \Omega_n \subseteq \bigcap_{n=1}^{\infty} V_{n,i_n} \subseteq R$
and $R$ is residual in $X$. The proof is complete.
\end{proof}

Recall that $\mathcal{S}_\textrm{cl}(X)$ is the collection of all non-void closed separable subspaces of $X$. Let $X_0\subset X$ be a dense set and $\mathcal{S}_\textrm{cl}(X|X_0)=\{F\in\mathcal{S}_\textrm{cl}(X)\,|\,\exists B\subseteq X_0\textit{ s.t. }B\textit{ is countable }\ \&\ \overline{B}=F\}$. Then:

\begin{A.6}
If $X$ has countable tightness and $X_0\subset X$ a dense set, then $\mathcal{S}_\textrm{cl}(X|X_0)$ is a rich family for $X$.
\end{A.6}

\begin{proof}
By Lemma~\ref{4.1.3}.
\end{proof}

\begin{A.7}
Let $O\subseteq X\times Y$ be an open dense set. Then for all $U\in\mathscr{O}(X)$ and $W_1,\dotsc, W_m\in\mathscr{O}(Y)$, there exist $V\in\mathscr{O}(U)$ and $y_1\in W_1, \dotsc, y_m\in W_m$ such that $V\times\{y_1,\dotsc,y_m\}\subseteq O$.
\end{A.7}

\begin{proof}
    Obvious.
\end{proof}

\begin{A.8}[{Topological Fubini Theorem~II; cf.~\cite[Thm.~4.3]{LM08} for $Y$ a $W$-space of G}]
Let $X$ be a space and $Y$ an almost $W$-space of G with countable tightness. Let $\mathcal{F}$ be any rich family for $Y$. If $\mathcal{G}=\{G_n\}_{n=1}^\infty$ is a sequence of dense open subsets of $X\times Y$, then
\begin{enumerate}
\item[] $X_\mathcal{G}=\{x\in X\,|\,\exists F\in\mathcal{F}\textrm{ s.t. }F\cap G_{n,x}\textrm{ is dense open in }F\ \forall n\in\mathbb{N}\}$
\end{enumerate}
is residual in $X$.
(In particular, if $R$ is a residual set in $X\times Y$, then
\begin{enumerate}
\item[] $X_R=\{x\in X\,|\,\exists F\in\mathcal{F}\textrm{ s.t. }F\cap R_x\textrm{ is residual in }F\}$
\end{enumerate}
is residual in $X$.)
\end{A.8}

\begin{proof}
Let $Y_0$ be the dense set of $W$-points of G in $Y$. Then by Lemma~A.6,
$\mathcal{S}_\textrm{cl}(Y|Y_0)$ is a rich family for $Y$. Let $\mathcal{F}_0=\mathcal{S}_\textrm{cl}(Y|Y_0)\cap\mathcal{F}$.
Without loss of generality, assume $\mathcal{G}$ is a decreasing sequence. If $Y$ is finite (not necessarily discrete in our non-$T_1$ setting), then $X_\mathcal{G}$ is residual in $X$ by Lemma~A.1. So, in what follows, suppose $Y$ is infinite.

For any $a\in Y_0$, let $t_a$ be a winning strategy for Player $\alpha$ in the $\mathscr{G}_a(Y)$-game (cf. Def. \ref{5.1.1}). We shall inductively define a winning strategy $\sigma$ for Player $\alpha$ in the BM$_{X_G}(X)$-game. For that, first let $Z_0=\emptyset$ and $\mathscr{F}_0=\{y_{0,j}\in Y_0\,|\,j\in\mathbb{N}\}$ any countable set such that $\overline{\mathscr{F}}_0\in\mathcal{F}_0$.

\item \textsl{Base Step}: For all $B_1\in\mathscr{O}(X)$, by using Lemma~A.7 we can define a set $\mathscr{F}_1=\{y_{1,j}\in Y_0\,|\,j\in\mathbb{N}\}$ so that $Z_0\cup\mathscr{F}_0\subseteq\overline{\mathscr{F}}_1\in\mathcal{F}_0$, and define
$\sigma(B_1)\in\mathscr{O}(B_1)$ and $z_{1,1,1}\in t_{y_{1,1}}(y_{1,1})$ so that $\sigma(B_1)\times\{z_{1,1,1}\}\subseteq G_1$.
Define $Z_1=\{z_{1,1,1}\}=\{z_{i,j,l}\,|\,i,j,l\in\mathbb{N}\textit{ s.t. }i+j+l\le 1+2\}$.

\item \textsl{Inductive Hypothesis}: Suppose $(B_1,\dotsc,B_k)$ is a partial $\sigma$-string in $\mathscr{O}(X)$, and for each $1\le n\le k$ the following terms have been defined:
    $$
    \mathscr{F}_n=\{y_{n,j}\in Y_0\,|\,j\in\mathbb{N}\},\quad Z_n=\{z_{i,j,l}\,|\,i,j,l\in\mathbb{N}\textit{ s.t. }i+j+l\le n+2\},\quad \sigma(B_1,\dotsc,B_n)\in\mathscr{O}(B_n)
    $$
    such that
\begin{enumerate}
\item[(a)] $Z_{n-1}\cup\mathscr{F}_{n-1}\subseteq\overline{\mathscr{F}}_n\in\mathcal{F}_0$;
\item[(b)] $z_{i,j,l}\in t_{y_{i,j}}(y_{i,j},z_{i,j,1},\dotsc,z_{i,j,l-1})$ for all $i,j,l\in\mathbb{N}$ with $i+j+l=n+2$; and
\item[(c)] $\sigma(B_1,\dotsc,B_n)\times\{z_{i,j,l}\colon i+j+l=n+2\}\subseteq G_n$.
\end{enumerate}

\item \textsl{Inductive Step}: Suppose $(B_1,\dotsc,B_{k+1})$ is a partial $\sigma$-string, i.e., $B_{k+1}\in\mathscr{O}(\sigma(B_1,\dotsc,B_k))$. Then:
\begin{enumerate}
\item[(i)] Define $\mathscr{F}_{k+1}=\{y_{k+1,j}\in Y_0\,|\,j\in\mathbb{N}\}$ such that $Z_k\cup\mathscr{F}_k\subseteq\overline{\mathscr{F}}_{k+1}\in\mathcal{F}_0$;
\item[(ii)] By the inductive hypothesis, $(y_{i,j},z_{i,j,1},\dotsc,z_{i,j,l})$ is a partial $t_{y_{i,j}}$-string for all $i,j,l\in\mathbb{N}$ with $i+j+l=k+2$.

Next, define $\sigma(B_1,\dotsc,B_{k+1})\in\mathscr{O}(B_{k+1})$ and $Z_{k+1}=\{z_{i,j,l}\,|\,i,j,l\in\mathbb{N}\textit{ s.t. }i+j+l\le(k+1)+2\}$ so that:
\begin{enumerate}
\item[(a)] $z_{i,j,l}\in t_{y_{i,j}}(y_{i,j},z_{i,j,1},\dotsc,z_{i,j,l-1})$ for all $i,j,l\in\mathbb{N}$ with $i+j+l=(k+1)+2$;
\item[(b)] $\sigma(B_1,\dotsc,B_{k+1})\times\{z_{i,j,l}\colon i+j+l=(k+1)+2\}\subseteq G_{k+1}$.
\end{enumerate}
\end{enumerate}
This completes the inductive definition of $\sigma$.

Finally, we will consider any $\sigma$-sequence $\{B_n\}_{n=1}^\infty$ of the BM$_{X_G}(X)$-game. For that for every point $x\in\bigcap_{n=1}^\infty B_n$ (if exists), let $F=\overline{\bigcup_{n=1}^\infty\mathscr{F}_n}\in\mathcal{F}_0$. Given $y_{i,j}\in \mathscr{F}_i\,(\subseteq\mathcal{F})$ and $N\in\mathbb{N}$, we have that $F\ni z_{i,j,l}\to y_{i,j}$ as $l\to\infty$ for $t_{y_{i,j}}$ is a winning strategy for Player $\alpha$ in the $\mathscr{G}_{y_{i,j}}(Y)$-game; and moreover, $\{x\}\times\{z_{i,j,l}\colon i+j+l=n+2\}\subseteq G_n\subseteq G_N$, i.e., $\{z_{i,j,l}\colon i+j+l=n+2\}\subseteq G_{N,x}$, as $n\ge N$. Thus, $F\cap G_{N,x}$ is dense in $F$ for all $N\in\mathbb{N}$ so that $x\in X_\mathcal{G}$. Then $\bigcap_{n=1}^\infty B_n\subseteq X_\mathcal{G}$ is residual in $X$ by Lemma~A.5. The proof is complete.
\end{proof}

In applications of Lemma~A.8, $\mathcal{F}$ is often a rich family of non-meager subspaces for $Y$. However, even a metric space need not have such a rich family.

Finally Lemma~A.1 may be compared with Lemma~A.8. The two lemmas overlap, but neither includes the other.
See \cite[Prop.~3.1]{V70}, \cite[Lem.~5.2]{G90}, \cite[Lem.~5.3]{D23}, and \cite[Thm.~2.2.5]{DFLX} for some other variants of Fubini's theorem in the setting $p\colon W\rightarrow X$ in place of $p\colon W=X\times Y\rightarrow X$, where $p$ is only a semi-open continuous mapping but $W$ is a second countable space or has a $p$-fiber countable $\pi$-base.
\section{Proof of Lemma~\ref{7.2}}\label{B}
Recall that the so-called Schwarz function $S\colon [0,1]\times[0,1]\rightarrow[0,1]$, defined by $S(s,t)=0$ if
$(s,t)=(0,0)$ and $2st/(s^2+t^2)$ if $(s,t)\not=(0,0)$, is separately continuous, but jointly continuous at $(s,t)$ if and only if $(s,t)\not=(0,0)$.

\begin{7.2}
Let $X$ be a completely regular space and $F\subset X$ a nowhere dense set. Then there exists a compact Hausdorff space $Y$ and a separately continuous function $f\colon X\times Y\rightarrow[0,1]$ such that for each $x\in F$, there is a point $y\in Y$ such that $f$ is discontinuous at $(x,y)$.
\end{7.2}

The proof of Lemma~\ref{7.2}, due to Saint-Raymond (1983) \cite[Lem.~4]{SR83}, was written in French. So we reprove it here for our convenience.

\begin{proof}
We may assume $F$ is closed without loss of generality. Using induction, we can choose a family $\Phi=\{\varphi_i\,|\,i\in \Lambda\}$ in $\mathcal{C}(X,[0,1])$ such that: ${\varphi_i}\!\upharpoonright_F\equiv0$ for all $i\in \Lambda$, $\varphi_i\cdot\varphi_j\equiv0$ for all $i\not=j\in \Lambda$, and $\varOmega:=\bigcup_{i\in \Lambda}\{x\in X\,|\,\varphi_i(x)>0\}$ is dense open in $X$.

Consider $\Lambda$ as a discrete topological space so that $\Lambda\times[0,1]$ is a locally compact Hausdorff space. Let $Y=\Lambda\times[0,1]\cup\{\infty\}$ be the one-point compactification of $\Lambda\times[0,1]$. Define a map $f\colon X\times Y\rightarrow[0,1]$ such that
$$
f(x,y)=\begin{cases}
0 & \textrm{if }x\in X\textrm{ and }y=\infty,\\
S(\varphi_i(x),t)& \textrm{if }x\in X\textrm{ and }y=(i,t)\in I\times[0,1].
\end{cases}
$$
If $\{(i_j,t_j)\}_{j\in J}$ is a net in $\Lambda\times[0,1]$ such that $(i_j,t_j)\to\infty$ in $Y$, then for each $k\in \Lambda$, there exists $j_k\in J$ such that $i_j\not=k$ as $j\ge j_k$. Let $x\in X$. Then there exists at most one index $k(x)\in \Lambda$ such that $\varphi_{k(x)}(x)\not=0$. So $\varphi_{i_j}(x)=0$ as $j\ge j_{k(x)}$. Thus, $f(x,(i_j,t_j))=0$ as $j\ge j_{k(x)}$.
Then it is easy to verify that $f$ is separately continuous. Let $x\in F$. We can choose a net $\{x_j\}_{j\in J}$ in $\varOmega$ with $x_j\to x$. For each $j\in J$, we choose an index $i_j\in \Lambda$ such that $t_j:=\varphi_{i_j}(x_j)>0$ so that $f(x_j,(i_j,t_j))=1$. Since $Y$ is compact, we may assume (a subnet of) $(i_j,t_j)\to y=(i,t)\in Y$, and so $(x_j,(i_j,t_j))\to(x,y)$. As $f(x,y)=0$, it follows that $f$ is not continuous at $(x,y)$.
The proof is complete.
\end{proof}
\end{appendix}
\section*{Acknowledgements}
\noindent 
The authors would like to thank the referee for her/his constructive comments. This work was supported by National Natural Science Foundation of China (Grant No. 12271245).


\end{document}